\documentclass[reqno,11pt]{amsart}
\DeclareMathOperator*{\esssup}{\mathrm{ess\,sup}}

\usepackage{amssymb,amsthm}
\usepackage{amsmath}
\usepackage{amsfonts}
\usepackage{dsfont}

\usepackage[a4paper,  margin=2.5cm]{geometry}
\usepackage[active]{srcltx}
\makeatletter\@addtoreset{equation}{section}\makeatother

\newtheorem{theorem}{Theorem}[section]
\newtheorem{corollary}[theorem]{Corollary}
\newtheorem{lemma}[theorem]{Lemma}

\newtheorem{proposition}[theorem]{Proposition}

\newtheorem{definition}[theorem]{Definition}
\newtheorem{remark}[theorem]{Remark}
\numberwithin{equation}{section}

\title[A Markov process for a particle system with attraction]{A Markov process for a continuum infinite particle system with attraction}

\author{ Yuri  Kozitsky}

\address{Instytut Matematyki, Uniwersytet Marii Curie-Sk{\l}odowskiej, 20-031 Lublin, Poland}
\email{jkozi@hektor.umcs.lublin.pl}

\author{Michael R\"ockner}

\address{Fakult\"at f\"ur Mathematik, Universit\"at Bielefeld, Bielefeld, Germany}

\email{roeckner@math.uni-bielefeld.de}

\begin{document}


\begin{abstract}

An infinite system of point particles placed in $\mathds{R}^d$ is
studied. The particles are of two types; they perform random walks
in the course of which those of distinct types repel each other. The
interaction of this kind induces an effective multi-body attraction
of the same type particles, which leads to the multiplicity of
states of thermal equilibrium in such systems. The pure states of
the system are locally finite counting measures on $\mathds{R}^d$.
The set of such states $\Gamma^2$ is equipped with the vague
topology and the corresponding Borel $\sigma$-field. For a special
class $\mathcal{P}_{\rm exp}$ of  probability measures defined on
$\Gamma^2$, we prove the existence of a family $\{P_{t,\mu}: t\geq
0, \ \mu \in \mathcal{P}_{\rm exp}\}$  of probability measures
defined on the space of c{\`a}dl{\`a}g paths with values in
$\Gamma^2$, which is a unique solution of the restricted martingale
problem for the mentioned stochastic dynamics. Thereby, the
corresponding Markov process is specified.
\end{abstract}

\maketitle

\tableofcontents

\section{Introduction}

The stochastic dynamics of infinite systems of interacting particles
placed in a continuous habitat is an actual and highly demanding
subject of modern probability theory.  In its comprehensive version,
one deals with stochastic processes. Thus far,  Markov processes
have been constructed only for those particle systems where one
cannot expect phase transitions, i.e., multiplicity of states of
thermal equilibrium existing at the same values of the external
parameters. In the present work, we deal with a system, for which
such phase transitions are possible \cite{WR,GH,KK,Mazel}, that
ought to have an essential impact on its stochastic dynamics, cf.
\cite{Kissel}. Namely, the object we study is an infinite collection
of point particles of two types placed in $X=\mathds{R}^d$, $d\geq
1$. The particles perform random walks (jumps) in the course of
which those belonging to the same type (component) do not interact,
whereas different type particles repel each other. This model can be
viewed as a jump version of the Widom-Rowlinson model or the
continuum two-state Potts model, see \cite{WR,KK} and
\cite{GH,Mazel}, respectively, as well as the literature quoted in
these publications. By integrating out the coordinates of one of the
components, one obtains a single-component particle system with a
multi-body attraction (see subsect. \ref{comSS} below), responsible
for phase transitions -- the multiplicity of states of thermal
equilibrium, see \cite{Lebowitz} for more on this issue.

Similarly as in \cite{KR}, we construct the process in question  by
solving a \emph{restricted martingale problem}, cf. \cite[page
79]{Dawson}. The basic aspects of this construction can be outlined
as follows. The starting point is the evolution of states $\mu_0 \to
\mu_t$ of the considered model obtained in \cite[Theorem
3.5]{asia1}; whereas the final outcome is the family of
c{\`a}dl{\`a}g path measures that solves the mentioned martingale
problem. The one-dimensional marginals of these path measures
coincide with the corresponding $\mu_t$ constructed in \cite{asia1}.

Let us now present the main ingredients of our theory. First, we set
the state space -- the collection $\Gamma^2$ of all possibly
infinite configurations in $X$. Let $\gamma$ be an integer valued
Radon measure on $X$. For each ball $B_r(x) :=\{y\in X: |x-y|\leq
r\}$, $r>0$, one thus has $\gamma (B_r(x))\in \mathds{N}_0$. For
$x\in X$, we set $n_\gamma (x) = \inf_{r>0} \gamma (B_r(x))$, and
also $p(\gamma) =\{x \in X: n_\gamma (x)
>0\}$. Each such $\gamma$ can be associated with a locally finite
system of point `particles' such that each $x\in p(\gamma)$ is
occupied by $n_\gamma(x)$ of them, cf. \cite{Lenard}. Keeping this
in mind, we will call $\gamma$ and $p(\gamma)$ \emph{configuration}
and \emph{ground configuration}, correspondingly. The set of all
such configurations $\gamma$ is denoted by $\Gamma$. Since we are
going to consider a two-component system of particles, its state
space is $\Gamma^2 =\Gamma\times \Gamma$, consisting of the pairs
$\gamma=(\gamma_0, \gamma_1)$, $\gamma_i\in \Gamma$. In the sequel,
$\gamma$ without indices will always denote such a pair, whereas
$\gamma_i$ will stand for the configuration of particles of type
$i=0,1$. Then the ground configuration of
$\gamma=(\gamma_0,\gamma_1)$ is $p(\gamma)=p(\gamma_0) \cup
p(\gamma_1)$. We also set $n_\gamma(x) = n_{\gamma_0}(x) +
n_{\gamma_1}(x)$. If $n_\gamma(x) =1$ for each $x\in p(\gamma)$,
then $\gamma$ is called a \emph{simple} configuration. The set of
all simple configurations is then
\begin{equation}
  \label{no16}
  \breve{\Gamma}^2 = \{ \gamma \in \Gamma^2: \forall x\in p(\gamma) \
  n_\gamma(x)=1\}.
\end{equation}
As mentioned above, the configurations are assumed to be locally
finite, i.e., each $\gamma_i$ takes finite values on every compact
$\Lambda\subset X$. Let $x_1, x_2, \dots, $ be an enumeration of a
given $\gamma_i$. Then by $\sum_{x\in \gamma_i}g(x)$ we mean
$\sum_{j} g(x_j)$, where $g:X\to \mathds{R}$ is a suitable function.
Obviously, this interpretation of $\sum_{x\in \gamma_i}g(x)$ is
independent of the enumeration  used in the second sum. Note also
that $\sum_{x\in \gamma_i}g(x) = \int_X g(x) \gamma_i(dx)$, see
\cite{Lenard} for more on this subject. Then $\Gamma$ is equipped
with the vague topology, which is the weakest topology that makes
continuous the maps $\gamma_i \mapsto \sum_{x\in \gamma_i} g(x)$,
$g\in C_{\rm cs}(X)$, where the latter is the collection of all
continuous compactly supported numerical functions. Correspondingly,
the set $\Gamma^2 = \Gamma\times \Gamma$ is equipped with the
product topology, and thereby with the  Borel $\sigma$-field
$\mathcal{B}({\Gamma}^2)$. This allows us to employ probability
measures defined thereon, the set of which is denoted by
$\mathcal{P}({\Gamma^2})$. Their evolution is described by the
Fokker-Planck equation
\begin{equation}
  \label{FPE}
  \mu_{t_2}(F) = \mu_{t_1}(F) + \int_{t_1}^{t_2}\mu_s ({L}F) d s,
  \qquad t_2 >t_1 \geq 0,
\end{equation}
see \cite{FKP} for a general theory of such and similar objects. In
(\ref{FPE}), we use the notation $\mu(F) = \int F d \mu$ and ${L}$
is the Kolmogorov operator, which in the considered case has the
following form
\begin{eqnarray}
  \label{L}
  ( {L} F) ({\gamma}) & = & \sum_{x\in {\gamma}_0}\int_{X} a_0 (x-y) \exp\left( - \sum_{z\in {\gamma}_1}
  \phi_0
  (z-y)\right) \left[F({\gamma}\setminus x \cup_0 y) - F({\gamma})
  \right] d y \qquad  \\[.2cm] \nonumber & + & \sum_{x\in {\gamma}_1} \int_{X} a_1 (x-y)
  \exp\left( - \sum_{z\in {\gamma}_0} \phi_1
  (z-y)\right) \left[F({\gamma}\setminus x \cup_1 y) - F({\gamma})
  \right] d y.
\end{eqnarray}
Here and in the sequel, by writing ${\gamma}\cup_i y$ we mean the
element of $\Gamma^2$ obtained from $\gamma$ by adding $y\in X$ to
its component $\gamma_i$, $i=0,1$. Likewise, by writing $\gamma
\setminus x$ we mean the configuration obtained from $\gamma$ by
subtracting $x$ from the corresponding $\gamma_i$ if it is clear
which $i$ is meant. Otherwise, we indicate it explicitly, see  the
next section for more detail.

The first summand in (\ref{L}) describes the following elementary
act: a particle located at $x\in \gamma_0$  instantly changes its
position (jumps) to $y\in X$ with rate
\begin{equation}
  \label{rate}
c_0(x,y;\gamma) = a_0(x-y) \exp\left(- \sum_{z\in {\gamma}_1}
\phi_0(z-y)\right).
\end{equation}
It depends on $\gamma_1$ through the multiplier $\exp\left(-
\sum_{z\in {\gamma}_1} \phi_0(z-y)\right)$, $\phi_0\geq 0$, the role
of which is diminishing the free jump rate $a_0(x-y)$ if the target
point is `close' to $\gamma_1$. In view of this, we shall call $a_i$
and $\phi_i$, $i=0,1$, jump and repulsion kernels, respectively.

As is typical for infinite particle systems, among the states
$\mathcal{P}(\Gamma^2)$ one distinguishes a proper subset to which
the evolution described by (\ref{FPE}) is restricted. In
\cite{asia1}, there was introduced a subset $\mathcal{P}_{\rm exp}
\subset\mathcal{P}(\Gamma^2)$, cf. Definition \ref{no1dfg} below,
consisting of \emph{sub-Poissonian measures}, and then constructed a
map $t \mapsto \mu_t\in \mathcal{P}_{\rm exp}$ corresponding to
(\ref{FPE}) in the following sense. For a certain class of
(unbounded) functions $F:\Gamma^2\to \mathds{R}$, it was shown that:
(a) $LF$ belongs to this class; (b) each such $F$ is
$\mu$-integrable for all $\mu\in \mathcal{P}_{\rm exp}$; (c) the
mentioned map satisfies (\ref{FPE}). Our present results are
essentially based on this construction. In a sense, we `superpose'
the mentioned map $t \mapsto \mu_t\in \mathcal{P}_{\rm exp}$ and
obtain a family of c{\`a}dl{\`a}g path measures $\{P_{s,\mu}: s \geq
0, \mu \in \mathcal{P}_{\rm exp}\}$, which is the unique solution of
the \emph{restricted initial value martingale problem} corresponding
to (\ref{L}), see \cite[page 79]{Dawson}, and is such that the
one-dimensional marginal of $P_{s,\mu}$ corresponding to $t>s$ is
$\mu_t$ if $\mu_s=\mu$. This construction consists of the following
steps:
\begin{itemize}
  \item[(a)] We pick a subset $\Gamma_*^2\subset \Gamma^2$ and equip it  with a
topology that makes this set a Polish space, continuously embedded
in $\Gamma^2$, and such that $\mu(\Gamma_*^2)=1$ for all $\mu \in
\mathcal{P}_{\rm exp}$. This extends the set of continuous functions
$F:\Gamma^2 \to \mathds{R}$  and allows us to redefine the members
of $\mathcal{P}_{\rm exp}$ as measures on $\Gamma_*^2$. Then we
construct a sufficiently massive set $\mathcal{D}(L)$ of bounded
continuous functions $F:\Gamma_*^2 \to \mathds{R}$, which will serve
as the domain of the Kolmogorov operator. Its crucial property is
that $LF$ remains bounded for all $F\in \mathcal{D}(L)$.
\item[(b)] We prove that any solution $t\mapsto \mu_t \in \mathcal{P}(\Gamma^2_*)$ of the Fokker-Planck equation
(\ref{FPE}) with $F\in \mathcal{D}(L)$ and $\mu_0\in
\mathcal{P}_{\rm exp}$ is such that $\mu_t \in\mathcal{P}_{\rm exp}$
for all $t>0$. Thereby, we prove that there is only one such
solution  given by the map $t\mapsto \mu_t \in \mathcal{P}_{\rm
exp}$ constructed in \cite{asia1}.
\item[(c)] Then we introduce auxiliary models described by
$L^\sigma$, $\sigma\in [0,1]$ obtained by  replacing $a_i(x-y) \to
a_i^\sigma(x,y)$, in such a way that $L^0=L$, whereas $L^\sigma$
with $\sigma \in (0,1]$  admits constructing transition functions
$p_t^\sigma$, by means of which  we  obtain Markov processes
$\mathcal{X}^\sigma$ with values in $\Gamma_*^2$.
\item[(d)] Thereafter, we prove that the finite-dimensional
distributions of $\mathcal{X}^\sigma$ satisfy Chentsov-like
estimates, uniformly in $\sigma\in (0,1]$. By this we get that: (i)
each $\mathcal{X}^\sigma$ has a c{\`a}dl{\`a}g modification, which
corresponds to the existence of families $\{P^\sigma_{s,\mu}: s \geq
0, \mu \in \mathcal{P}_{\rm exp}\}$, $\sigma \in (0,1]$, of
c{\`a}dl{\`a}g path measures; (ii) as $\sigma \to 0$, the measures
$P^\sigma_{s,\mu}$  have accumulation points which solve the
restricted initial value martingale problem for $(L,\mathcal{D}(L),
\mathcal{P}_{\rm exp})$. Then we prove that all these accumulation
points coincide as their one-dimensional marginals solve
(\ref{FPE}), which has a unique solution, that was proved in (b).
Thereby, we obtain the unique solution of the mentioned martingale
problem $\{P_{s,\mu}: s\geq 0, \ \mu\in \mathcal{P}_{\rm exp}\}$.
\item[(e)] Finally, we prove that the constructed Markov process
with probability one takes values in $\Gamma^2_*
\cap\breve{\Gamma}^2$, see (\ref{no16}).
\end{itemize}
The structure of this paper is as follows. In Sect. 2, we  introduce
the main ingredients of our construction, among which are the spaces
of tempered configurations $\Gamma^2_*$, $\breve{\Gamma}^2_*$,  the
basic classes of bounded continuous functions $F:\Gamma^2_*\to
\mathds{R}$, and the set of \emph{sub-Poissonian measures}
$\mathcal{P}_{\rm exp} \subset\mathcal{P}(\Gamma^2)$, see Definition
\ref{no1dfg}. In Sect. 3, we formulate our assumptions concerning
the properties of the parameters $a_i$ and $\phi_i$ that appear in
(\ref{L}), (\ref{rate}). Next, we introduce the corresponding spaces
of c{\`a}dl{\`a}g paths and the very notion of a solution of the
restricted initial value martingale problem for our model, see
Definition \ref{A1df}. Then the result of this work is formulated in
Theorem \ref{1tm}, followed by a number of comments. The remaining
sections are dedicated to the proof of Theorem \ref{1tm}. In Sect.
4, we reformulate the corresponding results of \cite{asia1} in the
form adapted to the present context, as well as develop a number of
additional technicalities. The basic result of Sect. 5 is Lemma
\ref{W1lm} which states that every solution of the Fokker-Planck
equation (\ref{FPE}) for our model lies in $\mathcal{P}_{\rm exp}$
whenever $\mu_0$ is in $\mathcal{P}_{\rm exp}$. Its proof is mostly
based on combinatorial estimates obtained in \cite{KR} and those
derived here in subsect. 5.2. Then we prove that (\ref{FPE}) has a
unique solution $t\mapsto \mu_t$, constructed in fact in
\cite{asia1}, see Lemma \ref{W2lm}. By means of Lemmas \ref{W1lm}
and \ref{W2lm}  we then prove that the martingale problem can have
at most one solution. In Sect. 6, we introduce $L^\sigma$ and show
that the solution $t \mapsto \mu^\sigma_t$ of the Fokker-Planck
equation for $L^\sigma$, $\sigma\in (0,1]$, has the property
$\mu^\sigma_t \Rightarrow \mu_t$ as $\sigma\to 0$, where $\mu_t$ is
the solution corresponding to the main model and $\Rightarrow$
denotes weak convergence. In Sect. 7, we obtain the evolution of
states $t \mapsto \hat{\mu}^\sigma_t$, $\sigma\in (0,1]$, by
constructing stochastic semigroups $S^\sigma =\{S^\sigma(t)\}_{t\geq
0}$ acting in the Banach space of signed measures on $\Gamma^2_*$,
see Lemma \ref{U1lm}. This construction becomes possible due to the
modification $a_i(x-y) \to a^\sigma_i(x,y)$ and is based on a
perturbation technique developed in \cite{TV}. By construction, $t
\mapsto \hat{\mu}^\sigma_t$ solves the Fokker-Planck equation for
$L^\sigma$, which by Lemma \ref{W2lm} yields $\hat{\mu}^\sigma_t =
\mu^\sigma_t$. At the same time, by means of the semigroups
$S^\sigma$ we get the corresponding transition functions
$p^\sigma_t$, and thus  Markov processes $\mathcal{X}^\sigma$ with
values in $\Gamma^2_*$. Thereafter, in Lemma \ref{U4lm} we show that
these processes satisfy Chentsov-like estimates, uniform in
$\sigma\in (0,1]$. By means of this result, in Sect. 8 we complete
the proof of Theorem \ref{1tm}, including the property mentioned in
item (e) above.

\section*{Notations}

In view of the size of this work, for the reader convenience we
collect here essential notations and notions used throughout the
paper.

\subsection*{Sets and spaces}
\begin{itemize}
  \item The considered particle system dwells in $X=\mathds{R}^d$, $d\geq
1$. By $\Lambda$ we always denote a compact subset of $X$, its
Euclidean volume is denoted $|\Lambda|$;
$\mathds{R}_{+}=[0,+\infty)$; $\mathds{N}$ -- the set of natural
numbers, $\mathds{N}_0 = \mathds{N}\cup \{0\}$; $B_r(y) = \{x\in
\mathds{R}^d: |x-y|\leq r\}$, $B_r = B_r(0)$, $r>0$ and $y\in
\mathds{R}^d$. For a finite set $\Delta$, by $|\Delta|$ we mean its
cardinality.
\item A Polish space is a separable topological space, the topology of which
is consistent with a complete metric, see, e.g., \cite[Chapt.
8]{Cohn}. Subsets of such spaces are usually denoted by $\mathbb{A},
\mathbb{B}$, whereas $A, B$ (with indices) are reserved for denoting
linear operators. For a Polish space $E$, by $C_{\rm b}(E)$ and
$B_{\rm b}(E)$ we denote the sets of bounded continuous and bounded
measurable functions $g : E \to \mathds{R}$, respectively;
$\mathcal{B}(E)$ denotes the Borel $\sigma$-field of subsets of $E$.
By $\mathcal{P}(E)$ we denote the set of all probability measures
defined on  $(E, \mathcal{B}(E))$. For a suitable set $\Delta$, by
$\mathds{1}_\Delta$ we denote the indicator of $\Delta$.
\item $\Gamma$ stands for the set of all locally finite counting
measures on $X$, interpreted also as configurations of point
particles with possible multiple locations, see \cite{Lenard} and
(\ref{100}) below. By $\Gamma_0$ we mean the subset of $\Gamma$
consisting of finite configurations, i.e., such that
$\gamma(X)<\infty$; by $\Gamma^2$ we denote the set of
configurations of the two-component particle system which we
consider. That is, $\Gamma^2=\Gamma\times \Gamma$ consists of
$\gamma =(\gamma_0, \gamma_1)$, $\gamma_i\in \Gamma$ with $i=0,1$
always indicating the particle type. The set of simple
configurations $\breve{\Gamma}^2$ is defied in (\ref{no16}). The set
of tempered configurations $\Gamma_*$ is defined in (\ref{n021z}) by
means of $\psi(x) = (1+|x|^{d+1})^{-1}$; then $\Gamma^2_*=\Gamma_*
\times \Gamma_*$, see also (\ref{no21}). The metric properties of
$\Gamma^2_*$ are described in Lemma \ref{N2pn}. Finally,
$\breve{\Gamma}^2_*$ stands for $\Gamma^2_* \cap\breve{\Gamma}^2$,
see (\ref{no21}). The relationships between these sets (Polish
spaces) are described in Corollary \ref{N1co}.
\item By $\mathcal{P}_{\rm exp}$ we denote the set of sub-Poissonian
measures, see Definition \ref{no1dfg}, which is one of the basic
notions of this research. Their essential properties are given in
Proposition \ref{Se1pn} and (\ref{C4}). By $\mathcal{M}$ (with
indices) we denote the Banach spaces of signed measures on
$\Gamma_*^2$, see also (\ref{U}), (\ref{U1}), (\ref{U3}).
\item By $\mathfrak{D}_{[s,+\infty)} (\Gamma_*^2)$ and $\mathfrak{D}_{[s,+\infty)}
(\breve{\Gamma}_*^2)$, $s\geq 0$, we denote the spaces of
c{\`a}dl{\`a}g maps $\gamma: [s,+\infty) \to \Gamma^2_*$ and
$\gamma: [s,+\infty) \to \breve{\Gamma}^2_*$, respectively. Equipped
with Skorohod's metric they become Polish metric spaces, see
subsect. 3.2.
\end{itemize}
\subsection*{Functions, measures, operators}

\begin{itemize}
  \item By $x,y,z$ we always  denote elements (points) of $X=\mathds{R}^d$;
  for $k\in \mathds{N}$, we write $\mathbf{x}^k =(x_1, \dots ,
  x_k)\in X^k$. By small letters $f, g, u, v, \theta, \psi, \phi$ we denote
  numerical functions defined on $X$ or $X^k$. Important classes of such functions
$\varTheta_\psi$, $\varTheta^{+}_\psi$ are defined in (\ref{T1}). By
means of the function $\psi_\sigma (x) = (1+\sigma |x|^{d+1})^{-1}$
we modify the model (\ref{L}). For $v_{i,1}, \dots
  v_{i,k} \in C_{\rm b}(X)$, $i=0,1$ and $k\in \mathds{N}$, we write
  $\mathbf{v}_i(\mathbf{x}) = v_{i,1}(x_1) \cdots v_{i,k}(x_k)$, see
  (\ref{de1}) and also (\ref{no3a}). For an integrable
  $\theta:X \to \mathds{R}_{+}$, we write $\langle \theta \rangle =
  \int_X \theta(x) d x$. Numerical functions defined of $\Gamma^2$
  or $\Gamma^2_*$ are denoted by capital letters $F, H$, etc. By $F$
  with indices we usually denote functions on $\Gamma^2, \Gamma^2_*$, whereas
  $G$  (with indices) are defined on finite configurations.
  Significant examples of such functions are $F^\theta$, see
  (\ref{no14}), $\varPsi$ (\ref{no20}), $\widetilde{F}^\theta_\tau$,
  see (\ref{T4}) and Proposition \ref{T1pn}, $\widehat{F}_\tau^m$
  (\ref{de2}). The class of functions $\mathcal{D}(L)$ is
  introduced in Definition \ref{THdf}, its properties are given in
  Proposition \ref{T3pn}. By ${\sf F}$, ${\sf G}$, ${\sf K}$ we
  denote numerical functions defined on the path spaces.
\item Measures on $\Gamma^2$ and $\Gamma^2_*$ are usually denoted by
$\mu$ with indices. Their correlation measures $\chi^{(m)}_\mu$ are
defined in (\ref{no5}). By $\pi_\kappa$ we denote the Poisson
measure, see (\ref{no21}).  Probability measures on the path spaces
$\mathfrak{D}_{[s,+\infty)} (\Gamma^2_*)$ are denoted by $P$ with
indices. For suitable measure and function, we write $\mu(f) = \int
f d \mu$.
\item By $L$ and $L^\sigma$  we denote the Kolmogorov operator
(\ref{L}) and its modifications. By $L^{\dagger, \sigma}$ we denote
the operators dual to $L^\sigma$, see (\ref{U12}), (\ref{U13}).
Their domains are set in (\ref{U14}). By $\widehat{L}$ and
$L^\Delta$ with indices we denote the counterparts of $L$ acting on
functions $G$ and correlation functions, respectively, see
(\ref{W13}) and (\ref{L1}).
\end{itemize}

\section{Preliminaries}

\subsection{Configurations spaces and correlation measures}
\label{2.1ss}

By $\Gamma$ we denote the standard set of Radon counting measures on
$X$, i.e., $\gamma (\Lambda) \in \mathds{N}_0$ for each $\gamma\in
\Gamma$ and a compact $\Lambda \subset X$. Then we also define
$n_\gamma(x) = \inf_{r>0} \gamma(B_r(x))$ and $p(\gamma)=\{x\in X:
n_\gamma(x)>0\}$. Thus, $p(\gamma)$  is a locally finite subset of
$X$, see (\ref{no16}). For $x\in p(\gamma)$, by $\gamma\setminus x$
we denote the element of $\Gamma$ such that $n_{\gamma\setminus
x}(x) = n_\gamma(x)-1$ and $n_{\gamma\setminus x}(y) = n_\gamma(y)$
whenever $y\neq x$. Similarly, $\gamma\cup x$, $x\in X$, denotes the
measure such that $n_{\gamma\cup x}(x) = n_\gamma(x) + 1$ and
$n_{\gamma\cup x}(y) = n_\gamma(y)$ for $y\neq x$. For simplicity,
with a certain abuse of notations we write
\begin{equation}
  \label{100}
  \sum_{x\in \gamma} g(x) = \int_X g(x) \gamma(dx) =\sum_{x\in p(\gamma)} n_\gamma(x) g(x),
\end{equation}
where $g$ is a positive numerical function. Note that the left-hand
side of (\ref{100}) can also be interpreted as $\sum_{j} g(x_j)$ for
a certain enumeration $\mathds{N}\ni j \mapsto x_j$ of the elements
of $p(\gamma)$, in which each $x\in p(\gamma)$ is repeated
$n_\gamma(x)$ times, see \cite{Lenard} for more detail. In the same
way, we will understand sums
\begin{equation*}
  \sum_{x\in \gamma}\sum_{y\in \gamma\setminus x} g(x,y) = \int_X
  \int_X g(x,y) \gamma(dx) \gamma(dy) - \int_X g(x,x) \gamma(dx),
\end{equation*}
that can also be generalized to all $m\in \mathds{N}$
\begin{eqnarray}
  \label{102}
& & \sum_{x_1\in \gamma}\sum_{x_2\in \gamma\setminus x_1} \cdots
\sum_{x_m \in \gamma\setminus \{x_1, \dots , x_{m-1}\}} g(x_1,\dots
x_m) \\[.2cm] \nonumber & & \quad  = \sum_{G\in \mathbb{K}_m}
(-1)^{l_G} \int_{X^{n_G}} g_G (y_1 , \dots , y_{n_G}) \gamma(d y_1)
\cdots \gamma(d y_{n_G}) \\[.2cm] \nonumber & & \quad  = \sum_{G\in \mathbb{K}_m}
(-1)^{l_G} \sum_{y_1\in \gamma} \cdots \sum_{y_{n_G}\in \gamma} g_G
(y_1 , \dots , y_{n_G}),
\end{eqnarray}
where $\mathbb{K}_m$ is the collection of all graphs with vertices
$\{1,2,\dots , m\}$, $l_G$ and $n_G$ are the number of edges and the
connected components of $G$, respectively, whereas $g_G(y_1 , \dots
, y_{n_G})$ is obtained from $g(x_1 , \dots , x_{m})$ by setting the
arguments $x_{l_1}, \dots x_{l_{s_j}}$ of the latter equal $y_j$
where $l_1, \dots l_{s_j}$ are the vertices of the $j$-th connected
component of $G$.

Since the particles which we consider are of two types, their pure
states are set to be pairs $\gamma=(\gamma_0, \gamma_1)$ such that
$\gamma_i \in \Gamma$, $i=0,1$. Thus, $\Gamma^2 =\Gamma\times \Gamma
$ is the set of all pure states of the system. Correspondingly, we
set $n_\gamma (x) = n_{\gamma_0}(x) + n_{\gamma_1}(x)$ and
$p(\gamma) = p(\gamma_0) \cup p(\gamma_1)$. We will call $p(\gamma)$
the \emph{ground} configuration for $\gamma$.

For $\gamma \in \Gamma^2$ and $m=(m_0,m_1) \in \mathds{N}^2_0$, the
counting measure $Q_{{\gamma}}^{(m)}$ on $X^{m_0}\times X^{m_1}$ is
defined by its values on compact subsets $\Delta \subset X^{m_0}
\times X^{m_1}$ in the following way. For $m_0=m_1=0$, we set
$Q_{{\gamma}}^{(m)} \equiv 1$ for each ${\gamma}$. For $m_0>0$,
$m_1=0$ and $\Delta \subset X^{m_0}$, $Q_{{\gamma}}^{(m)}(\Delta)$
is equal to the number of different ordered strings $(i_1, \dots ,
i_{m_0})$ such that $\mathbf{x}:=(x_{i_1}, \dots , x_{i_{m_0}}) \in
\Delta$. Likewise one defined $Q_{\gamma}^{(m)}$ for $m_0=0$ and
$m_1>0$. For $m\in \mathds{N}^2$, $Q_{\gamma}^{(m)}(\Delta)$ is
equal to the number of different ordered strings $(i_1, \dots ,
i_{m_0})$ and $(j_1 , \dots , j_{m_1})$ such that $(\mathbf{x},
\mathbf{y}) \in \Delta$, where $\mathbf{x}= (x_{i_1}, \dots ,
x_{m_0})$, $x_l \in \gamma_0$, and $\mathbf{y}= (y_{j_1}, \dots ,
y_{m_1})$, $y_l \in \gamma_1$. It is obvious that this definition is
independent of the enumerations of both $\gamma_i$. Then we get, cf.
(\ref{102}),
\begin{eqnarray}
  \label{no3}
 & & Q^{(m)}_{\gamma} (\Delta)\\[.2cm] \nonumber & & \ = \sum_{x_1\in \gamma_0}
  \sum_{x_2 \in \gamma_0 \setminus x_1} \cdots
  \sum_{x_{m_0}\in \gamma_0 \setminus \{x_1 , \dots ,
  x_{m_0-1}\}}\sum_{y_1\in \gamma_1}
  \sum_{y_2 \in \gamma_1 \setminus y_1} \cdots
  \sum_{y_{m_1}\in \gamma_1 \setminus \{y_1 , \dots ,
  y_{m_1-1}\}} \mathds{1}_\Delta (\mathbf{x},\mathbf{y}).
\end{eqnarray}
To simplify notations, for suitable $\varphi_0$, $\varphi$, $k\in
\mathds{N}$ and $m\in \mathds{N}_0^2$, we write
\begin{gather}
  \label{no3a}
  \sum_{\mathbf{x}^k \in \gamma_0} \varphi_0({\bf x}^k) = \sum_{x_1\in \gamma_0}
  \sum_{x_2 \in \gamma_0 \setminus x_1} \cdots
  \sum_{x_{k}\in \gamma_0 \setminus \{x_1 , \dots ,
  x_{k-1}\}} \varphi_0 (x_1 , \dots , x_k), \\[.2cm] \nonumber
  \sum_{(\mathbf{x}^{m_0}, \mathbf{y}^{m_1}) \in \gamma}
  \varphi (\mathbf{x}^{m_0}, \mathbf{y}^{m_1}) =  \sum_{\mathbf{x}^{m_0} \in
  \gamma_0}  \sum_{\mathbf{y}^{m_1} \in \gamma_1}
  \varphi (x_1 , \dots , x_{m_0}, y_1 , \dots , y_{m_1}).
\end{gather}
As above, we will write $\mathbf{x}$ instead of $\mathbf{x}^k$ if
the dimension $k$ is clear from the context. For a compact $\Lambda
\subset X$ and $\gamma_i \in \Gamma$, $i=0,1$, we let $N_\Lambda
(\gamma_i)$ be the number of the elements of $\gamma_i$ contained in
$\Lambda$. Then
\begin{equation}
  \label{W}
N_\Lambda (\gamma_i) = \sum_{x\in \gamma_i} \mathds{1}_\Lambda (x)
=\gamma_i(\Lambda),
\end{equation}
that is, $N_\Lambda (\gamma_i)= Q^{(m)}_\gamma (\Lambda)$ for the
corresponding $m$, see (\ref{no3}). For $p\in \mathds{N}$,  we have,
cf. \cite[page 8]{KR},
\begin{eqnarray}
  \label{W1}
  N^p_\Lambda (\gamma_i) = \left[\sum_{x\in \gamma_i} \mathds{1}_\Lambda (x)
  \right]^p = \sum_{l=1}^p S(p,l)\sum_{\mathbf{x}^l \in \gamma_i}
  \mathds{1}_\Lambda (x_1) \cdots \mathds{1}_\Lambda (x_l),
\end{eqnarray}
where $S(p,l)$ is Stirling's number of second kind $=$ the number of
ways to divide $p$ labeled items into $l$ unlabeled groups. Below,
expressions like that on the right-hand side of (\ref{W1}) with
$p=0$ are set to be identically equal to one.

It can be shown, cf. \cite[Theorem 1]{Lenard}, that the map
$\gamma\mapsto Q_{\gamma}^{(m)} (\Delta)$ is measurable for all
compact $\Delta$ and $m\in \mathds{N}_0^2$. However, it may be
unbounded. Let $\mathcal{P}({\Gamma^2})$ be the set of all
probability measures defined on the Polish space ${\Gamma^2}$. For a
given $\mu\in \mathcal{P}({\Gamma^2})$ and $m\in \mathds{N}_0^2$,
set
\begin{equation}
  \label{no5}
  \chi_\mu^{(m)} = \int_{{\Gamma^2}} Q^{(m)}_{\gamma}
  \mu (d \gamma),
\end{equation}
which exists for at least $m=(0,0)$. If it does for a given positive
$m$, we call it \emph{correlation measure} corresponding to these
$\mu$ and $m$. If $\chi^{(m)}_\mu(\Delta)< \infty$ for all $m\in
\mathds{N}_0^2$ and compact $\Delta$, we say that $\mu$ has
\emph{finite correlations}. In this case, each $\chi^{(m)}_\mu$ is a
Radon measure on $X^{m_0}\times X^{m_1}$.

\subsection{Sub-Poissonian measures}

We begin by recalling that  $C_{\rm cs}(X)$ is dense in $L^1(X) :=
L^1 (X, dx)$, see e.g., \cite[Theorem 4.12, page 97]{Brezis}.

 For $k\in \mathds{N}$ and $\theta \in C_{\rm cs}(X)$, by
$\theta^{\otimes k}$ we denote the function such that
$\theta^{\otimes k}(x_1 , \dots, x_k) = \theta(x_1) \cdots
\theta(x_k)$, which we extend to $k=0$ by setting $\theta^{\otimes
0}\equiv 1$. Likewise, for $\theta_0,\theta_1\in C_{\rm cs}(X)$ and
$m\in \mathds{N}_0^2$, we set
\begin{equation}
  \label{no6}
  \theta^{\otimes m}(\mathbf{x}, \mathbf{y}) = \theta_0(x_1) \cdots \theta_0 (x_{m_0})
\theta_1(y_1) \cdots \theta_1 (y_{m_1}).
\end{equation}
\begin{definition}
  \label{no1dfg}
The set of sub-Poissonian measures $\mathcal{P}_{\rm exp}$ consists
of all those $\mu\in\mathcal{P}({\Gamma^2})$ that have finite
correlations such that, for each $m\in \mathds{N}_0^2$ and
$\theta_0, \theta_1\in C_{\rm cs}(X)$, the following holds
\begin{equation}
  \label{no7}
\left|\chi_\mu^{(m)} (\theta^{\otimes m})\right| \leq
\varkappa^{|m|} \|\theta_0\|^{m_0}_{L^1(X)}
\|\theta_1\|^{m_1}_{L^1(X)}, \qquad |m| := m_0 + m_1,
\end{equation}
with some $\mu$-dependent $\varkappa>0$.
\end{definition}
The aforementioned density and the  estimate in (\ref{no7})  imply
that the map $(\theta_0,\theta_1) \mapsto \chi_\mu^{(m)}
(\theta^{\otimes m})$ can be extended to a continuous homogeneous
polynomial on $L^1(X) \times L^1(X)$. In this case, there exists a
unique positive and symmetric $k^{(m)}_\mu \in L^\infty
(X^{m_0}\times X^{m_1})$ such that, see (\ref{no6}),
\begin{eqnarray}
  \label{no8}
   \chi_\mu^{(m)} (\theta^{\otimes m}) &= & \int_{X^{m_0} \times X^{m_1}}
  k^{(m)}_\mu (x_1 , \dots , x_{m_0}; y_1, \dots , y_{m_1})\\[.2cm]  \nonumber & \times &  \theta_0(x_1) \cdots \theta_0 (x_{m_0})
\theta_1(y_1) \cdots \theta_1 (y_{m_1}) d x_1 \cdots d x_{m_0} d y_1
\cdots dy_{m_1} \\[.2cm] \nonumber & =: & \int_{X^{m_0} \times X^{m_1}}
  k^{(m)}_\mu (\mathbf{x}, \mathbf{y}) \theta^{\otimes m}(\mathbf{x}, \mathbf{y}) d^{m_0} \mathbf {x} d^{m_1}
  \mathbf
  {y} =:
\langle\! \langle k^{(m)}_\mu, \theta^{\otimes m} \rangle \!
\rangle.
\end{eqnarray}
The mentioned symmetricity means that
\begin{equation}
  \label{nosym}
   k^{(m)}_\mu (x_1 , \dots , x_{m_0}; y_1, \dots , y_{m_1}) =  k^{(m)}_\mu (x_{\sigma_0(1)} , \dots , x_{\sigma_0(m_0)}; y_{\sigma_1(1)},
 \dots , y_{\sigma_1(m_1)}),
\end{equation}
holding for all corresponding permutations $\sigma_0, \sigma_1$,
whereas the positivity and the bound in (\ref{no7}) yield
\begin{equation}
  \label{no9}
  0 \leq k^{(m)}_\mu (\mathbf{x},\mathbf{y})
  \leq \varkappa^{|m|},
\end{equation}
holding for Lebesgue-almost all $(\mathbf{x},\mathbf{y}) \in
X^{m_0}\times X^{m_1}$. The upper estimate in (\ref{no9}) is known
as Ruelle's bound \cite{Ruelle}. Noteworthy, for each $\mu$,
\begin{equation}
  \label{JULY}
  k^{(0,0)}_\mu = 1,
\end{equation}
which readily follows by the very definition of the counting measure
$Q_\gamma$ and (\ref{no5}).

For $m=(m_0,m_1)\in \mathds{N}_0^2$, $\theta =(\theta_0, \theta_1)$,
$\theta_i \in C_{\rm cs}(X)$, $i=0,1$, we set, cf. (\ref{no3}),
\begin{eqnarray}
  \label{no10}
H^{(m)}_\theta (\gamma) & = & H^{(m_0)}_{\theta_0}
(\gamma_0) H^{(m_1)}_{\theta_1} (\gamma_1), \\[.2cm]
\nonumber
  H^{(m_i)}_{\theta_i} (\gamma_i) & = &
  \sum_{x_1 \in \gamma_i} \sum_{x_2\in \gamma_i
  \setminus x_1} \cdots \sum_{x_{m_i}\in \gamma_i \setminus
  \{x_1 , \dots , x_{m_i-1}\}} \theta_i (x_1) \cdots \theta_i
  (x_{m_i}), \\[.2cm]
\nonumber & = & \sum_{\mathbf{x}\in \gamma_i} \theta_i^{\otimes m_i}
(\mathbf{x}), \qquad m_i\geq 1,
  \qquad i=0,1,
\end{eqnarray}
and $ H^{(0)}_{\theta_i} (\gamma_i) \equiv 1$. Then by means of
(\ref{no5}) we rewrite (\ref{no8}) as follows
\begin{equation*}
  \chi^{(m)}_\mu (\theta^{\otimes m}) = \mu(H^{(m)}_\theta).
\end{equation*}
Now for $n=(n_0,n_1)\in \mathds{N}_0^2$, let us consider
\begin{equation*}
  \bar{H}_\theta^{(n)} (\gamma) = \sum_{m_0=0}^{n_0}
  \sum_{m_1=0}^{n_1} \frac{1}{m_0 ! m_1 !} H^{(m)}_\theta
  (\gamma),
\end{equation*}
which is obviously finite for all $\gamma\in{\Gamma}$. For every
$\mu\in \mathcal{P}_{\rm exp}$, by (\ref{no7}) we have that
\begin{equation}
  \label{no13}
\mu(\bar{H}_\theta^{(n)}) \leq \exp\left[\varkappa (\|\theta_0
\|_{L^1(X)} + \|\theta_1 \|_{L^1(X)} )\right],
\end{equation}
where $\varkappa$ is as in (\ref{no7}). By (\ref{no13}), for $\theta
=(\theta_0,\theta_1)\in L^1(X)\times L^1(X)$, the sequence
$\{\bar{H}_\theta^{(n)} (\gamma)\}_{n\in\mathds{N}_0^2}$ is
$\mu$-almost everywhere convergent to
\begin{eqnarray}
  \label{no14}
  F^\theta (\gamma) & = & F^{\theta_0} (\gamma_0)F^{\theta_1}
  (\gamma_1),\\[.2cm] F^{\theta_i} (\gamma_i) & := &
  \prod_{x\in \gamma_i} ( 1 + \theta_i(x)) = \exp\left(
  \sum_{x\in \gamma_i} \log( 1 + \theta_i(x))\right).
  \nonumber
\end{eqnarray}
Moreover, by (\ref{no13}) it follows that each $F^\theta$, $\theta
\in L^1(X)\times L^1(X)$ is $\mu$-integrable and
\begin{equation}
  \label{no14a}
  \mu(F^\theta) \leq \exp\left[\varkappa (\|\theta_0
\|_{L^1(X)} + \|\theta_1 \|_{L^1(X)} )\right].
\end{equation}
This means that the map $L^1(X)\times  L^1(X)\ni \theta \mapsto
\mu(F^\theta)\in \mathds{R}$ is an exponential type real entire
function, which is reflected in the notation $\mathcal{P}_{\rm
exp}$. Then borrowing terminology from the theory of entire
functions, we will call the \emph{type of $\mu$} the least
$\varkappa$ that verifies (\ref{no7}), (\ref{no9}). For the
homogeneous Poisson measure $\pi_\kappa$, $\kappa=(\kappa_0,
\kappa_1)$, $\kappa_0, \kappa_1
>0$, we have
\begin{equation}
  \label{no20A}
  k^{(m)}_{\pi_\kappa} (\mathbf{x}, \mathbf{y}) =
  \kappa_0^{m_0}\kappa_1^{m_1}, \qquad (\mathbf{x}, \mathbf{y}) \in
  X^{m_0}\times X^{m_1},
\end{equation}
which yields, see (\ref{no8}) and (\ref{no14a}),
\begin{equation}
\label{no21A}
  \pi_\kappa (F^\theta) = \exp\left(\kappa_0 \int_X \theta_0 (x) d x+ \kappa_1 \int_X \theta_1 (x) d x
  \right).
\end{equation}
Hence, the type of $\pi_\kappa\in \mathcal{P}_{\rm exp}$ is
$\varkappa = \max\{\kappa_0;\kappa_1\}$. In general, a Poisson
measure, $\pi_\chi$, is completely characterized by the pair
$\chi=(\chi_0, \chi_1)$ of its intensity measures in such a way
that, see (\ref{no8}),
\begin{eqnarray*}
\chi^{(m)}_{\pi_\chi} (\theta^{\otimes m}) = \left(\int_X
\theta_0(x) \chi_0(d x) \right)^{m_0}  \left(\int_X \theta_1(x)
\chi_1(d x) \right)^{m_1}.
\end{eqnarray*}
Note that $\pi_\chi$ is sub-Poissonian in the sense of Definition
\ref{no1dfg} if and only if $\chi_i (d x) = \varrho_i(x) d x$ with
$\varrho_i \in L^\infty (X)$, $i=0,1$.

For a symmetric $G^{(m)}\in C_{\rm cs}(X^{m_0} \times X^{m_1})$, see
(\ref{nosym}), by (\ref{no3}) we have, cf. (\ref{no3a}),
\begin{gather}
  \label{no15a}
Q^{(m)}_{\gamma} (G^{(m)}) = \sum_{(\mathbf{x},\mathbf{y})\in
\gamma} G^{(m)} (\mathbf{x}, \mathbf{y})  =: m_0! m_1! (K
G^{(m)})(\gamma) ,
\end{gather}
which by  (\ref{no5}) yields
\begin{eqnarray}
  \label{no15}
\chi_\mu^{(m)} (G^{(m)}) & = & \int_{X^{m_0} \times X^{m_1}}
  k^{(m)}_\mu (\mathbf{x}, \mathbf{y}) G^{(m)}(\mathbf{x}, \mathbf{y}) d^{m_0} \mathbf {x} d^{m_1}
  \mathbf{y} \\[.2cm] \nonumber & =& m_0! m_1! \mu(KG^{(m)})  =:  m_0! m_1!
\langle\! \langle k^{(m)}_\mu, G^{(m)} \rangle \! \rangle.
\end{eqnarray}
In view of (\ref{no9}), this can be continued to all $G^{(m)}\in
L^1(X^{m_0} \times X^{m_1})$. For positive $G^{(m)}$, by (\ref{no9})
one also gets
\begin{equation}
  \label{15b}
\mu (K G^{(m)})  \leq \pi_\kappa(K G^{(m)}) , \qquad \kappa_0 =
\kappa_1 = \varkappa,
\end{equation}
which, in particular, justifies the name \emph{sub-Poissonian}. Let
us now consider the following important version of (\ref{15b}). For
a compact $\Lambda\subset X$, we let $N_\Lambda (\gamma) = N_\Lambda
(\gamma_0)+ N_\Lambda (\gamma_1)$, see (\ref{W}). Then for $n\in
\mathds{N}$, by (\ref{W1}) we have
\begin{gather*}
  N_\Lambda^n (\gamma) = \sum_{p=0}^n {n \choose p} \sum_{l_0=1}^p
  \sum_{l_1=1}^{n-p} S(p, l_0) S(n-p, l_1) \sum_{(\mathbf{x}^{l_0},
  \mathbf{y}^{l_1}) \in \gamma} \mathds{1}_\Lambda (\mathbf{x}^{l_0},
  \mathbf{y}^{l_1}),
\end{gather*}
which for $\mu \in \mathcal{P}_{\rm exp}$ yields
\begin{gather}
  \label{W2}
  \mu(N_\Lambda^n) = \sum_{p=0}^n {n \choose p} \sum_{l_0=1}^p
  \sum_{l_1=1}^{n-p} S(p, l_0) S(n-p, l_1) \\[.2cm] \nonumber \times \int_{X^{l_0}\times
  X^{l_1}} k_\mu^{(l_0,l_1)} (\mathbf{x}^{l_0},
  \mathbf{y}^{l_1}) \mathds{1}_\Lambda (\mathbf{x}^{l_0},
  \mathbf{y}^{l_1}) d \mathbf{x}^{l_0} d \mathbf{y}^{l_1} \nonumber \\[.2cm]\leq  \sum_{p=0}^n {n \choose p} \sum_{l_0=1}^p
  \sum_{l_1=1}^{n-p} S(p, l_0) S(n-p, l_1) (\varkappa
  |\Lambda|)^{l_0 + l_1} \nonumber \\[.2cm] \nonumber = \sum_{p=0}^n {n \choose
  p}T_p (\varkappa|\Lambda|) T_{n-p} (\varkappa|\Lambda|) =
  T_n(2\varkappa|\Lambda|),
  \nonumber
\end{gather}
where $|\Lambda|$ is the Euclidean volume (Lebesgue measure) of
$\Lambda$ and $T_n$, $n\in \mathds{N}$, are Touchard's polynomials,
attributed also to J. A. Grunert, S. Ramanujan, and others, see
\cite[page 6]{Boy}. Along with the already mentioned ones,
sub-Poissonian measures have the following significant property.
Recall that the set of simple configurations $\breve{\Gamma}^2$ is
defined in (\ref{no16}).
\begin{proposition}
 \label{Se1pn}
For each $\mu \in \mathcal{P}_{\rm exp}$, it follows that
$\mu(\breve{\Gamma}^2)=1$.
\end{proposition}
\begin{proof}
For a compact $\Lambda \subset X$, $N\in \mathds{N}$ and $\epsilon
\in (0,1)$, we set
\begin{gather}
  \label{no17}
h_{\Lambda,N}(x,y) =  \mathds{1}_\Lambda (x) \mathds{1}_\Lambda
(y) \min\{N; |x-y|^{-d\epsilon}\}, \quad x,y \in X, \\[.2cm] \nonumber
H_{\Lambda,N} (\gamma) = \sum_{x\in \gamma_0} \sum_{y\in
\gamma_0\setminus x} h_{\Lambda,N}(x,y) + \sum_{x\in \gamma_1}
\sum_{y\in \gamma_1\setminus x} h_{\Lambda,N}(x,y) + \sum_{x\in
\gamma_0} \sum_{y\in \gamma_1} h_{\Lambda,N}(x,y).
\end{gather}
According to (\ref{no15}) and (\ref{no9}) we have
\begin{gather}
  \label{no18}
  \mu(H_{\Lambda,N} ) = \int_{\Lambda^2}\left(k_\mu^{(2,0)}(x,y) + k^{(1,1)} (x,y) + k_\mu^{(0,2)}(x,y)
  \right)h_{\Lambda,N}(x,y)d x dy \leq 3 \varkappa^2
  \mathcal{I}_{\Lambda, N},\\[.2cm] \nonumber
 \mathcal{I}_{\Lambda, N} := \int_{\Lambda^2} h_{\Lambda,N}(x,y)d x
 dy =: \mathcal{I}_{\Lambda, N}^{(1)}(r) + \mathcal{I}_{\Lambda,
 N}^{(2)}(r),
\end{gather}
where, for a certain $r>0$, we set and then get
\begin{gather}
  \label{no19}
\mathcal{I}_{\Lambda, N}^{(1)}(r) = \int_{\Lambda}
\left(\int_{\Lambda \cap B_r(x)} h_{\Lambda,N}(x,y)d y \right)
 dx \leq \int_{\Lambda}
\left(\int_{ B_r} \frac{d z}{|z|^{d\epsilon}} \right) d x = \frac{c_dr^{d(1-\epsilon)}}{d(1-\epsilon)}  |\Lambda|, \\[.2cm] \nonumber
\mathcal{I}_{\Lambda, N}^{(2)}(r) = \int_{\Lambda}
\left(\int_{\Lambda \cap B^c_r(x)} h_{\Lambda,N}(x,y)d y  \right)
 dx \leq \frac{1}{r^{d\epsilon}} |\Lambda|^2.
\end{gather}
Here $B_r^c(x) = X \setminus B_r(x)$, and $|\Lambda|$ and $c_d/d$
denote the Euclidean volume of $\Lambda$ and the unit ball in $X$,
respectively. We apply these estimates in (\ref{no18}) and obtain
that
\[\mu(H_{\Lambda,N} )\leq C_{\mu, \Lambda},\]
for a suitable $C_{\mu,\Lambda}$ that is independent of $N$.
Clearly, $0\leq \mu(H_{\Lambda,N}) \leq \mu(H_{\Lambda,N+1})$, which
by the monotone convergence theorem yields that the pointwise limit
\begin{equation}
  \label{S4a}
\lim_{N\to +\infty} H_{\Lambda,N} (\gamma) =: H_{\Lambda} (\gamma)
\end{equation}
is finite for $\mu$-almost all $\gamma$, i.e., for all $\gamma$
belonging to some $\breve{\Gamma}^2_{\mu,\Lambda}$ such that
$\mu(\breve{\Gamma}^2_{\mu,\Lambda}) =1$. For $C>0$, we set
$\breve{\Gamma}_C^2  =\{\gamma:H_{\Lambda} (\gamma) \leq C\}$. Then
$|x-y|\geq C^{-1/d\epsilon}$ for each $x,y\in (\gamma_0\cap\Lambda)
\cup (\gamma_1\cap\Lambda)$  and each $\gamma\in
\breve{\Gamma}_C^2$. That is, $\gamma_\Lambda:= \gamma\cap \Lambda
=(\gamma_0\cap\Lambda, \gamma_1\cap\Lambda)$ is simple whenever
$\gamma \in \breve{\Gamma}_C^2$. This yields that $\cup_{k\in
\mathds{N}_0} \breve{\Gamma}_{C+k}^2 \subset
\breve{\Gamma}^2_{s,\Lambda}$, , where the latter is the set of all
those $\gamma\in {\Gamma^2}$ for which $\gamma_\Lambda$ is simple.
At the same time, $\cup_{k\in \mathds{N}_0} \breve{\Gamma}_{C+k}^2
\supset \breve{\Gamma}_{\mu,\Lambda}^2$; hence,
$\mu(\breve{\Gamma}^2_{s,\Lambda}) =1$. Note that
$\breve{\Gamma}^2_{s,\Lambda}$ is an open subset of ${\Gamma^2}$,
cf. the proof of Lemma \ref{N2pn} below.  Now we take an ascending
sequence $\{\Lambda_k\}$ that exhausts $X$, and obtain
\begin{equation}
  \label{GLS}
\breve{\Gamma}^2 = \bigcap_{k}\breve{\Gamma}^2_{s,\Lambda_k},
\end{equation}
which completes the proof.
\end{proof}

\subsection{Tempered configurations}
\label{SS.2.3}

 Since we are going to essentially exploit
the sub-Poissonian measures, it might be reasonable to restrict our
consideration to subsets of ${\Gamma^2}$ the complements of which
are null-sets for each $\mu\in \mathcal{P}_{\rm exp}$. To this end,
we introduce the following function of $x \in X:= \mathds{R}^d$
\begin{equation}
  \label{P}
  \psi (x) = \frac{1}{1+|x|^{d+1}},   \qquad \langle \psi \rangle := \int_X \psi(x) d x,
\end{equation}
and also
\begin{equation}
  \label{no20}
  \varPsi(\gamma) = \Psi(\gamma_0) +
  \Psi(\gamma_1) = \sum_{x\in \gamma_0}\psi(x) +
  \sum_{y\in \gamma_1} \psi(y).
\end{equation}
Then we define
\begin{equation}
  \label{no21}
  {\Gamma}^{(n)}_{*} = \{\gamma\in {\Gamma^2}:
  \varPsi(\gamma) \leq n\}, \quad \Gamma^2_* =
  \bigcup_{n\in \mathds{N}}{\Gamma}^{(n)}_{*}, \quad \breve{\Gamma}^2_* =
  \Gamma_*^2 \cap \breve{\Gamma}^2.
\end{equation}
Elements of ${\Gamma}^2_*$ (resp. $\breve{\Gamma}^2_*$) are called
\emph{tempered configurations} (resp. \emph{tempered simple
configurations}). Clearly, $\Gamma^2_*\in \mathcal{B}(\Gamma^2)$ as
${\Gamma}^{(n)}_{*}\in \mathcal{B}(\Gamma^2)$ for all
$n\in\mathds{N}$. According to (\ref{no21}), we can also write
\begin{equation}
  \label{n021z}
  {\Gamma}_*^2 ={\Gamma}_{*} \times
  {\Gamma}_{*}, \qquad {\Gamma}_{*} := \{
  \gamma_i \in {\Gamma}: \Psi(\gamma_i) <
  \infty\}, \ \ i=0,1.
\end{equation}
By (\ref{no15}), for $\mu \in \mathcal{P}_{\rm exp}$, we then have
\begin{equation}
  \label{no22}
  \mu(\varPsi) = \int_X \left( k_\mu^{(1,0)}(x) + k_\mu^{(0,1)}(x)  \right) \psi(x) d
  x \leq 2 \varkappa \langle \psi \rangle,
\end{equation}
which by Proposition \ref{Se1pn} yields
\begin{equation}
  \label{C4}
 \forall \mu \in \mathcal{P}_{\rm exp} \qquad   \mu(\breve{\Gamma}^2_*)= \mu(\Gamma^2_*)=1.
\end{equation}
This crucial property of the elements of $\mathcal{P}_{\rm exp}$
will allow us to consider mostly configurations belonging to
$\Gamma^2_*$. In particular, this means that we will use the
following sub-fields of $\mathcal{B}({\Gamma^2})$:
\begin{equation}
\label{C4z} \breve{\mathcal{A}}_* = \{\mathbb{A} \in
\mathcal{B}({\Gamma^2}): \mathbb{A} \subset \breve{\Gamma}^2_*\},
\qquad \mathcal{A}_* = \{\mathbb{A} \in \mathcal{B}(\Gamma^2):
\mathbb{A} \subset \Gamma_*^2\}.
\end{equation}
Performing the same calculations as in obtaining (\ref{W2}) one
readily gets
\begin{equation}
  \label{W3}
  \mu (\varPsi^n) \leq T_n (\varkappa \langle \psi \rangle), \qquad n \in \mathds{N} ,
\end{equation}
which can be used to get the following estimate
\begin{equation}
  \label{W4}
  \int_{\Gamma^2} \exp\left(\beta \varPsi(\gamma) \right) \mu ( d \gamma)
  \leq \exp\left( 2 \varkappa \langle \psi \rangle (e^\beta - 1)
  \right), \qquad \beta >0,
\end{equation}
holding for $\mu\in \mathcal{P}_{\rm exp}$ of type $\leq \varkappa$.

Now we recall that $C_{\rm b}(X)$ (resp. $B_{\rm b}(X)$) stands for
the set of all bounded continuous (resp. bounded measurable)
functions $g:X \to \mathds{R}$. For $\psi$ defined in (\ref{P}), we
then set
\begin{eqnarray}
  \label{T1}
 \varTheta_\psi & = & \{ \theta (x) = g(x) \psi (x) :  g\in C_{\rm
 b}(X), \ \ g(x)\geq 0\}, \\[.2cm] \nonumber
 \varTheta^{+}_\psi & = & \{ \theta \in \varTheta_\psi: \theta(x) >0\}.
\end{eqnarray}
Clearly, each $\theta\in \varTheta_\psi$ is integrable. For such
$\theta$, we also define
\begin{equation}
  \label{C80}
  c_\theta = \sup_{x\in X} \frac{1}{\psi(x)}\log\left(1+{\theta(x)}
  \right),\qquad  \bar{c}_\theta:= e^{c_\theta} -1.
\end{equation}
Then
\begin{equation}
  \label{C801}
 0\leq  \theta (x) \leq \bar{c}_\theta
  \psi(x), \qquad \theta \in \varTheta_\psi.
\end{equation}
Next we define the following measures on $X$
\begin{equation}
  \label{no23}
  (\psi\gamma_i) (dx) = \psi(x) \gamma_i (dx),
  , \qquad \gamma_i\in \Gamma_*, \ \ i=0,1.
\end{equation}
Then
\begin{equation}
  \label{no24}
\varPsi(\gamma) =   (\psi\gamma_0) (X) + (\psi\gamma_1) (X), \qquad
(\psi\gamma_i) (X) = \sum_{x\in \gamma_i}
  \psi(x) =\Psi(\gamma_i) , \ \ i=0,1.
\end{equation}
Let $\mathcal{N}$ be the set of all positive finite Borel measures
on $X$. In view of (\ref{no24}) and (\ref{n021z}), we have that
$\psi\gamma_i\in \mathcal{N}$ for each $\gamma_i \in {\Gamma}_{*}$,
$i=0,1$. Consider
\begin{equation*}
C_{\rm b}^L(X) =\{g \in C_{\rm b}(X): \|g\|_L <\infty \}, \quad
\|g\|_L := \sup_{x,y\in X, \ x\neq y} \frac{|g(x) - g(y)|}{|x-y|},
\end{equation*}
and then define
\begin{equation*}
  \|g\|_{BL} = \|g\|_L + \sup_{x\in X}|g(x)|, \qquad g\in
  C_{\rm b}^L(X),
\end{equation*}
and also
\begin{equation}
  \label{N3}
 \upsilon (\nu, \nu') =\max\left\{1; \sup_{g: \|g\|_{BL}\leq 1} \left\vert
 \nu(g) - \nu'(g)\right\vert \right\}, \qquad \nu, \nu' \in {\mathcal{N}}.
\end{equation}
\begin{proposition}{\cite[Theorem 18]{Dudley}}
  \label{N1pn}
The following three types of the convergence of a sequence
$\{\nu_n\}\subset {\mathcal{N}}$ to a certain $\nu\in \mathcal{N}$
are equivalent:
\begin{itemize}
  \item[(i)] $\nu_n(g) \to \nu(g)$ for all $g\in C_{\rm
  b}(X)$;
  \item[(ii)] $\nu_n(g) \to \nu(g)$ for all $g\in C_{\rm
  b}^L(X)$;
  \item[(iii)] $\upsilon (\nu_n, \nu) \to 0$.
\end{itemize}
\end{proposition}
That is, $\upsilon$ metrizes the weak convergence of the elements of
${\mathcal{N}}$. Our aim now is to metrize ${\Gamma}^2_*$. In view
of (\ref{n021z}), to this end it is enough to metrize
${\Gamma}_{*}$. Set
\begin{equation}
  \label{N4}
  \varTheta^{BL}_\psi = \{ \theta (x) = g(x) \psi(x): \|g\|_{BL}\leq 1\},
\end{equation}
and then define
\begin{equation}
  \label{N4a}
\upsilon^* (\gamma, \gamma') = \upsilon (\psi\gamma_0,
\psi\gamma'_0) + \upsilon (\psi\gamma_1, \psi\gamma'_1).
\end{equation}
Note that
\begin{equation}
  \label{N5}
  \upsilon (\psi\gamma_i, \psi\gamma'_i) =  \max\left\{1; \sup_{\theta\in \varTheta_\psi^{BL}}\left\vert
  \sum_{x\in \gamma_i} \theta(x) - \sum_{x\in
 \gamma'_i}\theta(x)\right\vert \right\}, \qquad \gamma_i , \gamma'_i \in {\Gamma_*}, \ \ i=0,1.
\end{equation}
Before going further, we recall that the set of simple
configurations is defined in (\ref{no16}), see also (\ref{GLS}) and
(\ref{no21}).
\begin{lemma}
  \label{N2pn}
The metric space  $({\Gamma}^2_*, \upsilon^*)$ is complete and
separable. Its metric topology is the weakest topology that makes
continuous all the maps ${\Gamma}^2_* \ni \gamma \mapsto
\sum_{i=0,1}\sum_{x\in \gamma_i}\theta_i(x)$, $\theta_0, \theta_1\in
\varTheta_\psi$.  The set $\breve{\Gamma}^2_*$ defined in
(\ref{no21}) is a $G_\delta$ subset of the Polish space
 $\Gamma^2_*$, and hence is also a Polish spaces. Its
 completion in the metric defined in (\ref{N4a}) is
 ${\Gamma}^2_*$.
\end{lemma}
\begin{proof}
The completeness of $(\Gamma^2_{*},\upsilon^* )$ follows from the
completeness of  $({\Gamma}_{*},\upsilon_* )$, which was obtained in
\cite[Lemma 2.7]{KR}. The second part of the statement follows by
the corresponding property of $\breve{\Gamma}_*$ obtained
\emph{ibid}.
\end{proof}
The following formula summarizes the relationship between the
configuration sets
\begin{equation}
  \label{Jul1}
  \breve{\Gamma}^2_* \subset \Gamma^2_* \subset \Gamma^2.
\end{equation}
Recall that each of them is a Polish space with the topology as
discussed above. Let $\mathcal{B}(\breve{\Gamma}^2_*)$ and
$\mathcal{B}(\Gamma^2_*)$ be the corresponding Borel
$\sigma$-fields, that can be compared with the $\sigma$-fields
introduced in (\ref{C4z}).
\begin{corollary}
  \label{N1co}
The embeddings in (\ref{Jul1}) are continuous. Therefore,
$\mathcal{B}(\breve{\Gamma}^2_*)=\breve{\mathcal{A}}_* =
\{\mathbb{A}\in \mathcal{B}(\Gamma^2_*): \mathbb{A}\subset
\breve{\Gamma}^2_*\}$ and $\mathcal{B}(\Gamma^2_*)=\mathcal{A}_*$.
\end{corollary}
\begin{proof}
The continuity of ${\Gamma}^2_*\subset {\Gamma^2}$  follows by the
fact that $C_{\rm cs}(X)$ is a proper subset of $\varTheta_\psi$.
The other one follows by Lemma \ref{N2pn}. The stated equality of
the $\sigma$-fields follows by Kuratowski's theorem \cite[Theorem
3.9, page 21]{Part}.
\end{proof}
\begin{remark}
  \label{derk} The aforementioned equality of the $\sigma$-fields
  allows one to redefine each $\mu\in \mathcal{P}({\Gamma}^2)$
  possessing the property $\mu(\breve{\Gamma}^2_*)=1$ as a probability
  measure on $({\Gamma}^2_*, \mathcal{B}({\Gamma}^2_*))$ or $(\breve{\Gamma}^2_*, \mathcal{B}(\breve{\Gamma}^2_*))$. By
(\ref{C4}) this relates to all $\mu\in \mathcal{P}_{\rm exp}$.
\end{remark}

\subsection{Families of test functions}
 For $\theta\in \varTheta_\psi$, see (\ref{T1}), we set
\begin{equation}
  \label{de}
  \varsigma_\tau^\theta(x) = \tau - \frac{1}{\psi(x)} \log ( 1 + \theta
  (x)), \qquad \varSigma = \{\varsigma^\theta_\tau : \theta \in
  \varTheta_\psi, \ \tau > c_\theta \},
\end{equation}
where $c_\theta$ is as in (\ref{C80}). Then $\varSigma \subset
C_{\rm b}(X)$ and its elements are separated away from zero. It is
closed with respect to pointwise addition since $\theta + \theta' +
\theta \theta'$ belongs to $\varTheta_\psi$ whenever $\theta,
\theta'\in\varTheta_\psi$. For $\tau_i > c_{\theta_i}$ and $\gamma_i
\in
  {\Gamma}_{*}$, $i=0,1$, we define, see (\ref{no23}),
\begin{gather}
  \label{T4}
  \widetilde{F}_{\tau_i}^{\theta_i} (\gamma_i) = \prod_{x\in
  \gamma_i} (1+\theta_i(x)) e^{-\tau_i \psi(x)} =
  \exp\left( - (\psi\gamma_i)
  (\varsigma_{\tau_i}^{\theta_i})\right),  \\[.2cm] \nonumber
\widetilde{F}_{\tau}^{\theta} (\gamma) =
\widetilde{F}_{\tau_0}^{\theta_0} (\gamma_0)
\widetilde{F}_{\tau_1}^{\theta_1} (\gamma_1),
\end{gather}
and also
\begin{equation}
  \label{T4a}
\widetilde{\mathcal{F}} =\{\widetilde{F}_{\tau}^{\theta}: \tau
=(\tau_0,\tau_1), \ \tau_i > c_{\theta_i}, \ \theta_i \in
\varTheta_\psi, \ i=0,1\},
\end{equation}
which includes the case $\widetilde{F}_{\tau}^{\theta} \equiv 1$
corresponding to the zero $\tau$ and $\theta$. Note that in
expressions like those in (\ref{T4}), (\ref{T4a}), by $\theta$ we
understand  $(\theta_0, \theta_1)$, $\theta_i\in \varTheta_\psi$.
\begin{definition}
\cite[page 111]{EK}
  \label{V1df}
A sequence $\{\varPhi_n\}_{n\in \mathds{N}} \subset B_{\rm
b}({\Gamma}^2_*)$ is said to boundedly and pointwise (bp-) converge
to a given $\varPhi\in B_{\rm b}({\Gamma}^2_*)$ if it converges
pointwise and $$\sup_{n\in \mathds{N}} \sup_{\gamma\in {\Gamma}^2_*}
|\varPhi_n(\gamma)| < \infty.$$ The bp-closure of a set $M\subset
B_{\rm b}({\Gamma}^2_*)$ is the smallest subset of $B_{\rm
b}({\Gamma}^2_*)$ that contains $M$ and is closed under the
bp-convergence. In a similar way, one understands also the
bp-convergence of a sequence of functions $\phi_n :X \to
\mathds{R}$.
\end{definition}
It is quite standard, see \cite[Proposition 4.2, page 111]{EK} or
\cite[Lemma 3.2.1, page 41]{Dawson}, that $C_{\rm b}(X)$ contains a
countable family of nonnegative functions, $\{g_j\}_{j\in
\mathds{N}}$, which is {\it convergence determining} and such that
its linear span is bp-dense in $B_{\rm b}(X)$. This means that a
sequence $\{\nu_n\}\subset \mathcal{N}$ weakly converges to a
certain $\nu$ if and only if $\nu_n(g_j) \to \nu (g_j)$, $n \to
+\infty$ for all $j\in \mathds{N}$. One may take such a family
containing the constant function $g(x)\equiv 1$ and closed with
respect to pointwise addition. Moreover, one may assume that
\begin{equation}
  \label{S}
 \forall j\in \mathds{N} \qquad  \inf_{x\in X} g_j(x) =: \zeta_j >0.
\end{equation}
If this is not the case for a given $g_j$, in place of it one may
take $\tilde{g}_j(x)= g_j(x) + \zeta_j$ with some $\zeta_j>0$. The
new set $\{\tilde{g}_j\}$ has both mentioned properties and also
satisfies (\ref{S}). Then assuming the latter we conclude that
\begin{equation}
  \label{T3z}
 \varSigma_0 := \{g_j\}_{j\in \mathds{N}} \subset \varSigma,
\end{equation}
where the latter is defined in (\ref{de}). To see this, for a given
$g_j$, take $\tau_j \geq \sup_{x} g_j(x)$ and then set
\begin{equation}
  \label{MG}
  \theta_j (x) = \exp\bigg{(} [\tau_j - g_j(x)]\psi(x) \bigg{)} -1.
\end{equation}
Clearly, $\theta_j(x) \geq 0$. Since $\psi^n (x) \leq \psi(x)$,
$n\in \mathds{N}$, we have that $\theta_j(x) \leq e^{\tau_j}
\psi(x)$, and hence $\{\theta_j\}_{j\in \mathds{N}}\subset
\varTheta_\psi$, see (\ref{T1}). At the same time,
$\varsigma^{\theta_j}_{\tau_j} =g_j$ and $c_{\theta_j} = \sup_{x}
(\tau_j - g_j(x))< \tau_j$ in view of (\ref{S}). By (\ref{MG}),
(\ref{T3z}) and \cite[Theorem 3.4.5, page 113]{EK}, see also
\cite[page 43]{Dawson}, one readily proves the following statement.
\begin{proposition}
  \label{T1pn}
The set $\widetilde{\mathcal{F}}$ defined in (\ref{T4a}) is closed
with respect to pointwise multiplication. Additionally:
\begin{itemize}
  \item[(i)] It is separating: $\mu_1 (F)=\mu_2 (F)$, holding for all
  $F\in \widetilde{\mathcal{F}}$, implies $\mu_1 = \mu_2$ for all $\mu_1,
  \mu_2 \in \mathcal{P}({\Gamma}^2_*)$.
\item[(ii)] It is convergence determining: if a sequence $\{\mu_n\}_{n\in \mathds{N}}\subset \mathcal{P}({\Gamma}^2_*)$
is such that $\mu_n(F) \to \mu(F)$, $n\to +\infty$ for all $F\in
\widetilde{\mathcal{F}}$ and some $\mu\in
\mathcal{P}({\Gamma}^2_*)$, then $\mu_n(F) \to \mu(F)$ for all $F\in
C_{\rm b}({\Gamma}^2_*)$.
\item[(iii)]   The set $B_{\rm b}({\Gamma}^2_*)$ is the bp-closure of the linear span of $\widetilde{\mathcal{F}}$.
\end{itemize}
\end{proposition}
Now we introduce another class of functions $F:{\Gamma}^2_*\to
\mathds{R}$ which we then use to define the domain of ${L}$. For
$k\in \mathds{N}$, $v_{i,1}, \dots, v_{i,k} \in C_{\rm b}(X)$ and
$\gamma_i\in {\Gamma}_{*}$, $i=0,1$, we write
\begin{equation}
  \label{de1}
  \mathbf{v}_i(\mathbf{x}^k) = v_{i,1}(x_1) \cdots v_{i,k}(x_k), \qquad \gamma_i \setminus \mathbf{x}^k =
  \gamma_i \setminus \{x_1, \dots  , x_k\},
\end{equation}
see subsect. \ref{2.1ss}. As is (\ref{no3}), we will omit $k$ if the
dimension of $\mathbf{x}$ is known from the context. Then for $m=
(m_0, m_1) \in \mathds{N}^2_0$, $v_{i,j}\in \varTheta_\psi^{+}$ (see
(\ref{T1})) and $\tau =(\tau_0,\tau_1)$, $\tau_i>0$, we set, see
also (\ref{no3a}) and (\ref{no20A}),
\begin{gather}
  \label{de2}
  \widehat{F}^{m_i}_{\tau_i} (\mathbf{v}_i|\gamma_i)=
  \sum_{\mathbf{x}^{m_i}\in \gamma_i} \mathbf{v}_i
  (\mathbf{x}^{m_i}) \exp\left( - \tau_i \Psi(\gamma_i \setminus \mathbf{x}^{m_i})
  \right), \qquad i = 0,1,\\[.2cm] \nonumber
\widehat{F}^{m}_{\tau} (\mathbf{v}|\gamma)=
\widehat{F}^{m_0}_{\tau_0}
(\mathbf{v}_0|\gamma_0)\widehat{F}^{m_1}_{\tau_1}
(\mathbf{v}_1|\gamma_1), \qquad \gamma\in
{\Gamma}^2_*,\\[.2cm] \nonumber \widehat{\mathcal{F}} = \{\widehat{F}^{m}_{\tau} (\mathbf{v}|\cdot): m\in \mathds{N}^2_0,
\ v_{i,j}\in \varTheta_\psi^{+}, \ \tau_i >0 , \ i=0,1 \}.
\end{gather}
Here $\widehat{F}^{(0,0)}_{\tau} \equiv 0$, which is also an element
of $\widehat{\mathcal{F}}$.
\begin{proposition}
  \label{TH1pn}
For each $m=(m_0,m_1) \in \mathds{N}_0^2$, $\tau = (\tau_0,\tau_1)
>0$ and $v_{i,j}\in   \varTheta_\psi^{+}$,
$j=1,. \dots , m_j$, $i=0,1$, it follows that
$\widehat{F}^{m}_{\tau} (\mathbf{v}|\cdot) \in C_{\rm
b}({\Gamma}^2_*)$.
\end{proposition}
\begin{proof}
Clearly, it suffices to show that $ \widehat{F}^{m_i}_{\tau_i}
(\mathbf{v}_i|\cdot)  \in C_{\rm b}({\Gamma}_{*})$. To prove the
continuity in question we rewrite the first line of (\ref{de2}) in
the form
\begin{gather}
  \label{de3}
\widehat{F}^{m_i}_{\tau_i} (\mathbf{v}_i|\gamma_i)= \exp\left( -
\tau_i \Psi(\gamma_i)\right) \sum_{x_1 \in \gamma_i}\sum_{x_2\in
\gamma_i\setminus x_1} \cdots \sum_{x_{m_i}\in \gamma_i \setminus
\{x_1 , \dots , x_{m_i-1}\}} u_{1} (x_1)
\cdots u_{m_i}(x_{m_i}), \\[.2cm] \nonumber u_j(x) :=
v_{i,j } (x) e^{\tau_i \psi(x)}, \qquad j=1, \dots , m_i.
\end{gather}
Obviously, every $u_j$ belongs to $\varTheta_\psi^{+}$. For each
$m_i$, by (\ref{102}) the sum in (\ref{de3}) can be written as the
sum of the products of the functions
\[
\gamma_i \mapsto U_{l_1, \dots l_s} (\gamma_i) = \sum_{y\in
\gamma_i} u_{l_1}(y) \cdots u_{l_s}(y).
\]
Each such $U_{l_1, \dots l_s}$ is continuous as $\varTheta_\psi^{+}$
is closed under pointwise multiplication. Obviously, $\gamma_i
\mapsto \Psi(\gamma_i)$ is also continuous, which yields the
continuity of the function defined in (\ref{de3}). To prove its
boundedness, we use the estimate $u_j(x) \leq c \psi(x) e^{\tau_i
\psi(x)} \leq c e^\tau \psi(x)$. Then, see (\ref{no3a}),
\begin{gather}
  \label{ND}
  \widehat{F}^{m_i}_{\tau_i} (\mathbf{v}_i|\gamma_i) \leq
  \exp\left( - \tau_i \Psi(\gamma_i)\right) \sum_{{\bf x}\in
  \gamma_i} u^{\otimes m_i}({\bf x}) \leq c^{m_i} e^{m_i\tau_i} \exp\left( - \tau_i
  \Psi(\gamma_i)\right) \left(\sum_{x\in \gamma_i} \psi(x)
  \right)^{m_i} \\[.2cm] \nonumber = c^{m_i}
  \Psi^{m_i}(\gamma_i) \exp\bigg{(}- \tau_i [\Psi(\gamma_i) - m_i] \bigg{)}
\leq \left( \frac{c m_i}{\tau_i}\right)^{m_i} e^{m_i(\tau_i - 1)},
\end{gather}
which yields the proof.
\end{proof}

\section{The Statement}

\subsection{The Kolmogorov operator}

 Regarding the parameters of the Kolmogorov operator introduced in
 (\ref{L}), we will assume that the following holds
 \begin{gather}
  \label{C7}
  \max_{i=0,1}\sup_{x} a_i(x) =: \|a\| <\infty,
\end{gather}
\begin{gather}
\label{C7a}
   \max_{i=0,1} \int_{X}\left( 1- e^{ -\phi_i(x)}\right) d x =: \varphi <
  \infty,
\end{gather}
and
\begin{equation}
  \label{C8}
  \int_{X} |x|^{l}a_i(x) d x =:\bar{a}^{(l)}_i < \infty, \qquad
  {\rm for} \ \ i=0,1 \  {\rm and} \ \ l =0,1 , \dots ,  d+1.
\end{equation}
\begin{definition}
  \label{THdf}
By $\mathcal{D}({L})$ we denote the linear span of the set
$\widetilde{\mathcal{F}}\cup \widehat{\mathcal{F}}$, where
$\widetilde{\mathcal{F}}$ and $\widehat{\mathcal{F}}$ are defined in
(\ref{T4a}) and (\ref{de2}), respectively.
\end{definition}
Our aim now is to show that ${L}\widehat{F}^{m}_{\tau}
(\mathbf{v}|\cdot)\in B_{\rm b}({\Gamma}^2_*)$ holding for each
$\widehat{F}^{m}_{\tau} (\mathbf{v}|\cdot)\in
\widehat{\mathcal{F}}$. To this end, for $x\in \gamma_i$, $y\in X$
and a suitable $F$, we define
\begin{equation}
  \label{de4}
  (\nabla_i^{y,x}  F)(\gamma_i)= F(\gamma_i \setminus x \cup y) -
  F(\gamma_i).
\end{equation}
By (\ref{de2}) we can write
\begin{equation}
  \label{de5}
\widehat{F}^{m_i}_{\tau_i} (\mathbf{v}_i|\gamma_i) = \sum_{z \in
\gamma_i} v_{i,1}(z) \widehat{F}^{m_i-1}_{\tau_i}
(\mathbf{v}_i^1|\gamma_i\setminus z),
\end{equation}
where ${\bf v}_i^1$ is obtained by setting $j=1$ in the formula
\begin{equation}
  \label{de6}
\mathbf{v}_i^j ({\bf x}^{m_i-1}) = v_{i,1} (x_1) \cdots v_{i,j-1}
(x_{j-1}) v_{i,j+1}(x_{j+1}) \cdots  v_{i,m_i}(x_{m_i}),
\end{equation}
see (\ref{de1}), and
\begin{equation*}
\widehat{F}^{m_i-1}_{\tau_i} (\mathbf{v}_i^1|\gamma_i\setminus z) =
\sum_{{\bf x}^{m_i-1} \in \gamma_i \setminus z} {\bf v}_i^1 ({\bf
x}^{m_i-1}) \exp\left( - \tau_i \Psi({(\gamma}_i \setminus z)
\setminus {\bf x}^{m_i-1}) \right).
\end{equation*}
According to (\ref{de4}) and (\ref{de5}) we then get
\begin{equation}
  \label{de8}
\nabla_i^{y,x}  \widehat{F}^{m_i}_{\tau_i} (\mathbf{v}_i|\gamma_i)=
[v_{i,1} (y) - v_{i,1}(x)] \widehat{F}^{m_i-1}_{\tau_i}
(\mathbf{v}_i^1|\gamma_i\setminus x) + \sum_{z \in \gamma_i\setminus
x} v_{i,1}(z) \nabla_i^{y,x} \widehat{F}^{m_i-1}_{\tau_i}
(\mathbf{v}_i^1|\gamma_i\setminus z).
\end{equation}
By iterating the latter we arrive at, see also (\ref{de6}),
\begin{gather}
  \label{de9}
\nabla_i^{y,x}  \widehat{F}^{m_i}_{\tau_i} (\mathbf{v}_i|\gamma_i)=
\sum_{j=1}^{m_i}[v_{i,j} (y) -
v_{i,j}(x)]\widehat{F}^{m_i-1}_{\tau_i}
(\mathbf{v}_i^j|\gamma_i\setminus x) \\[.2cm] \nonumber +
\bigg{(}e^{-\tau_i \psi(y)} - e^{-\tau_i \psi(x)}
\bigg{)}\widehat{F}^{m_i}_{\tau_i} (\mathbf{v}_i|\gamma_i\setminus
x).
\end{gather}
For $\theta\in \varTheta_\psi$ and $a_i$ as in (\ref{C7}),
(\ref{C8}), we set
\begin{equation}
  \label{TH3}
  (a_i \ast \theta)(x) = \int_X a_i(x-y) \theta(y) dy =  \int_X a_i(y) \theta(x-y) dy,
  \qquad i=0,1.
\end{equation}
Then $a_i\ast \theta\in C_{\rm b}(X)$. Moreover, by (\ref{C801}) and
(\ref{C8}) we obtain
\begin{eqnarray}
  \label{C81}
 (a_i \ast \theta)(x) & \leq & \bar{c}_\theta \psi(x) \int_X
 \left(1 + |x|^{d+1}\right)a_i(x-y) \psi(y) d y \\[.2cm] \nonumber
 &\leq & \bar{c}_\theta \psi(x) \left[\bar{a}_i^{(0)} + \int_X \left(|x-y| + |y| \right)^{d+1} a_i(x-y) \psi(y) d y \right]
\\[.2cm] \nonumber
 &= & \bar{c}_\theta \psi(x) \left[\bar{a}_i^{(0)} + \sum_{l=0}^{d+1} { d +1 \choose l} \int_X |x-y|^{d+1-l} |y|^l \psi(y) a_i(x-y)d y  \right]
\\[.2cm] \nonumber
 &\leq  & \bar{c}_\theta \psi(x) \left[\bar{a}_i^{(0)} + \sum_{l=0}^{d+1} { d +1 \choose
 l}\bar{a}_i^{(l)} \right] =: \bar{c}_\theta \bar{\alpha}_i \psi(x),
\end{eqnarray}
where we have used also the fact that $|y|^l \psi(y)\leq 1$ for all
$l=0, \dots , d+1$. Therefore, for each $\theta \in \varTheta_\psi$,
it follows that
\begin{eqnarray}
  \label{TH4}
  (a_i \theta)(x) &:= & (a_i \ast \theta)(x) + \theta(x) \leq \bar{c}_\theta (\bar{\alpha}_i +1)
  \psi(x) \leq \bar{c}_\theta c_a \psi(x), \\[.2cm] \nonumber c_a &:= &\max\{\bar{\alpha}_0; \bar{\alpha}_1\} + 1.
\end{eqnarray}
At the same time,
\begin{equation}
  \label{TH5}
  e^{-\tau_i \psi(y)} -  e^{-\tau_i \psi(x)} \leq \tau_i \psi(y)
  \psi(x) \big| |x|^{d+1} - |y|^{d+1}\big|, \qquad i=0,1,
\end{equation}
which after calculations similar to those in (\ref{C81}) yields
\begin{equation}
  \label{TH6}
  \int_{X} a_{i}(x-y) \big{|}  e^{-\tau_i \psi(y)} -  e^{-\tau_i \psi(x)} \big{|}
  dy \leq \tau_i c_a \psi(x),
\end{equation}
where $c_a$ is as in (\ref{TH4}). For ${\bf v}_i$ as in (\ref{de1})
and $a_i\theta$ as in (\ref{TH4}), we set, cf. (\ref{de6}),
\begin{equation}
  \label{de10}
  (a_i^j {\bf v}_i)({\bf x}^{m_i}) = v_{i,1}(x_1) \cdots
  v_{i,j-1}(x_{j-1}) (a_i v_{i,j})(x_j) v_{i,j+1}(x_{j+1}) \cdots
  v_{i,m_i}(x_{m_i}).
\end{equation}

Then by (\ref{de2}), (\ref{de8}), (\ref{de9}) and also by
(\ref{TH4}), (\ref{TH6}), (\ref{de10}) we arrive at
\begin{eqnarray}
  \label{TH7}
  \left| {L} \widehat{F}^m_\tau ({\bf
  v}|\gamma)\right| & \leq & \left(\int_X \sum_{x\in
  \gamma_0} a_0(x-y) \left|\nabla_0^{y,x} \widehat{F}^{m_0}_{\tau_0} ({\bf v}_0|\gamma_0) \right|d y \right)
  \widehat{F}^{m_1}_{\tau_1} ({\bf v}_1|\gamma_1) \\[.2cm]
  \nonumber & + & \left(\int_X \sum_{x\in
  \gamma_1} a_1(x-y) \left|\nabla_1^{y,x} \widehat{F}^{m_1}_{\tau_1} ({\bf v}_1|\gamma_1) \right|d y \right)
  \widehat{F}^{m_0}_{\tau_0} ({\bf v}_0\gamma_0) \\[.2cm]
  \nonumber & \leq & \left(\sum_{j=1}^{m_0}\widehat{F}^{m_0}_{\tau_0} (a_0^j{\bf v}_0|\gamma_0) + \tau_0 c_a \bar{c}({\bf v}_0)
  \widehat{F}_{\tau_0}^{m_0+1}(\gamma_0)  \right) \widehat{F}^{m_1}_{\tau_1} ({\bf v}_1|\gamma_1) \\[.2cm]\nonumber
  & + & \left(\sum_{j=1}^{m_1}\widehat{F}^{m_1}_{\tau_1} (a_1^j{\bf v}_1|\gamma_1) + \tau_1 c_a \bar{c}({\bf v}_1)
  \widehat{F}_{\tau_1}^{m_1+1}(\gamma_1)  \right) \widehat{F}^{m_0}_{\tau_0} ({\bf
  v}_0|\gamma_0),
\end{eqnarray}
where $\bar{c}(\mathbf{v}_i)= \max_{j} \bar{c}_{v_{i,j}}$, see
(\ref{C80}), and
\begin{equation}
  \label{de11}
\widehat{F}_{\tau_i}^{m_i}(\gamma_i) =
\widehat{F}_{\tau_i}^{m_i}({\bf v}_i|\gamma_i), \quad {\rm with}
\quad {\bf v}_i ({\bf x}) = \psi(x_1) \cdots \psi(x_{m_i}),
\end{equation}
see the first line of (\ref{de2}). Then the boundedness of ${L}
\widehat{F}^m_\tau ({\bf
  v}|\cdot)$ follows by Proposition \ref{TH1pn}. Let us prove now
that ${L}\widetilde{F}_\tau^\theta \in B_{\rm b}({\Gamma}_*)$,
holding for all $\theta_i\in \varTheta_\psi$ and $\tau_i >
c_{\theta_i}$, $i =0,1,$ see (\ref{T4}). According to (\ref{de4}) we
have
\begin{equation*}
 \nabla_i^{y,x}\widetilde{F}_{\tau_i}^{\theta_i}(\gamma_i) =
 \bigg{(}\big{[}e^{- \tau_i\psi (y)} - e^{- \tau_i\psi(x)}\big{]} +
 [\theta_i(y)e^{- \tau_i\psi (y)} - \theta_i(x)e^{- \tau_i\psi (x)}] \bigg{)}\widetilde{F}_{\tau_i}^{\theta_i}(\gamma_i\setminus
 x).
\end{equation*}
Then by means of (\ref{TH5}) and (\ref{C81}) we obtain
\begin{eqnarray*}
 & & \ \  \left| {L} \widetilde{F}_\tau^\theta (\gamma)\right|
  \leq  \widetilde{Q}(\gamma) \widetilde{F}_{\tau-\tau^0}^\theta
 (\gamma), \\[.2cm] \nonumber \widetilde{Q}(\gamma)
 & := & (\tau_0 + \bar{c}_{\theta_0}) c_a e^{\tau_0}
 \Psi(\gamma_ 0) e^{-\tau^0_0 \Psi(\gamma_0)} + (\tau_1 + \bar{c}_{\theta_1}) c_a e^{\tau_1}
 \Psi(\gamma_ 1) e^{-\tau^0_1 \Psi(\gamma_1)} \\[.2cm] \nonumber &
 \leq &
 e^{\tau_0} c_a \frac{\tau_0 + \bar{c}_{\theta_0}}{e \tau_0^0}  + e^{\tau_1} c_a \frac{\tau_1 + \bar{c}_{\theta_1}}{e
 \tau_1^0},
 \end{eqnarray*}
where $\tau^0=(\tau^0_0 , \tau_1^0)$, $\tau^0_i >0$, is chosen in
such a way that $\tau_i-\tau_i^0 > c_{\theta_i}$, which is possible
since $\tau_i > c_{\theta_i}$. Then the boundedness in question
follows by (\ref{T4}). The  next statement summarizes the properties
of $\mathcal{D}({L})$.
\begin{proposition}
  \label{T3pn}
The set $\mathcal{D}({L})$ introduced in Definition \ref{THdf} has
the following properties:
\begin{itemize}
  \item[{\it (a)}] $\mathcal{D}({L}) \subset C_{\rm
  b}({\Gamma}^2_*)$; $\ {L}:\mathcal{D}({L}) \to B_{\rm
  b}({\Gamma}^2_*)$.
  \item[{\it (b)}] The set $B_{\rm
  b}({\Gamma}^2_*)$ is the bp-closure of
  $\mathcal{D}({L})$.
    \item[{\it (c)}] $\mathcal{D}({L})$ is separating for
    $\mathcal{P}_{\rm exp}$.
    \item[{\it (d)}] For each $F\in \widetilde{\mathcal{F}}$, the measure $F \mu/\mu(F)$
    belongs to $\mathcal{P}_{\rm exp}$.
\end{itemize}
\end{proposition}
\begin{proof}
Claim (a) has been just proved. Claims (b) and (c) follow by
Proposition \ref{T1pn}. Claim (d) holds true since $F\in
\widetilde{\mathcal{F}}$ is positive and bounded, and thus
multiplication does not affect the property defined in (\ref{no5}),
(\ref{no7}).
\end{proof}

\subsection{Formulating the result}
Following \cite[Chapter 5]{Dawson} - and similarly as in \cite{KR}
-- we will obtain the Markov process in question by solving a {\it
restricted initial value martingale problem} for $({L},
\mathcal{D}({L}), \mathcal{P}_{\rm exp})$. Here we explicitly employ
the complete metric $\upsilon^*$ of $\Gamma^2_*$, defined in
(\ref{N5}). Since the elements of $\mathcal{P}_{\rm exp}$ ``do not
distinguish" between multiple and single configurations, see
Proposition \ref{Se1pn}, one may expect that the constructed Markov
process has the corresponding property. We will show that it does.
Note that the direct construction of the process with values in
$\breve{\Gamma}^2_*$ is rather impossible in this way  as the latter
space is not complete in $\upsilon^*$.

To proceed further, we introduce the corresponding spaces of
c{\`a}dl{\`a}g paths. By $\mathfrak{D}_{\mathds{R}_{+}}
(\breve{\Gamma}^2_*)$ and $\mathfrak{D}_{\mathds{R}_{+}}
(\Gamma^2_*)$ we denote the spaces of c{\`a}dl{\`a}g maps
$[0,+\infty)=:\mathds{R}_{+} \ni t \mapsto \gamma_t \in
\breve{\Gamma}^2_*$ and $\mathds{R}_{+} \ni t \mapsto \gamma_t \in
\Gamma^2_*$, respectively. Then the evaluation maps are
$\varpi_t(\gamma)=\gamma_t$, $\gamma\in
\mathfrak{D}_{\mathds{R}_{+}} ({\Gamma}^2_*)$, $t\in
\mathds{R}_{+}$; hence,
\begin{equation*}
\varpi_t^{-1}(\mathbb{A}) = \{ \gamma\in
\mathfrak{D}_{\mathds{R}_{+}} ({\Gamma}^2_*): \varpi_t(\gamma) =
\gamma_t \in \mathbb{A}\}, \qquad \mathbb{A}\in
\mathcal{B}(\Gamma^2_*).
\end{equation*}
Analogously one defines $\mathfrak{D}_{[s,+\infty)}
(\breve{\Gamma}^2_*)$, $\mathfrak{D}_{[s,+\infty)} (\Gamma^2_*)$,
$s>0$. For $s, t \geq 0$, $s < t$, by $\mathfrak{F}^0_{s,t}$ we
denote the $\sigma$-field of subsets of
$\mathfrak{D}_{[s,+\infty)}(\Gamma^2_*)$) generated by the family
 $\{ \varpi_u: u\in
[s,t]\}$. Then we set
\[
\mathfrak{F}_{s,t}= \bigcap_{\varepsilon
>0}\mathfrak{F}^0_{s,t+\varepsilon}, \qquad \mathfrak{F}_{s,+\infty}
=\bigvee_{n\in \mathds{N}} \mathfrak{F}_{s,s+n}.
\]
In the next definition -- which is an adaptation of the
corresponding definition in \cite[Section 5.1, pages 78, 79]{Dawson}
-- we deal with families of probability measures $\{P_{s,\mu}:s\geq
0, \mu \in \mathcal{P}_{\rm exp}\}$ defined on
$(\mathfrak{D}_{[s,+\infty)} ( {\Gamma}^2_*),
 {\mathfrak{F}}_{s,+\infty})$. Depending on the context, each $\mu\in
\mathcal{P}_{\rm exp}$ is considered as a measure either on
$\breve{\Gamma}^2_*$ or $\Gamma^2_*$, see Remark \ref{derk}. Since
$\breve{\Gamma}^2_*$ and $\Gamma^2_*$ are Polish spaces, both
$\mathfrak{D}_{[s,+\infty)} ( \breve{\Gamma}^2_*)$ and
$\mathfrak{D}_{[s,+\infty)} (\Gamma^2_*)$ are also Polish. The
latter one is complete in Skorohod's metric, see \cite[Theorem 5.6,
page 121]{EK}. Then a probability measure $P$ on
$(\mathfrak{D}_{[s,+\infty)} ( {\Gamma}^2_*),
{\mathfrak{F}}_{s,+\infty})$ with the property
$P(\mathfrak{D}_{[s,+\infty)} ( \breve{\Gamma}^2_*))=1$ can be
redefined as a measure on $\mathfrak{D}_{[s,+\infty)}
(\breve{\Gamma}^2_*)$, that holds for all $s\geq0$.
\begin{definition}
 \label{A1df}
A family of probability measures  $\{P_{s,\mu}: s\geq 0, \ \mu\in
\mathcal{P}_{\rm exp}\}$ is said to be a solution of the restricted
initial value martingale problem for $({L}, \mathcal{D}({L}),
\mathcal{P}_{\rm exp})$ if for all $s\geq 0$ and $\mu \in
\mathcal{P}_{\rm exp}$, the following holds:  (a) $P_{s,\mu}\circ
\varpi_s^{-1} = \mu$;  (b) $P_{s,\mu}\circ \varpi_t^{-1} \in
\mathcal{P}_{\rm exp}$ for all $t
>s$;  (c)  for each $F \in
\mathcal{D}({L})$,  $t_2 \geq t_1 \geq s$ and any bounded function
${\sf G}:\mathfrak{D}_{[s,+\infty)} ({\Gamma}^2_*) \to \mathds{R}$
which is ${\mathfrak{F}}_{s,t_1}$-measurable, the function
\begin{eqnarray}
  \label{A120}
{\sf H}(\gamma) & := & \left[F(\varpi_{t_2} (\gamma)) -
F(\varpi_{t_1} (\gamma)) - \int_{t_1}^{t_2} ({L} F) (\varpi_u
(\gamma)) d u \right] {\sf G}(\gamma)
\end{eqnarray}
is such that
\begin{gather*}
\int_{{\mathfrak{D}}_{[s,+\infty)}(\Gamma_*^2)}{\sf H}(\gamma)
P_{s,\mu} (d \gamma) = 0.
\end{gather*}
The restricted initial value martingale problem is
 {\it well-posed} if for each $s\geq 0$ and $\mu\in
\mathcal{P}_{\rm exp}$, there exists a unique  path measure
$P_{s,\mu}$ satisfying conditions (a), (b) and (c) mentioned above.
\end{definition}
\begin{remark}
  \label{de1rk}
 Instead of taking all ${\sf G}$ as in claim (c) of Definition \ref{A1df}, it is enough to take
in the form
\begin{equation}
  \label{A1200}
 {\sf G} (\gamma) = F_1 ({\varpi}_{s_1} (\gamma)) \cdots F_m ({\varpi}_{s_m}
  (\gamma)),
\end{equation}
with all possible choices of $m\in \mathds{N}$, $F_1, \dots, F_m \in
\widetilde{\mathcal{F}}$ (see Proposition \ref{T1pn}), and $s\leq
s_1 < s_2 <\cdots < s_m \leq t_1$, see \cite[eq. (3.4), page
174]{EK}.
\end{remark}
Now we can formulate our result.
\begin{theorem}
  \label{1tm}
For the model defined in (\ref{L}) and satisfying (\ref{C7}) --
(\ref{C8}), the following is true:
\begin{itemize}
  \item[(i)] The restricted initial value martingale problem for
  $({L}, \mathcal{D}({L}), \mathcal{P}_{\rm exp})$
is well-posed in the sense of Definition \ref{A1df}.
\item[(ii)] Its solution has the property $P_{s,\mu}
(\mathfrak{D}_{[s,+\infty)}(\breve{\Gamma}^2_*))=1$, holding for all
$s\geq 0$ and $\mu \in \mathcal{P}_{\rm exp}$.
\item[(iii)] The stochastic
process related to the family $$(\mathfrak{D}_{[s,+\infty)}
(\Gamma_*^2), \mathfrak{F}_{s,+\infty}, \{\mathfrak{F}_{s,t}: t\geq
s\}, \{P_{s,\mu}:\mu \in \mathcal{P}_{\rm exp}\})_{s\geq 0}$$ is
Markov. This  means that, for all $t>s$ and $\mathbb{B}\in
\mathfrak{F}_{t, +\infty}$, the following holds
\[
P_{s,\mu} (\mathbb{B}|\mathfrak{F}_{s,t}) = P_{s,\mu}
(\mathbb{B}|\mathfrak{F}_{t}), \qquad P_{s,\mu} \ - \ {\rm almost} \
{\rm surely}.
\]
Here  $\mathfrak{F}_{t}$ is the smallest $\sigma$-field of subsets
of $\mathfrak{D}_{[s,+\infty)}(\Gamma^2_*)$ that contains all
$\varpi^{-1}_t(\mathbb{A})$, $\mathbb{A}\in
\mathcal{B}(\Gamma^2_*)$.
\end{itemize}
\end{theorem}
In the proof of Theorem \ref{1tm}, we crucially use the
Fokker-Planck equation (\ref{FPE}).
\begin{definition} \label{A01df}
For a given $s\geq 0$, a map $[s,+\infty)\ni t \mapsto \mu_t\in
\mathcal{P}({\Gamma}^2_*)$ is said to be measurable if the maps
$[s,+\infty) \ni t \mapsto \mu_t(\mathbb{A}) \in \mathds{R}$ are
measurable for all $\mathbb{A}\in \mathcal{B}({\Gamma}^2_*)$. Such a
map is said to be a solution of the Fokker-Planck equation for
$({L}, \mathcal{D}({L}))$ if for each $F\in \mathcal{D}({L})$ and
any $t_2
> t_1 \geq s$, the equality in (\ref{FPE}) holds true.
\end{definition}
Note that ${L}F \in B_{\rm b}({\Gamma}_*)$; hence, the integral in
the right-hand side (\ref{FPE}) is well defined for measurable
$t\mapsto \mu_t$.
\begin{remark}
  \label{A2rk}
By taking ${\sf G}\equiv 1$ in (\ref{A120}) one comes to the
following conclusion. Let $\{{P}_{s,\mu}: s\geq 0, \mu \in
\mathcal{P}_{\rm exp}\}$ be a solution as in Definition \ref{A1df}.
Then for each $s$ and $\mu \in \mathcal{P}_{\rm exp}$, the map
$[s,+\infty) \ni t \mapsto {P}_{s,\mu}\circ {\varpi}_{t}^{-1}$
solves (\ref{FPE}) for all $t_2>t_1\geq s$.
\end{remark}

\subsection{Comments}
\label{comSS} In statistical physics, the first model where
attraction is induced by an inter-component repulsion was proposed
by Widom and Rowlinson in \cite{WiR}. A mathematically rigorous
proof that the Gibbs states in this model can be multiplicity was
done by Ruelle in \cite{R1}. In both these works, the repulsion is
of the hard-core type, which in our case corresponds to $\phi_0(x) =
\phi_1(x) = \ell_r(|x|)$, $r>0$, with $\ell_r(\rho) = 0$ for
$\rho>r$, and $\ell_r(\rho) = +\infty$ for $\rho\leq r$. In the
single-component version of the Widom-Rowlinson model, the energy of
the multiparticle attraction induced by the hard core repulsion in a
finite configuration $\eta_0 \subset \Gamma$ is given by the
formula, see \cite[eq. (1.1)]{WR},
\begin{equation}
  \label{WRa} U (\eta_0) = V(\eta_0) - |\eta_0| |B_r|,
\end{equation}
where $|B_r|$ is the volume of $B_r$ and $V(\eta_0)$ is the volume
of $\cup_{x\in \eta_0} B_r(x)$. The relationship between the single-
and the two-component versions was analyzed in detain in \cite{WR},
see also \cite{KK} where the interaction of the Curie-Weiss type (in
place of the hard-core repulsion) was studied. A significant feature
of (\ref{WRa}) is that this interaction is super-stable in the sense
of \cite{Ruelle}, see \cite[eq. (1.2)]{WR}. For such interactions,
the states of thermal equilibrium (Gibbs states) have correlation
functions that satisfy (\ref{no9}), see \cite{Ruelle}, which means
that the Gibbs states  are sub-Poissonian. This is one more argument
in favor of using such states. Note that our assumption (\ref{C7a})
covers the case of hard core repulsion mentioned above. In
\cite{asia1}, the results of which we will use in the remaining part
of this work, the repulsion kernels $\phi_i$ were assumed bounded
and integrable, which is a stronger version of (\ref{C7a}) that does
not cover the hard core repulsion. However, the boundedness was used
there only in the part where the mesoscopic limit of the model was
studied. That is, the part of \cite{asia1} the results of which we
will use here remains valid if one assumes only (\ref{C7a}).

\section{The Evolution of Sub-Poissonian States}
As mentioned above, in \cite{asia1} there was constructed a map
$t\mapsto \mu_t \in \mathcal{P}_{\rm exp}$ which describes the
evolution of states of the model (\ref{L}). Here we show that this
map is the unique solution of the Fokker-Planck equation
(\ref{FPE}), which is then used in the proof of Theorem \ref{1tm}.
In this section, we outline the construction realized in
\cite{asia1} in the form adapted to the present context, which
includes also passing to states on the space of multiple
configurations $\Gamma^2_*$. This is possible since
$\mu(\breve{\Gamma}^2_*) = \mu({\Gamma}^2_*)=1$, that holds for all
$\mu\in \mathcal{P}_{\rm exp}$,  see Remark \ref{derk}.

The key idea of  \cite{asia1} may be described as follows. Since
each $\mu\in \mathcal{P}_{\rm exp}$ is fully characterized by its
correlation functions $k^{(m)}_\mu$, $m\in \mathds{N}_0^2$, see
Definition \ref{no1dfg} and (\ref{no8}), instead of solving
(\ref{FPE}) directly one can pass to the evolution equation for the
corresponding correlation functions defined in appropriate Banach
spaces. An addition task, however, will be to prove that its
solutions are correlation functions - an analog of the classical
moment problem in this setting.

\subsection{The evolution of correlation functions}

 For $m\in \mathds{N}_0^2$, let a symmetric $G^{(m)}$
be in $C_{\rm cs}(X^{m_0} \times X^{m_1})$, see (\ref{no15a}). As
above, $m=(0,0)$ corresponds to constant functions. Let
$G:=\{G^{(m)}\}_{m\in \mathds{N}_0^2}$ be a collection of such
functions. We equip the set of all such collections with the usual
(member-wise) linear operations and then write $G(\eta)=G^{(m)}({\bf
x}, {\bf y})$ for $\eta = (\eta_0, \eta_1)$, $\eta_0 =\{x_1, \dots ,
x_{m_0}\}$,  $\eta_1=\{y_1 , \dots , y_{m_1}\}$, and  $({\bf x},
{\bf y}) = (x_1, \dots , x_{m_0}; y_1 , \dots , y_{m_1})$, cf.
(\ref{no6}), (\ref{no8}). Each $\eta=(\eta_0,\eta_1)$ is a pair of
finite configurations, and thus $\eta \in {\Gamma}^2$. That is,
$\eta_i$ is a finite configuration of particles of type $i=0,1$; by
$\Gamma_0$ we denote the subset of $\Gamma$ consisting of all finite
(possibly multiple) configurations. Let $\mathcal{G}_{\rm fin}$
denote the set of all aforementioned collections $G$ verifying
$G^{(m)} \equiv 0$ for all $m_0+m_1=:|m|> N_G$ for some $N_G\in
\mathds{N}$. Then the map $K$ as in (\ref{no15a}) can be defined on
$\mathcal{G}_{\rm fin}$ by the formula
\begin{gather}
\label{Qo}
 (KG) (\gamma) = \sum_{\eta \subset \gamma} G(\eta)
 = \sum_{\eta_0 \subset {\gamma_0}}\sum_{\eta_1 \subset
 {\gamma_1}} G(\eta_0, \eta_1)
 \\[.2cm] \nonumber  =   \sum_{m_0=0}^{\infty}\sum_{m_1=0}^{\infty}
 \frac{1}{m_0! m_1!} \sum_{(\mathbf{x},\mathbf{y})\in \gamma}
 G^{(m)}(\mathbf{x},\mathbf{y}).
\end{gather}
For $\mu \in \mathcal{P}_{\rm exp}$ and $G\in \mathcal{G}_{\rm
fin}$, by (\ref{15b}) $KG$ is $\mu$-integrable and the following
holds
\begin{eqnarray}
  \label{W6}
  \mu(KG) & = & \sum_{m\in \mathds{N}_0^2}  \frac{1}{m_0! m_1!}  \langle \! \langle
  k^{(m)}_\mu, G^{(m)} \rangle\!\rangle =:  \langle \! \langle
  k_\mu, G \rangle\!\rangle \\[.2cm] \nonumber & =: & \int_{\Gamma_0} \int_{\Gamma_0}
k_\mu(\eta_0,\eta_1) G(\eta_0,\eta_1) \lambda (d \eta_0) \lambda (d
\eta_1).
\end{eqnarray}
Here $k_\mu$ is the collection of the correlation functions
$k^{(m)}_\mu$, $m\in \mathds{N}_0^2$, that can also be considered as
a function $k_\mu :\Gamma_0^2 \to \mathds{R}$ such that
\begin{equation}
  \label{W6a}
  k_\mu (\eta) =   k_\mu (\eta_0 , \eta_1)= k^{(m)}(\mathbf{x}, \mathbf{y}), \quad  \eta_0 =\{x_1, \dots , x_{m_0}\}, \ \ \eta_1 =\{y_1, \dots ,
  y_{m_1}\}.
\end{equation}
The integrals in (\ref{W6}) are understood in the following way, cf.
(\ref{no15}),
\begin{gather}
  \label{W6b}
\int_{\Gamma_0} \int_{\Gamma_0}
  k_\mu(\eta_0,\eta_1) G(\eta_0,\eta_1) \lambda (d \eta_0 ) \lambda(d \eta_1)\\[.2cm] \nonumber = \sum_{m\in \mathds{N}_0^2}
  \frac{1}{m_0! m_1!} \int_{X^{m_0}\times X^{m_0}} k^{(m)}_\mu (\mathbf{x},
  \mathbf{y}) G^{(m)} (\mathbf{x}, \mathbf{y}) d^{m_0}\mathbf{x} d^{m_1}
  \mathbf{y}.
\end{gather}
In (\ref{W6a})  and (\ref{W6b}), $k_\mu$  -- similarly as $G$ in
(\ref{W6}) -- is the collection of symmetric $k^{(m)}_\mu \in
L^{\infty}(X^{m_0}\times X^{m_1})$, considered as an element of the
corresponding real linear space, which we denote by $\mathcal{K}$.
Keeping in mind that we deal with $\mu(LF)=\mu(LKG)$, see
(\ref{FPE}), assume that we are given ${L}^\Delta$ such that
\begin{equation}
  \label{W7}
\mu(  L KG) = \langle \! \langle
 L^\Delta k_\mu, G \rangle\!\rangle.
\end{equation}
This ${L}^\Delta$ can be calculated explicitly, see \cite[eq.
(2.23)]{asia1}. To present it here, we define
\begin{equation}
  \label{W8}
  \tau_x^i (y) = e^{-\phi_i(x-y)}, \qquad t_x^i (y) = \tau_x^i (y)
  -1, \qquad x,y \in X, \quad \ i=0,1,
\end{equation}
and
\begin{eqnarray}
  \label{W9}
  (\Upsilon_y^0 k)(\eta_0,\eta_1) & = & \int_{\Gamma_0} k(\eta_0, \eta_1 \cup
  \xi) e(t_y^0; \xi) \lambda (d\xi), \\[.2cm] \nonumber (\Upsilon_y^1 k)(\eta_0,\eta_1) & = & \int_{\Gamma_0} k(\eta_0\cup \xi, \eta_1) e(t_y^1; \xi)
   \lambda (d\xi),
\end{eqnarray}
where, for an appropriate $\theta :X \to \mathds{R}$ and $\xi\in
\Gamma_0$, we write
\begin{equation*}
  e(\theta;\xi) = \prod_{x\in \xi} \theta(x).
\end{equation*}
The expressions in (\ref{W9}) are to be understood in the following
way. For a given $m\in \mathds{N}_0^2$, one sets
\begin{eqnarray}
  \label{104}
& & (\Upsilon_y^0 k)^{(m)}(\mathbf{x}, \mathbf{y}) = k^{(m)}
(\mathbf{x}, \mathbf{y}) \\[.2cm] \nonumber & & \ \  +
\sum_{n=1}^\infty\frac{1}{n!}  \int_{X^n} k^{(m_0, m_1+n)} (x_1 ,
\dots , x_{m_0}; y_1, \dots , y_{m_1}, z_1 , \dots , z_n)
\prod_{j=1}^n t^0_y(z_j) d z_1 \cdots d z_n.
\end{eqnarray}
The convergence of the series and the integrals will be shown below.
In the same way, one defines also the second line of (\ref{W9}). Now
the operator satisfying (\ref{W7}) presents in the following form
\begin{eqnarray}
  \label{L1}
  (L^\Delta k)(\eta_0 , \eta_1) & = & \sum_{y\in \eta_0} \int_X a_0 (x-y) e(\tau^0_y;
  \eta_1)(\Upsilon^0_y k) (\eta_0 \setminus y \cup x, \eta_1)  dx
  \\[.2cm] \nonumber & - & \sum_{x\in \eta_0} \int_X a_0 (x-y) e(\tau^0_y;
  \eta_1) (\Upsilon_y^0 k)(\eta_0, \eta_1) dy   \\[.2cm] \nonumber &
  +
  & \sum_{y\in \eta_1} \int_X a_1 (x-y) e(\tau^1_y;
  \eta_0)(\Upsilon^1_y k) (\eta_0, \eta_1 \setminus y \cup x)  dx
  \\[.2cm] \nonumber & - & \sum_{x\in \eta_1} \int_X a_1 (x-y) e(\tau^1_y;
  \eta_0) (\Upsilon_y^1 k)(\eta_0, \eta_1) dy.
\end{eqnarray}
For $\vartheta\in \mathds{R}$ and $k\in \mathcal{K}$, see
(\ref{W6a}), we set
\begin{eqnarray}
  \label{L5}
  \|k\|_\vartheta & = &\esssup_{\xi_0, \xi_1 \in \Gamma_0} |k(\xi_0, \xi_1)|
  \exp\bigg{(}
  - \vartheta (|\xi_0|+|\xi_1|)\bigg{)}\\[.2cm] \nonumber &=& \sup_{m\in
  \mathds{N}_0^2} e^{- \vartheta (m_0 + m_1)}
 \left( \esssup_{(\mathbf{x}, \mathbf{y})\in X^{m_0}\times X^{m_1}}  |k^{(m)} (\mathbf{x},
 \mathbf{y})|\right),
\end{eqnarray}
and then introduce
\begin{equation}
  \label{L5a}
  \mathcal{K}_\vartheta = \{k\in \mathcal{K}: \|k\|_\vartheta<
  \infty \}, \qquad \vartheta \in \mathds{R},
\end{equation}
which is a real Banach space of weighted $L^\infty$-type. By
(\ref{L5}) one readily gets that $\|k\|_{\vartheta'} \leq
\|k\|_\vartheta$ whenever $\vartheta'>\vartheta$, which yields
\begin{equation}
  \label{L5b}
  \mathcal{K}_\vartheta \hookrightarrow \mathcal{K}_{\vartheta'},
  \qquad \vartheta'>\vartheta,
\end{equation}
where $\hookrightarrow$ denotes continuous embedding.

Let us turn now to the following issue. Given $k\in \mathcal{K}$,
under which conditions is this $k$ the correlation function for some
$\mu \in \mathcal{P}(\Gamma^2)$? By (\ref{L5}) and Definition
\ref{no1dfg} one concludes,  that $k_\mu \in \mathcal{K}_\vartheta$
with $\vartheta = \log\varkappa$ for $\mu\in \mathcal{P}_{\rm exp}$,
where $\varkappa$ is the type of $\mu$. At the same time, if $G\in
\mathcal{G}_{\rm fin}$ is such that $(KG) (\gamma) \geq 0$, by
(\ref{no9}) and (\ref{W6}) it follows that $\langle \! \langle k_\mu
, G \rangle \!\rangle \geq 0$. Set $\mathcal{G}_{\rm fin}^\star =\{
G\in \mathcal{G}_{\rm fin}: (KG)(\gamma)\geq 0, \ \gamma\in
\Gamma^2\}$, and also
\begin{equation}
  \label{L6A}
\mathcal{K}^\star = \{ k \in \mathcal{K}: k^{(0,0)} = 1 \ {\rm and}
\ \langle \! \langle k , G \rangle \!\rangle \geq 0 \ \forall G\in
\mathcal{G}_{\rm fin}^\star\}, \qquad \mathcal{K}^\star_\vartheta =
\mathcal{K}^\star \cap \mathcal{K}_\vartheta, \ \ \vartheta \in
\mathds{R}.
\end{equation}
It is known \cite[Proposition 2.2]{asia1}, see also \cite{Lenard}
for a more comprehensive discussion, that each $k\in
\mathcal{K}^\star_\vartheta$ is the correlation function of a unique
$\mu\in \mathcal{P}_{\rm exp}$ the type of which does not exceed
$e^\vartheta$. That is, $k\in \mathcal{K}_\vartheta$ is the
correlation function of a unique sub-Poissonian state $\mu$ if and
only if $k\in \mathcal{K}^{\star}_\vartheta$.
\begin{proposition}
  \label{BJULYpn}
Let $k\in \mathcal{K}_\vartheta^\star$. Then $\|k\|_{\vartheta'} =1$
for each $\vartheta'>\vartheta$.
\end{proposition}
\begin{proof}
Firstly, we note that $\|k\|_\vartheta \geq 1$ for each
$\vartheta\in \mathds{R}$ since
$k(\varnothing,\varnothing)=k^{(0,0)} =1$, see (\ref{JULY}). By
(\ref{no9}) and the fact that $k\in \mathcal{K}_\vartheta$, it
follows that $\|k^{(m_1,m_2)}\|_{L^\infty} \leq e^{\vartheta(m_1 +
m_2)}$, which yields the proof.
\end{proof}

By (\ref{L1}) and (\ref{L5}) for $\vartheta \in \mathds{R}$  we then
have, see \cite[eq. (3.10)]{asia1} for more detail,
\begin{gather}
  \label{L7}
|(L^\Delta k)(\eta_0, \eta_1)| \leq 4 \alpha \|k\|_\vartheta
e^{\vartheta (|\eta_0|+ |\eta_1|)} \bigg{(} |\eta_0| + |\eta_1|
\bigg{)}\exp\left( \varphi e^{\vartheta} \right),
\end{gather}
where $\varphi$ is as in (\ref{C7a}) and
\begin{equation}
  \label{L2}
  \alpha := \max_{i=0,1} \bar{a}_i^{(0)},
\end{equation}
see (\ref{C8}). This estimate settles the convergence issue in
(\ref{104}). It also implies
\begin{equation}
  \label{L8}
\|L^\Delta k\|_{\vartheta'} \leq
\frac{4\alpha\|k\|_\vartheta}{e(\vartheta'-\vartheta)}
\exp\left(\varphi e^\vartheta \right), \quad \vartheta ' >
\vartheta,
\end{equation}
which allows one to define the corresponding bounded linear
operators acting from $\mathcal{K}_{\vartheta}$ to
$\mathcal{K}_{\vartheta'}$. Along with them, we define an unbounded
linear operator $L^\Delta_{\vartheta'}$, $\vartheta'\in \mathds{R}$,
which acts in $\mathcal{K}_{\vartheta'}$ according to (\ref{L1})
with domain
\begin{equation}
  \label{L5c}
\mathcal{D}(L^\Delta_{\vartheta'}) =\{ k\in \mathcal{K}: L^\Delta k
\in \mathcal{L}_{\vartheta'}\}.
\end{equation}
By (\ref{L8}) one concludes that
\begin{equation}
  \label{L5d}
  \mathcal{K}_\vartheta \subset \mathcal{D}(L^\Delta_{\vartheta'}), \qquad \vartheta < \vartheta'.
\end{equation}
Now we fix $\vartheta\in \mathds{R}$ and consider the following
Cauchy problem in the Banach space $\mathcal{K}_\vartheta$
\begin{equation}
  \label{W11}
  \frac{d}{dt} k_t = L^\Delta_\vartheta k_t, \qquad k_t|_{t=0}= k_0.
\end{equation}
\begin{definition}
  \label{Wdf}
By a solution of (\ref{W11}) on the time interval, $[0, T )$, $T>0$,
we mean a continuous map $[0, T ) \ni t \mapsto k_t \in
\mathcal{D}(L^\Delta_\vartheta)\subset \mathcal{K}_\vartheta$ such
that the map $[0, T ) \ni t \mapsto d k_t/d t \in
\mathcal{K}_\vartheta$ is also continuous and both equalities in
(\ref{W11}) are verified.
\end{definition}
In view of the complex structure of (\ref{L1}), as well as of the
fact that $\mathcal{K}_\vartheta$ is a weighted $L^\infty$-type
Banach space, it is barely possible to solve (\ref{W11}) with all
$k_0\in \mathcal{D}(L^\Delta_\vartheta)$, e.g.,  by employing
$C_0$-semigroup techniques. In \cite{asia1}, the solution was
constructed for $k_0$ taken from $\mathcal{K}_{\vartheta_0}$ with
$\vartheta_0< \vartheta$, see (\ref{L5d}). Its characteristic
feature is that $k_t$ lies in some $t$-dependent
$\mathcal{K}_{\vartheta'}$ such that, cf. (\ref{L5b}),
\begin{equation}
  \label{scale}
\mathcal{K}_{\vartheta_0} \hookrightarrow \mathcal{K}_{\vartheta'}
\hookrightarrow \mathcal{K}_{\vartheta}.
\end{equation}
More precisely, the main result of \cite{asia1}  can be formulated
as follows, see Theorem 3.5 \emph{ibid}.
\begin{proposition}
  \label{W1pn}
For each $\mu \in \mathcal{P}_{\rm exp}$ and $T>0$, the Cauchy
problem in (\ref{W11}) with $\vartheta=\vartheta(T):=\log \varkappa
+ \alpha T$ has a unique solution $k_t \in
\mathcal{K}_{\vartheta(T)}^\star$, where $\varkappa$ is the type of
$\mu$ and $\alpha$ is as in (\ref{L2}).
\end{proposition}
\begin{remark}
  \label{W1rk}
The proof of Proposition \ref{W1pn} is performed in the following
three steps. First one shows that the Cauchy problem (\ref{W11})
with $k_0 \in \mathcal{K}_{\vartheta_0}$, $\vartheta_0 < \vartheta$,
has a unique local solution $k_t \in \mathcal{K}_{\vartheta}$, see
(\ref{scale}), i.e., existing for $t\in[0,T(\vartheta,\vartheta_0))$
with
\begin{equation}
  \label{W16}
  T(\vartheta, \vartheta_0) = \frac{\vartheta-\vartheta_0}{4\alpha}
  \exp\left( - \varphi e^{\vartheta}\right).
\end{equation}
The next (and the hardest) step is showing that, given $k_0 \in
\mathcal{K}^\star$, the solution $k_t$ lies in $\mathcal{K}^\star$
and hence is the correlation function of a unique $\mu_t\in
\mathcal{P}_{\rm exp}$. Finally, by means of the positivity as in
(\ref{L6A}) one makes continuation of the local solution $k_t$ to
all $t>0$ in such a way that $k_t\in \mathcal{K}_{\vartheta(t)}$
with $\vartheta(t) = \vartheta_0 +\alpha t$, cf. (\ref{L5d}).
\end{remark}

\subsection{The predual evolution}
Along with the evolution $t \mapsto k_t$ described in Proposition
\ref{W1pn} we will need the following one. Assume that we are given
$\widehat{L}$ such that, cf. (\ref{W7}),
\begin{equation}
  \label{W12}
  \langle \! \langle L^\Delta k, G\rangle \! \rangle = \langle \! \langle  k, \widehat{L} G\rangle \!
  \rangle,
\end{equation}
holding for all appropriate $k\in \mathcal{K}$ and $G\in
\mathcal{G}_{\rm fin}$. This operator can be derived similarly as
$L^\Delta$ given in (\ref{L1}). It has the following form
\begin{eqnarray}
  \label{W13}
  & &  (\widehat{L} G) (\eta_0, \eta_1)  \\[.2cm] \nonumber& & =  \sum_{x\in \eta_0} \int_X
  \sum_{\xi \subset \eta_1} e(\tau^0_y ; \eta_1 \setminus \xi)
  e(t^0_y;\xi) \left[ G(\eta_0 \setminus x \cup y, \eta_1\setminus \xi) - G(\eta_0,  \eta_1\setminus \xi)\right] d
  y \qquad \\[.2cm] \nonumber & & +  \sum_{x\in \eta_1} \int_X
  \sum_{\xi \subset \eta_0} e(\tau^1_y ; \eta_0 \setminus \xi)
  e(t^1_y;\xi) \left[ G(\eta_0 \setminus \xi,  \eta_1\setminus x \cup y) - G(\eta_0\setminus \xi,  \eta_1)\right] d
  y,
\end{eqnarray}
with $t^i_y$ and $\tau_y^i$ given in (\ref{W8}). Obviously,
$\widehat{L}$ is defined for each $G\in \mathcal{G}_{\rm fin}$. Our
aim now is to extend it to $G$ taken from the spaces predual to
those defined in (\ref{L5a}). To this end, we introduce the norm
\begin{equation}
  \label{L6}
  |G|_\vartheta= \int_{\Gamma_0}\int_{\Gamma_0} |G(\xi_0, \xi_1)| \exp\bigg{(}
  \vartheta (|\xi_0|+|\xi_1|)\bigg{)}\lambda ( d \xi_0) \lambda ( d
  \xi_1),
\end{equation}
and then define, cf. (\ref{L5a}),
\begin{equation*}
  \mathcal{G}_\vartheta =\{ G: |G|_\vartheta <\infty\}.
\end{equation*}
Thus, each $\mathcal{G}_\vartheta$ is a weighted $L^1$-type Banach
space. Noteworthy, cf. (\ref{L5b}),
\begin{equation}
  \label{L6b}
  \mathcal{G}_{\vartheta'} \hookrightarrow \mathcal{G}_\vartheta,
  \qquad \vartheta < \vartheta'.
\end{equation}
By employing (\ref{W13}), similarly as in (\ref{L8}) we get
\begin{equation}
  \label{L9}
  |\widehat{L} G|_{\vartheta} \leq \frac{4\alpha
|G|_{\vartheta'}}{e(\vartheta'-\vartheta)} \exp\left(\varphi
e^\vartheta \right), \quad \vartheta ' > \vartheta.
\end{equation}
The latter formula allows one to define (by induction in $n$) the
iterations of $\widehat{L}$, cf. (\ref{L6b}),
\[
(\widehat{L})^n_{\vartheta \vartheta'} : \mathcal{G}_{\vartheta'}
\to \mathcal{G}_\vartheta, \qquad n\in \mathds{N},
\]
the operator norms of which obey
\begin{equation}
  \label{W14}
  \| (\widehat{L})^n_{\vartheta \vartheta'}\| \leq n^n \left(
 \frac{4\alpha}{\vartheta'-\vartheta}  \right)^n \exp\left( n (\varphi e^{\vartheta'}
 -1)\right).
\end{equation}
Then we introduce the operators
\begin{equation}
  \label{W15}
  \varSigma_{\vartheta \vartheta'}(t) = 1 + \sum_{n=1}^\infty
  \frac{t^n}{n!} (\widehat{L})^n_{\vartheta \vartheta'},
\end{equation}
where the series converges in the norm of the Banach space
$\mathcal{L}( \mathcal{G}_{\vartheta'}, \mathcal{G}_\vartheta)$ of
bounded linear operators acting from $ \mathcal{G}_{\vartheta'}$ to
$ \mathcal{G}_{\vartheta}$ -- uniformly on compact subsets of
$[0,T(\vartheta', \vartheta))$, with $T(\vartheta', \vartheta)$
defined in (\ref{W16}). The latter fact readily follows by
(\ref{W14}). Then, for $t< T(\vartheta', \vartheta)$, we can set
\begin{equation}
  \label{W17}
  G_t = \varSigma_{\vartheta \vartheta'}(t) G, \qquad G\in
  \mathcal{G}_{\vartheta'}.
\end{equation}
For a given $\vartheta$, let $k\in \mathcal{K}_\vartheta$ be the
correlation function of a certain $\mu\in \mathcal{P}_{\rm exp}$.
According to Proposition \ref{W1pn}, see also Remark \ref{W1rk},
there exist the map $t\mapsto k_t$, $k_0=k$, that solves (\ref{W11})
and is such that $k_t \in \mathcal{K}_{\vartheta(t)}^\star$ with
$\vartheta (t) = \vartheta +\alpha t$. Let $\vartheta$ and
$\vartheta'$ be as in (\ref{W17}). For $t< T(\vartheta',
\vartheta)$, by (\ref{W16}) it follows that
\[
\vartheta(t) <\vartheta+ \alpha T(\vartheta',\vartheta) <
\vartheta',
\]
which means that $\mathcal{K}_{\vartheta(t)} \subset
\mathcal{K}_{\vartheta'}$. The continuation mentioned in Remark
\ref{W1rk} was done in \cite{asia1} by showing that the solution - a
priori lying in $\mathcal{K}_{\vartheta'}$ -- is in fact in
$\mathcal{K}_{\vartheta(t)}$. For $t< T(\vartheta',\vartheta)$, it
can be obtained similarly as in (\ref{W17}). By induction in $n$ one
defines bounded operators $(L^\Delta)^n_{\vartheta'\vartheta}:
\mathcal{K}_{\vartheta}\to \mathcal{K}_{\vartheta'}$, $n\geq 2$, the
norms of which are estimated as in (\ref{W14}). Then one sets
\begin{equation}
  \label{W18}
k_t =  \varXi_{\vartheta'\vartheta}(t) k_0, \qquad
\varXi_{\vartheta'\vartheta}(t) =1+ \sum_{n=1}^\infty
  \frac{t^n}{n!} (L^\Delta)^n_{\vartheta'\vartheta}, \qquad t<
  T(\vartheta', \vartheta).
\end{equation}
For each $t<
  T(\vartheta', \vartheta)$, one finds $\vartheta'' \in (\vartheta,
\vartheta')$ such that $t<
  T(\vartheta'', \vartheta)$, see (\ref{W16}), which means that
$\varXi_{\vartheta'\vartheta}(t)$ maps $\mathcal{K}_\vartheta$ to
$\mathcal{D}(L^\Delta_{\vartheta'})$, see (\ref{L5b}), (\ref{L5c}).
Furthermore, by the absolute convergence of the series in
(\ref{W18}) -- in the norm of the Banach space $\mathcal{L}(
\mathcal{K}_{\vartheta}, \mathcal{K}_{\vartheta'})$ -- it follows
that the map $t\mapsto \varXi_{\vartheta'\vartheta}(t)$ is
continuously differentiable in this norm and the following holds
\begin{equation}
  \label{W18Z}
\frac{d}{dt} \varXi_{\vartheta'\vartheta}(t) =
\varXi_{\vartheta'\vartheta''}(t) L^\Delta_{\vartheta''\vartheta}=
L^\Delta_{\vartheta'\vartheta''} \varXi_{\vartheta''\vartheta}(t) =
L^\Delta_{\vartheta'} \varXi_{\vartheta'\vartheta}(t) , \qquad t<
T(\vartheta',\vartheta),
\end{equation}
which yields that $k_t$ as in (\ref{W18}) solves (\ref{W11}). Then
the steps mentioned in Remark \ref{W1rk} amount to the following.
For fixed $\vartheta, \vartheta'$, one constructs
$\varXi_{\vartheta'\vartheta}(t)$, $t<T(\vartheta', \vartheta)$, and
shows by (\ref{W18Z}) that $k_t$ as in (\ref{W18}) solves the
corresponding Cauchy problem. Then one takes $k_0 =k_{\mu}\in
\mathcal{K}^\star_{\vartheta_0}$ and shows that $k_t$,
$t<T(\vartheta, \vartheta_0)$ for some $\vartheta>\vartheta_0$,
obtained as just mentioned, lies in $\mathcal{K}^\star$. Finally, by
the positivity as in (\ref{L6A}) one proves that this $k_t$ lies in
$\mathcal{K}_{\vartheta(t)}^\star$, $\vartheta(t) = \vartheta_0 +
\alpha t < \vartheta$ for $t < T(\vartheta, \vartheta_0)$. The
continuation to  $s>t $ is then performed by applying
$\varXi_{\vartheta\vartheta(t)}(s)$ to $k_t$, see \cite[Lemma
5.5]{asia1} fore more detail.

Complementary information concerning the operator norms of the maps
$t\mapsto \varSigma_{\vartheta\vartheta'}(t)$ and $t\mapsto
\varXi_{\vartheta'\vartheta}(t)$ is given by the following estimates
\begin{equation}
  \label{W18z}
\|\varSigma_{\vartheta\vartheta'}(t)\| \leq
\frac{T(\vartheta',\vartheta)}{T(\vartheta',\vartheta) - t}, \qquad
|\varXi_{\vartheta'\vartheta}(t)| \leq
\frac{T(\vartheta',\vartheta)}{T(\vartheta',\vartheta) - t}, \qquad
t< T(\vartheta',\vartheta),
\end{equation}
which readily follow by (\ref{W14}) and the corresponding estimate
of $(L^\Delta)^n_{\vartheta'\vartheta}$, respectively. By means of
(\ref{W18Z}) and (\ref{W18}) we also obtain the following.
\begin{proposition}
  \label{W2pn}
Given $\vartheta$ and $\vartheta'> \vartheta$, let $k_0\in
\mathcal{K}_\vartheta$ be the correlation function of a certain
$\mu\in \mathcal{P}_{\rm exp}$ and then $k_t$ be the solution as in
Proposition \ref{W1pn}. Let also $G$ be in
$\mathcal{G}_{\vartheta'}$. Then for each $t< T(\vartheta',
\vartheta)$ the following holds
\begin{equation}
  \label{W17a}
\langle\! \langle k_t , G \rangle \! \rangle = \langle\! \langle k_0
, G_t \rangle \! \rangle,
\end{equation}
where $G_t$ is as in (\ref{W17}).
\end{proposition}
We end up this section by producing appropriate extensions of the
map $G \mapsto KG$ defined in (\ref{Qo}) for $G\in \mathcal{G}_{\rm
fin}$. Set $\mathcal{G}_{\infty}=\cap_{\vartheta \in \mathds{R}}
\mathcal{G}_\vartheta$. As is usual, we do not distinguish between
the elements of $\mathcal{G}_{\infty}$ and the measurable functions
$G:\Gamma_0^2 \to \mathds{R}$ for which the integrals in the
right-hand side of (\ref{L6}) are finite for all $\vartheta\in
\mathds{R}$. Then $\mathcal{G}_{\rm fin}\subset
\mathcal{G}_{\infty}$. We recall that each measurable $G:\Gamma_0^2
\to \mathds{R}$ is a collection $\{G^{(m)}\}_{m\in \mathds{N}^2_0}$
of symmetric (cf. (\ref{nosym})) Borel functions. Similarly as in
\cite[Theorem 1]{Lenard}, one can show that, for each such $G$ and
$m\in \mathds{N}_0^2$,  the map
$$\Gamma^2_* \ni \gamma \mapsto \sum_{(\mathbf{x}, \mathbf{y})\in
\gamma} G^{(m)} (\mathbf{x}, \mathbf{y})$$ is
$\mathcal{B}(\Gamma^2_*)$-measurable. Then also the functions
(possibly taking infinite values)
\[
F_{G}(\gamma) := \sum_{m\in \mathds{N}_0^2 } \frac{1}{m_0! m_1!}
\sum_{(\mathbf{x}, \mathbf{y})\in \gamma} |G^{(m)}(\mathbf{x},
\mathbf{y})|, \qquad G\in \mathcal{G}_{\infty}, \ \ \gamma\in
\Gamma^2_*,
\]
enjoy this property; hence,  the sets
\begin{gather*}
\Gamma^2_G = \bigcup_{n\in \mathds{N}}\{\gamma\in \Gamma^2_*:
F_{G}(\gamma)\leq n\}, \quad G\in \mathcal{G}_{\infty},
\end{gather*}
are $\mathcal{B}(\Gamma^2_*)$-measurable. Moreover, $\Gamma^2_*$
itself is $\Gamma^2_{G_\psi}$ for $G_\psi$ such that
$G^{(1,0)}_\psi(x) = G^{(0,1)}_\psi(x) = \psi(x)$ and $G^{(m)}_\psi
= 0$ whenever $|m|\neq 1$, see (\ref{no21}). Let $\mu\in
\mathcal{P}_{\rm exp}$ be of type $e^{\vartheta}$ for some
$\vartheta\in \mathds{R}$. By (\ref{W6}) we thus have
\begin{equation*}
  \mu(F_{G}) \leq |G|_\vartheta.
\end{equation*}
Similarly as in (\ref{no22}) and (\ref{C4}) we then get that
$\mu(\Gamma^2_G) =1$. Therefore, for each $\mu\in \mathcal{P}_{\rm
exp}$ and $G\in \mathcal{G}_\infty$, the series in
\begin{equation}
  \label{Bie2}
  (KG) (\gamma):= \sum_{m\in \mathds{N}_0^2 } \frac{1}{m_0! m_1!} \sum_{(\mathbf{x}, \mathbf{y})\in
\gamma} G^{(m)} (\mathbf{x}, \mathbf{y})
\end{equation}
absolutely converges, $\mu$-almost everywhere on $\Gamma^2_*$. This
includes also the Poisson measures $\pi_\kappa$ with all $\kappa>0$,
see (\ref{no21A}). By means of these argument we obtain the
following conclusion.
\begin{proposition}
  \label{Bi1pn}
Let $\mu\in \mathcal{P}_{\rm exp}$ be of type $e^\vartheta$ for some
$\vartheta\in \mathds{R}$. Then the map $G \mapsto KG$ as in
(\ref{Bie2}) gives rise to the bounded linear operator $K$ acting
from the Banach space $\mathcal{G}_\vartheta$ to the Banach space
$L^1 (\Gamma^2_*, \mu)$, such that
\begin{equation*}
   \mu(K G) = \langle\!\langle k_\mu , G\rangle \! \rangle.
\end{equation*}
Moreover, if $G$ belongs to $\mathcal{G}_{\vartheta'}$ for some
$\vartheta'>\vartheta$, then
$$\mu(L K G) = \mu(K \widehat{L} G) = \langle \! \langle k_\mu ,
\widehat{L} G \rangle \! \rangle,$$ where
$\widehat{L}:\mathcal{G}_{\vartheta'} \to \mathcal{G}_\vartheta$ is
the linear operator defined in (\ref{L9}).
\end{proposition}

\section{Uniqueness}

\subsection{Solving the Fokker-Planck equation}

We begin by recalling Definition \ref{A01df}, in which we mention
maps $t\mapsto \mu_t \in \mathcal{P} (\Gamma^2_*)$.
\begin{lemma}
  \label{W1lm}
Let $\mu_0\in  \mathcal{P}_{\rm exp}$ be of type $\varkappa_0 =
e^{\vartheta_0}$ and consider the Fokker-Planck equation (\ref{FPE})
with the initial condition $\mu_t|_{t=0}=\mu_0$ and all choices of
$F \in \widehat{\mathcal{F}}$, see (\ref{de2}) and Definition
\ref{THdf}. Assume that $t\mapsto \mu_t$ is a solution of
(\ref{FPE}) with such $\mu_0$ and $F$. Then $\mu_t \in
\mathcal{P}_{\rm exp}$; moreover, for each $T>0$, there exists
$\vartheta_T>\vartheta_0$ such that the type of $\mu_t$ does not
exceed $e^{\vartheta_T}$ for all $t\leq T$.
\end{lemma}
Note that in this lemma we assume that only $\mu_0$ is
sub-Poissonian, and that $t\mapsto \mu_t$ solves (\ref{FPE}) only
with a part of $\mathcal{D}(L)$. Before proceeding further, we
recall that the families of functions $\widetilde{\mathcal{F}}$ and
$\widehat{\mathcal{F}}$ were introduced in (\ref{T4a}) and
(\ref{de2}), respectively.
\begin{proposition}
  \label{V1pn}
Set $\mathcal{F}_{\infty}=\{F=KG: G\in \mathcal{G}_\infty\}$, see
(\ref{Bie2}). Then
 both $\widetilde{\mathcal{F}}$ and
$\widehat{\mathcal{F}}$ are subsets of $\mathcal{F}_{\infty}$.
\end{proposition}
\begin{proof}
By (\ref{T4}) and then by (\ref{Qo}) one readily gets that
\begin{eqnarray}
  \label{V}
\widetilde{F}^\theta_\tau (\gamma) &= &\left(\sum_{m_0=0}^{\infty }
\sum_{\{x_1, \dots , x_{m_0}\}\subset \gamma_0}
\widetilde{G}^{(m_0)}_{\tau_0,\theta_0}(x_1 , \dots , x_{m_0})
\right) \left(\sum_{m_1=0}^{\infty } \sum_{\{y_1, \dots ,
y_{m_1}\}\subset \gamma_1}
\widetilde{G}^{(m_1)}_{\tau_1,\theta_1}(y_1 , \dots ,
y_{m_1}) \right) \nonumber \\[.2cm]  & =& (K
\widetilde{G}_{\tau,\theta})(\gamma), \qquad
\widetilde{G}_{\tau,\theta}(\eta_0, \eta_1) =
\widetilde{G}_{\tau_0,\theta_0}(\eta_0)\widetilde{G}_{\tau_1,\theta_1}(\eta_1),
\end{eqnarray}
where
\begin{equation}
  \label{V1}
\widetilde{G}^{(m_i)}_{\tau_i,\theta_i}(x_1 , \dots , x_{m_i})  =
\prod_{j=1}^{m_i} \theta^{\tau_i}_i(x_j), \quad \theta^{\tau_i}_i(x)
:= \theta_i(x) e^{-\tau_i \psi(x)} + e^{-\tau_i \psi(x)} -1, \ \
i=0,1.
\end{equation}
Clearly, $\theta^\tau_i \in L^1(X)$ for each $\tau\geq 0$,
$\theta_i\in \varTheta_\psi$, $i=0,1$. Hence,
$\widetilde{G}_{\tau,\theta} \in \mathcal{G}_\vartheta$ for any
$\vartheta\in \mathds{R}$, which yields
$\widetilde{\mathcal{F}}\subset \mathcal{F}_{\infty}$.

Now by the first line in (\ref{de2}) we have, see (\ref{no3a}),
\begin{eqnarray}
  \label{V2}
\widehat{F}^{m_i}_{\tau_i}(\mathbf{v}_i|\gamma_i) & = &
\sum_{\mathbf{x}^{m_i}\in \gamma_i} \mathbf{v}_i
(\mathbf{x}^{m_i})\prod_{x\in \gamma_i\setminus \mathbf{x}^{m_i}}
\left(1+ \varsigma_i(x)\right) \\[.2cm] \nonumber & = & \sum_{\eta_i
\subset \gamma_i} R^{m_i}(\mathbf{v}_i|\eta_i) \prod_{x\in
\gamma_i\setminus \mathbf{x}^{m_i}} \left(1+ \varsigma_i(x)\right),
\end{eqnarray}
where $\varsigma_i(x) = e^{-\tau_i \psi(x)} -1,$ and, see
(\ref{de1}),
\begin{equation}
  \label{V3}
R^{m_i}(\mathbf{v}_i|\eta_i) = \left\{\begin{array}{ll} \sum_{\sigma
\in S_{m_i}} v_{i,1}(x_{\sigma(1)}) \cdots
v_{i,_{m_i}}(x_{\sigma(m_i)}), \qquad \quad &{\rm if} \quad \eta_i=
\{x_1, \dots , x_{m_i}\}, \\[.3cm] 0, &{\rm otherwise}.
\end{array}
\right.
\end{equation}
Now we open the brackets in the product in (\ref{V2}) and get
\begin{eqnarray}
  \label{V4}
\widehat{F}^{m_i}_{\tau_i}(\mathbf{v}_i|\gamma_i) & = & \sum_{\eta_i
\subset \gamma_i} \widehat{G}^{(m_i)}_{\tau_i}
(\mathbf{v}_i|\eta_i), \\[.2cm] \nonumber  \widehat{G}^{(m_i)}_{\tau_i}
(\mathbf{v}_i|\eta_i) & := & \sum_{\xi_i \subset \eta_i}
R^{m_i}(\mathbf{v}_i|\xi_i) \prod_{x\in \eta_i \setminus \xi_i}
\varsigma_i(x).
\end{eqnarray}
To complete the proof we have to show the corresponding
integrability of $\widehat{G}^{(m_i)}_{\tau_i}
(\mathbf{v}_i|\cdot)$. Since $v_{i,j}\in \varTheta^{+}_\psi$ and
$\tau_i>0$, we have
\begin{equation*}
\left| \widehat{G}^{(m_i)}_{\tau_i} (\mathbf{v}_i|\eta_i) \right|
\leq \sum_{\xi_i \subset \eta_i} R^{m_i}(\mathbf{v}_i|\xi_i)
\prod_{x\in \eta_i \setminus \xi_i} \left[ \tau_i \psi(x) \right],
\end{equation*}
and hence
\begin{gather}
  \label{V6}
\int_{\Gamma_0} \left| \widehat{G}^{(m_i)}_{\tau_i}
(\mathbf{v}_i|\eta) \right| e^{\vartheta|\eta|} \lambda ( d \eta)
\leq \int_{\Gamma_0}\int_{\Gamma_0} e^{\vartheta|\xi|}
R^{m_i}(\mathbf{v}_i|\xi) e^{\vartheta |\eta|}  \prod_{x\in \eta}
\left[ \tau_i \psi(x) \right] \lambda ( d \xi) \lambda ( d
\eta) \\[.2cm] \nonumber = e^{m_i \vartheta} \langle v_{1,i}\rangle
\cdots \langle v_{m_i,i}\rangle \exp\left(\tau_i e^{\vartheta}
\langle \psi \rangle \right),
\end{gather}
where $\langle v_{j,i} \rangle$, $j=1, \dots , m_i$, and $\langle
\psi \rangle$ are the $L^1(X)$-norms of these functions. Similarly
as in (\ref{V}) we then have
\begin{equation}
  \label{V7}
\widehat{F}^m_\tau (\mathbf{v}|\gamma) = (K \widehat{G}^m_\tau
(\mathbf{v}|\cdot))(\gamma), \quad \widehat{G}^m_\tau
(\mathbf{v}|\eta) = \widehat{G}^{m_0}_{\tau_0} (\mathbf{v}_0|\eta_0)
\widehat{G}^{m_1}_{\tau_1} (\mathbf{v}_1|\eta_1),
\end{equation}
which completes the proof.
\end{proof}
\begin{lemma}
  \label{W2lm}
For each $\mu\in  \mathcal{P}_{\rm exp}$, the Fokker-Planck equation
(\ref{FPE}) with $\mu_0=\mu$ has exactly one solution.
\end{lemma}
\begin{proof}
\underline{Existence}: Let $t \mapsto k_t$ be as in Proposition
\ref{W1pn} with $k_0=k_{\mu}$. Since $k_t$ solves (\ref{W11}), it
follows that
\begin{equation}
  \label{W20}
k_{t_2} - k_{t_1} = \int_{t_1}^{t_2} L^\Delta_{\vartheta(T)} k_s d
s,
\end{equation}
holding for all $t_2 > t_1 \geq 0$ and $T>t_2$. Let $\mu_t\in
\mathcal{P}_{\rm exp}$ be the unique measure for which $k_t$ is the
correlation function, see Remark \ref{W1rk}. Then for each $G\in
\mathcal{G}_{\infty}$, we have
\[
\mu_{t_j} (K G) = \langle \! \langle k_{t_j} , G \rangle \! \rangle,
\qquad j=1,2,
\]
For each $\vartheta\in \mathds{R}$, the map $\mathcal{K}_\vartheta
\ni k \mapsto \langle \! \langle k , G \rangle \! \rangle$ is linear
and bounded -- hence continuous. Then by (\ref{W20}) and Proposition
\ref{Bi1pn} for $F=KG$ we get
\begin{eqnarray}
  \label{W21}
\mu_{t_2}(F) - \mu_{t_1}(F) & = & \langle \! \langle
\int_{t_1}^{t_2}L^\Delta_{\vartheta(T)} k_{s} ds , G \rangle \!
\rangle   =  \int_{t_1}^{t_2} \langle \! \langle
L^\Delta_{\vartheta(T)} k_{s} , G \rangle \! \rangle ds \\[.2cm] \nonumber &
= & \int_{t_1}^{t_2} \langle \! \langle  k_{s} , \widehat{L} G
\rangle \! \rangle ds  =  \int_{t_1}^{t_2} \mu_s (L F)ds.
\end{eqnarray}
Now we can take $G= \widehat{G}^m_\tau (\mathbf{v}|\cdot)$, see
(\ref{V7}), or $G= \widetilde{G}_{\tau,\theta}$, see (\ref{V}),
(\ref{V1}), and conclude that the map $t\mapsto \mu_t$ is a solution
of (\ref{FPE}) according to Definition \ref{A01df}.

\underline{Uniqueness}: Let $t\mapsto \tilde{\mu}_t$ be another
solution satisfying $\tilde{\mu}_t|_{t=0}=\mu$. Let  also
$\vartheta_0$ be such that $k_0 =k_\mu  \in
\mathcal{K}_{\vartheta_0}$. For a fixed $T>0$ and each $t\leq T$, by
Lemma \ref{W1lm} it follows that $\tilde{\mu}_t\in \mathcal{P}_{\rm
exp}$ and its type does not exceed $e^{\vartheta_T}$. That is, the
correlation function $\tilde{k}_t$ of this measure $\tilde{\mu}_t$
lies in $\mathcal{K}_{\vartheta_T}$. Without any harm we may take
$\vartheta_T$ big enough so that
\begin{equation}
  \label{JULY1}
  \sup_{s\in [0,T]} \|\tilde{k}_s\|_{\vartheta_T} = 1,
\end{equation}
see Proposition \ref{BJULYpn}, and also
 $\vartheta_T \geq \vartheta(T)=
\vartheta_0 + \alpha T$, see Proposition \ref{W1pn}.

It is known, see \cite[eqs. (4.6) -- (4.8)]{asia1}, that the map
$[\vartheta_T, +\infty) \ni \vartheta \mapsto T(\vartheta,
\vartheta_T)$, see (\ref{W16}), attains maximum $T_*(\vartheta_T)$
at $\tilde{\vartheta}_T = \vartheta_T + \delta(\vartheta_T)$, where
\begin{equation}
  \label{Ts}
T_*(\vartheta_T) = \frac{\delta(\vartheta_T)}{4\alpha} \exp\left( -
\frac{1}{\delta(\vartheta_T)}\right),
\end{equation}
and $\delta(\vartheta_T)$
 is the unique solution of
the equation
\[
\delta e^\delta = \exp\left( - \vartheta_T - \log \varphi \right).
\]
According to our assumption $\tilde{k}_t \in
\mathcal{K}_{\vartheta_T}\subset
\mathcal{D}(L^\Delta_{\tilde{\vartheta}_T})$, see  (\ref{L5d}), and
\begin{equation}
  \label{W22}
\tilde{\mu}_t (KG) - \mu(KG) = \langle\! \langle \tilde{k}_t - k_0 ,
G \rangle \! \rangle = \int_0^t \langle\! \langle
L^\Delta_{\tilde{\vartheta}_T} \tilde{k}_s, G \rangle\! \rangle ds,
\end{equation}
holding for all $t\leq T$ and $G$ such that $KG \in
\widetilde{\mathcal{F}} \cup \widehat{\mathcal{F}}$. That is,   $G$
is either $\widetilde{G}_{\tau,\theta}$ (\ref{V}) or
$\widehat{G}_{\tau}^m (\mathbf{v}|\cdot)$ (\ref{V7}).  The
integrations in (\ref{W22}) were interchanges for the same reasons
as in (\ref{W21}). Let us prove that (\ref{W22}) holds for all $G\in
\mathcal{G}_{\infty}$. By (\ref{L8}) and (\ref{JULY1}) we have
\begin{equation}
  \label{W23}
\|L^\Delta_{\tilde{\vartheta}_T} \tilde{k}_s
\|_{\tilde{\vartheta}_T} \leq 1/e T_*(\vartheta_T) ,
\end{equation}
holding for all $s\leq T$. Now we fix
\begin{equation}
  \label{Ts1}
t < \min\{  T; T_*(\vartheta_T)\},
\end{equation}
and set
\[
q = \tilde{k}_t - k_0 - \int_0^t L^\Delta_{\tilde{\vartheta}_T}
\tilde{k}_s ds.
\]
By Proposition \ref{BJULYpn}, and then by (\ref{JULY1}) and
(\ref{W23}), we get
\begin{equation}
  \label{W42}
\|q\|_{\tilde{\vartheta}_T} \leq 2 + t /e T_* (\vartheta_T).
\end{equation}
Then for $G=\widehat{G}_{\tau}^m (\mathbf{v}|\cdot)$ with
$\tau_i\leq 1$, $i=0,1$, by (\ref{W22}) it follows that $\langle \!
\langle q, G\rangle \! \rangle =0$. At the same time, by (\ref{V7}),
(\ref{V6}), (\ref{V4}) and (\ref{W42}), we have
\begin{eqnarray}
  \label{W24}
\left|\langle \! \langle q, G\rangle \! \rangle \right| & \leq &
\|q\|_{\tilde{\vartheta}_T} |G|_{\tilde{\vartheta}_T} \\[.2cm] \nonumber & \leq & \left(
2 + t /e T_*(\vartheta_T)\right) \exp\left(
(m_0+m_1)\tilde{\vartheta}_T + 2\langle\psi\rangle
e^{\tilde{\vartheta}_T}\right)\prod_{j_0=1}^{m_0}
\prod_{j_1=1}^{m_1} \langle v_{j,i} \rangle.
\end{eqnarray}
Let $G_\varepsilon$ denote $\widehat{G}_{\tau}^m (\mathbf{v}|\cdot)$
with $\tau_0 = \tau_1 =\varepsilon \leq 1$. Then by the dominated
convergence theorem and (\ref{W24}) we get
\begin{equation}
  \label{W26}
\langle \! \langle q, G_0\rangle \! \rangle = \lim_{\varepsilon \to
0} \langle \! \langle q, G_\varepsilon\rangle \! \rangle = 0,
\end{equation}
where $G_0$ is the pointwise limit of $G_\varepsilon$ as
$\varepsilon \to 0$. That is, see (\ref{V4}) and (\ref{V3}),
\begin{equation}
  \label{W27}
G_0(\eta) = \widehat{G}^{m_0}_0 (\mathbf{v}_0|\eta_0)
\widehat{G}^{m_1}_0 (\mathbf{v}_1|\eta_1) = R^{m_0}
(\mathbf{v}_0|\eta_0) R^{m_1} (\mathbf{v}_1|\eta_1).
\end{equation}
Now we use this $G_0$ in (\ref{W26}) and obtain, see (\ref{de1}),
\begin{eqnarray}
  \label{W28}
& &   \int_{X^{m_0} \times X^{m_1}} q^{(m)} (\mathbf{x}, \mathbf{y})
\mathbf{v}_0(\mathbf{x}) \mathbf{v}_1(\mathbf{y}) d^{m_0}
\mathbf{x} d^{m_1} \mathbf{y} \\[.2cm] \nonumber & &  =  \int_{X^{m_0} \times X^{m_1}}
q^{(m_0,m_1)} (x_1 , \dots , x_{m_0}; y_1, \dots , y_{m_1})
v_{0,1}(x_1) \cdots
v_{0,m_0}(x_{m_0}) \\[.2cm] \nonumber & & \times  v_{1,1}(y_1) \cdots
v_{1,m_1}(x_{m_1}) d x_1 \cdots dx_{m_0} dy_1 \cdots dy_{m_1} = 0,
\end{eqnarray}
holding for all $m=(m_0,m_1) \in \mathds{N}_0^2$ and $v_{i,j} \in
\varTheta_\psi^{+}$. For each $m \in \mathds{N}_0^2$, the set of
functions $(\mathbf{x},\mathbf{y})\mapsto
\mathbf{v}_0(\mathbf{x})\mathbf{v}_1(\mathbf{y})$, $v_{i,j} \in
\varTheta_\psi^{+}$, is closed with respect to the pointwise
multiplication and separates the points of $X^{m_0}\times X^{m_1}$.
Such functions vanish at infinity and are everywhere positive, see
(\ref{T1}). Then by the corresponding version of the of the
Stone-Weierstrass theorem \cite{dB}, the linear span of this set is
dense (in the supremum norm) in the algebra $C_0(X^{m_0}\times
X^{m_1})$ of continuous functions that vanish at infinity (recall
that $X=\mathds{R}^d$, hence $X^{m_0}\times X^{m_1}$ is locally
compact). At the same time, $C_0(X^{m_0}\times X^{m_1})\cap
L^1(X^{m_0}\times X^{m_1})$ is dense in $L^1(X^{m_0}\times X^{m_1})$
as its subset $C_{\rm cs}(X^{m_0}\times X^{m_1})$ has this property.
Thus,
\[
\langle \! \langle q^{(m)}, G^{(m)} \rangle \! \rangle =0,
\]
holding for all $G^{(m)}\in L^1(X^{m_0}\times X^{m_1})$. The
extension of the latter to
\[
\langle \! \langle q, G \rangle \! \rangle =0, \qquad {\rm for}
\qquad G\in {\mathcal{G}}_{\infty},
\]
is standard, which yields the validity of (\ref{W22}) for all such
$G$. By (\ref{W12}), (\ref{W21}) and (\ref{W22}) we have
\begin{gather}
  \label{W29b}
 \langle
\! \langle \tilde{k}_t, G \rangle \! \rangle =  \langle \! \langle
{k}_0, G \rangle \! \rangle + \int_0^t \langle \! \langle
L^{\Delta}_{\tilde{\vartheta}_T\vartheta_T}\tilde{k}_s, G \rangle \!
\rangle ds \\[.2cm] \nonumber =  \langle \! \langle
{k}_0, G \rangle \! \rangle + \int_0^t \langle \! \langle
\tilde{k}_s,\widehat{L}_{\vartheta_T\tilde{\vartheta}_T } G \rangle
\! \rangle ds  \qquad G\in \mathcal{G}_\infty.
\end{gather}
Note that, for $G\in \mathcal{G}_\infty$,
$\widehat{L}_{\vartheta_T\tilde{\vartheta}_T } G \in
\mathcal{G}_{\vartheta_T}$, where the latter space is predual to
$\mathcal{K}_{\vartheta_T}$, and
$\tilde{k}_s\in\mathcal{K}_{\vartheta_T}$ for all $s\leq t\leq T$.
For $G\in \mathcal{G}_\infty$, the action of
$\widehat{L}_{\vartheta_T \tilde{\vartheta}_T}$ on $G$ is the same
as in (\ref{W13}), that by (\ref{L9}) yields  $G_1:=\widehat{L}_{
\vartheta_T\tilde{\vartheta}_T}G\in \mathcal{G}_\infty$. Therefore,
one can write (\ref{W29b}) also for $G_1$. Repeating this procedure
$n$ times we arrive at the following
\begin{gather}
  \label{W29}
\langle \! \langle \tilde{k}_t, G \rangle \! \rangle =  \langle \!
\langle {k}_0, G \rangle \! \rangle + t \langle \! \langle {k}_0,
\widehat{L}_{\vartheta_T\tilde{\vartheta}_T} G \rangle \! \rangle +
\frac{t^2}{2} \langle \! \langle {k}_0,
(\widehat{L}_{\vartheta_T\tilde{\vartheta}_T})^2G \rangle \! \rangle
\\[.2cm] \nonumber + \cdots + \frac{t^{n-}}{(n-1)!} \langle \! \langle {k}_0,
(\widehat{L}_{\vartheta_T\tilde{\vartheta}_T})^{n-1}G \rangle \!
\rangle + \int_0^t \int_0^{t_1} \cdots \int_0^{t_{n-1}} \langle \!
\langle \tilde{k}_{t_n},
(\widehat{L}_{\vartheta_T\tilde{\vartheta}_T})^{n} G \rangle \!
\rangle d t_1 \cdots d t_n.
\end{gather}
Let $k_t$ be the solution as in (\ref{W20}). Our choice of
${\vartheta}_T$ is such that $k_t\in \mathcal{K}_{{\vartheta}_T}$,
hence (\ref{W29}) can also be written for this $k_t$, which yields
\begin{equation*}
\langle \! \langle \tilde{k}_t - k_t, G \rangle \! \rangle =
\int_0^t \int_0^{t_1} \cdots \int_0^{t_{n-1}} \langle \! \langle
\tilde{k}_{t_n} - k_{t_n},
(\widehat{L}_{\vartheta_T\tilde{\vartheta}_T })^{n} G \rangle \!
\rangle d t_1 \cdots d t_n.
\end{equation*}
Now by (\ref{W14}) we obtain from the latter, see (\ref{Ts1}),
\begin{gather*}
\left|\langle \! \langle \tilde{k}_t - k_t, G \rangle \!
\rangle\right| \leq 2 \frac{n^n}{n! e^n} \left(
\frac{t}{T_*(\vartheta_T)}\right)^n |G|_{\vartheta_0} \to 0, \qquad
{\rm as} \ n\to +\infty.
\end{gather*}
Thus, $\tilde{k}_t = k_t$ for $t$ satisfying (\ref{Ts1}). The
continuation of this equality to all $t$ can be made by repeating
this construction, similarly as in \cite[the proof of Theorem 3.5,
pages 659, 660]{asia1}. Now the equality $\tilde{\mu}_t = \mu_t$
follows by the fact that each $\mu\in \mathcal{P}_{\rm exp}$ is
identified by its correlation function, see Remark \ref{W1rk}.
\end{proof}

\subsection{Useful estimates}
The aim of this subsection is to prepare the proof of Lemma
\ref{W1lm}. A priori a solution $\mu_t$ need not be a sub-Poissonian
state, so one can speak of $\mu_t(F)$ only for bounded $F$, in
particular of $\mu_t (\widehat{F}_\tau(\mathbf{v}|\cdot))$. At the
same time, $\widehat{F}_\tau(\mathbf{v}|\cdot)$ is bounded for
positive $\tau_i$ only, see the proof of Proposition \ref{TH1pn}.
Assume that we have obtained an estimate of $\mu_t
(\widehat{F}_\tau(\mathbf{v}|\cdots))$ that is uniform in $\tau$,
which might allow for passing to the limit $\max_i\tau_i \to 0$.
Assume further that this limit satisfies an estimate similar to
(\ref{no6}) with a certain $t$-dependent $\varkappa$. Then the proof
will follow with the help of Definition \ref{no1dfg}. Let us then
turn to obtaining such estimates. Here we will mostly follow the way
elaborated in \cite{KR}.

Our  starting point is the estimate obtained in (\ref{TH7}) the
right-hand side of which is an element of $\mathcal{D}(L)$.
Significantly, it is independent of the interaction terms $\phi_i$,
$i=0,1$, where both components appear in a multiplicative form,
similarly as in $\widehat{F}_\tau(\mathbf{v}|\cdot)$ in (\ref{de2}).
Another observation is that in the latter function all $v_{i,j}$
with the same $i=0,1$ can be different, whereas (\ref{no7}) is based
on just two functions $\theta_0$, $\theta_1$. Keeping this fact in
mind, we introduce the following functions. Fix $\theta_0, \theta_1
\in \varTheta_\psi^{+}$ and set, cf. (\ref{de2}),
\begin{equation}
  \label{W32}
\varPhi_{\tau_i}^{m_i} (\theta_i|\gamma_i) =
\widehat{F}^{m_i}_{\tau_i} (\mathbf{v}_i|\gamma_i)|_{v_{i,j} =
\theta_i} = \sum_{\mathbf{x}^{m_i}\in \gamma_i} \theta_i^{\otimes
m_i} (\mathbf{x}^{m_i}) \exp\left( -\tau_i \Psi(\gamma_i \setminus
\mathbf{x}^{m_i}\right), \quad i=0,1.
\end{equation}
Along with this, we also introduce
\begin{eqnarray}
  \label{W33}
\varPhi_{\tau_i}^{m_i, \theta^1_i} (\theta_i|\gamma_i) & = &
\widehat{F}^{m_i}_{\tau_i} (\mathbf{v}_i|\gamma_i)|_{v_{i,1} =
a_i\theta_i, \ v_{i,j} = \theta_i, j\geq 2}, \\[.2cm] \varPhi_{\tau_i, 1}^{m_i} (\theta_i|\gamma_i) &
= & m_i\varPhi_{\tau_i}^{m_i, \theta^1_i} (\theta_i|\gamma_i) +
\tau_i c_a \bar{c}_{\theta_i} \widehat{F}^{m_i+1}_{\tau_i}
(\gamma_i), \nonumber
\end{eqnarray}
where $a_i \theta_i$ and $\bar{c}_{\theta_i}$ are as in (\ref{TH4})
and in (\ref{C80}), respectively; $\widehat{F}^{m_i+1}_{\tau_i}
(\gamma_i)$ is as in (\ref{de11}). Now we set
\begin{eqnarray}
  \label{W32a}
\varPhi_{\tau}^{m} (\theta|\gamma)& = & \varPhi_{\tau_0}^{m_0}
(\theta_0|\gamma_0)\varPhi_{\tau_1}^{m_1} (\theta_1|\gamma_1),
\end{eqnarray}
for which by the estimate in (\ref{TH7}) we then get
\begin{eqnarray}
  \label{W34}
\left| L\varPhi_{\tau}^{m} (\theta|\gamma) \right| \leq
\varPhi_{\tau_0, 1}^{m_0} (\theta_0|\gamma_0)\varPhi_{\tau_1}^{m_1}
(\theta_1|\gamma_1)+ \varPhi_{\tau_0}^{m_0} (\theta_0|\gamma_0)
\varPhi_{\tau_1, 1}^{m_1} (\theta_1|\gamma_1) =:
\varPhi_{\tau,1}^{m} (\theta|\gamma).
\end{eqnarray}
Each of the summands above is a linear combination of the
corresponding functions $\widehat{F}^{\bar{m}_i}_{\tau_i}
(\mathbf{v}_i|\cdot)$. Hence $\varPhi_{\tau,1}^{m} (\theta|\cdot)\in
\mathcal{D}(L)$, and one can estimate $L\varPhi_{\tau,1}^{m}
(\theta|\cdot)$ by repeating the above procedure based on
(\ref{TH7}). This yields
\begin{eqnarray}
  \label{W34a}
\left|L\varPhi_{\tau,1}^{m} (\theta|\cdot) (\gamma)\right| &\leq &
\left(\int_X \sum_{x\in \gamma_0} a_0(x-y) \left|\nabla_0^{y,x}
\varPhi_{\tau_0,1}^{m_0} (\theta_0|\gamma_0)\right| dy\right)
\varPhi_{\tau_1}^{m_1} (\theta_1|\gamma_1) \\[.2cm]\nonumber &+& \left(\int_X \sum_{x\in \gamma_1} a_1(x-y) \left|\nabla_1^{y,x}
\varPhi_{\tau_1}^{m_1} (\theta_1|\gamma_1)\right| dy\right)
\varPhi_{\tau_0,1}^{m_0} (\theta_0|\gamma_0) \\[.2cm]\nonumber &+& \left(\int_X \sum_{x\in \gamma_0} a_0(x-y) \left|\nabla_0^{y,x}
\varPhi_{\tau_0}^{m_0} (\theta_0|\gamma_0)\right| dy\right)
\varPhi_{\tau_1,1}^{m_1} (\theta_1|\gamma_1) \\[.2cm]\nonumber &+&
\left(\int_X \sum_{x\in \gamma_1} a_1(x-y) \left|\nabla_1^{y,x}
\varPhi_{\tau_1,1}^{m_1} (\theta_1|\gamma_1)\right| dy\right)
\varPhi_{\tau_0}^{m_0} (\theta_0|\gamma_0).
\end{eqnarray}
Each of the summands of the right-hand side of (\ref{W34a}) can be
estimated in the same way as in the last two lines of (\ref{TH7}).
This procedure was systematically elaborated in \cite{KR}, which we
are going to use now. To describe it, we introduce the following
notions. First, for $l\in \mathds{N}$ and $\theta_i$, $i=0,1$, we
define, see (\ref{TH4}),
\begin{equation}
  \label{W34b}
\theta_i^l = a_i\theta_i^{l-1}, \qquad \theta_i^0:=\theta_i.
\end{equation}
Then as in \cite[page 28]{KR}, for $p\in \mathds{N}$ and $q\in
\mathds{N}_0$, by $\mathcal{C}_{p,q}$ we denote the set of all
integer-valued sequences $c=\{c_l\}_{l\in \mathds{N}_0} \subset
\mathds{N}_0$ such that
\begin{equation}
  \label{W35}
c_0 + c_1 +\cdots + c_l + \cdots = p, \qquad  c_1 + 2c_2 +\cdots + l
c_l + \cdots = q.
\end{equation}
For instance, $\mathcal{C}_{p,0}$ is a singleton consisting of $c=
\{p, 0, \dots , 0, \dots\}$, $\mathcal{C}_{p,2}$ consists of
$c=\{p-1, 0, 1, 0, \dots\}$ and $c=\{p-2, 2, 0, \dots\}$ for $p\geq
2$. Thereafter, we set
\begin{eqnarray}
  \label{W37}
C_{p,q}(c) & = & \frac{p! q!}{c_0! c_1! c_2! \cdots (0!)^{c_0}
(1!)^{c_1}(2!)^{c_2}\cdots}, \qquad c\in
\mathcal{C}_{p,q}, \\[.2cm] \nonumber
w_k(p,q)& = & \varDelta^k p^q = \frac{1}{k!}\sum_{l=0}^k
(-1)^{k-l}{k \choose l} (p+l)^q, \qquad k\in \mathds{N}_0.
\end{eqnarray}
Note that $\varDelta$ is the step-one forward difference operator
for which
\begin{equation}
  \label{W37a}
\varDelta^k p^q = 0, \qquad {\rm for} \ k>q.
\end{equation}
Next, for $c\in \mathcal{C}_{m_i,q}$, we write
$\mathbf{v}_i^c(\mathbf{x}^{m_i})= \prod_{j=1}^{m_i} v_{i,j}(x_j)$,
see (\ref{de1}), where the number of $v_{i,j}$ equal to $\theta_i$
is $c_0$, the number of $v_{i,j}$ equal to $\theta^1_i$ is $c_1$,
see (\ref{W34b}), the number of $v_{i,j}$ equal to $\theta^2_i$ is
$c_2$, etc, cf. (\ref{W33}). Thereafter, for $\theta_i \in
\varTheta^{+}_\psi$, $i=0,1$, such that
\begin{equation}
  \label{W36a}
 \bar{c}_{\theta_i} = 1,
\end{equation}
see (\ref{C80}), we set
\begin{equation}
  \label{W36}
\varPhi_{\tau_i, q}^{m_i} (\theta_i|\gamma_i) = \sum_{c\in
\mathcal{C}_{m_i,q}} C_{m_i,q} (c) \widehat{F}^{m_i}_{\tau_i}
(\mathbf{v}_i^c|\gamma_i) + c_a^q \sum_{k=1}^{q} \tau_i^k w_k(m_i,q)
\widehat{F}^{m_i+k}_{\tau_i} (\gamma_i),
\end{equation}
see (\ref{de11}). For $q=0$ (resp. $q=1$), this function coincides
with that given in the first (resp. second) line of (\ref{W33}). Let
us now denote, cf. (\ref{de9}), (\ref{TH7}),
\begin{equation*}
\mathcal{L}_i \varPhi_{\tau_i, q}^{m_i} (\theta_i|\gamma_i) = \int_X
\sum_{x\in \gamma_i} a_i(x-y) \left| \nabla_i^{y,x}\varPhi_{\tau_i,
q}^{m_i} (\theta_i|\gamma_i)\right| dy, \quad i=0,1.
\end{equation*}
In \cite[Appendix]{KR}, the following was proved, see also (5.24)
\emph{ibid}.
\begin{equation}
  \label{W38}
\mathcal{L}_i \varPhi_{\tau_i, q}^{m_i} (\theta_i|\gamma_i)\leq
\varPhi_{\tau_i, q+1}^{m_i} (\theta_i|\gamma_i), \qquad i=0,1,
\end{equation}
holding for all $\theta_i\in \varTheta_\psi^{+}$ satisfying
(\ref{W36a}), $m_i\in \mathds{N}$, $q\in \mathds{N}_0$ and
$\tau_i\in (0,1]$. By means of (\ref{W38}) we then get from
(\ref{W34a}) the following estimate
\begin{eqnarray*}
\left|L\varPhi_{\tau,1}^{m} (\theta|\gamma)\right| &\leq &
\varPhi_{\tau_0, 2}^{m_0} (\theta_0|\gamma_0)\varPhi_{\tau_1}^{m_1}
(\theta_1|\gamma_1) + 2 \varPhi_{\tau_0, 1}^{m_0}
(\theta_0|\gamma_0)\varPhi_{\tau_1, 1}^{m_1} (\theta_1|\gamma_1)\\[.2cm] \nonumber&
+ & \varPhi_{\tau_0}^{m_0} (\theta_0|\gamma_0)\varPhi_{\tau_1,
2}^{m_1} (\theta_1|\gamma_1).
\end{eqnarray*}
The estimates obtained in (\ref{W34}), (\ref{W34a}) can be
summarized as follows. Set
\begin{equation}
  \label{W40}
\varPhi_{\tau,q}^{m} (\theta|\gamma) = \sum_{l=0}^q {q \choose l}
\varPhi_{\tau_0, q-l}^{m_0} (\theta_0|\gamma_0) \varPhi_{\tau_1,
l}^{m_1} (\theta_1|\gamma_1).
\end{equation}
Then the main result of this subsection is the following estimate
\begin{equation}
  \label{W40a}
\left|L\varPhi_{\tau,q}^{m} (\theta|\gamma)\right| \leq
\varPhi_{\tau,q+1}^{m} (\theta|\gamma),
\end{equation}
holding for all $q\in \mathds{N}_0$, $m\in \mathds{N}_0^2$,  $\tau
=(\tau_0, \tau_1)$, $\tau_i\in (0,1]$, and $\theta=(\theta_0,
\theta_1)$, $\theta_i\in \varTheta_\psi^{+}$ satisfying
(\ref{W36a}). The first step in the proof of (\ref{W40a}) is made as
in (\ref{TH7}), first estimate. Next, one applies (\ref{W38}).

\subsection{Proving Lemma \ref{W1lm}}

By (\ref{W40}) and (\ref{W36}), and then by Proposition \ref{TH1pn},
$\varPhi_{\tau,q}^{m} (\theta|\cdot)$ is a bounded continuous
function of $\gamma\in \Gamma^2_*$. However, the upper bound of it
may depend on $q$. Our aim is to estimate this dependence.
\begin{proposition}
  \label{W4pn}
For each $\varepsilon\in (0,1)$, $\tau=(\tau_0,\tau_1)$, $\tau_0,
\tau_1\in (0,1]$, and $m=(m_0,m_1)\in \mathds{N}_0^2$, there exists
$\bar{C}>0$, dependent on $\varepsilon$, $\tau$ and $m$, such that
the following holds
\begin{equation}
  \label{W36c}
\forall q\in \mathds{N}_0 \qquad \varPhi^m_{\tau,q} (\theta|\gamma)
\leq \frac{q!}{\rho_\varepsilon^q} \bar{C}, \qquad \rho_\varepsilon
:= \frac{1}{c_a} \log ( 1 +\varepsilon),
\end{equation}
uniformly in $\gamma\in \Gamma^2_*$ and
$\theta=(\theta_0,\theta_1)$,
 $\theta_0, \theta_1 \in \varTheta^{+}_\psi$ satisfying
(\ref{W36a}).
\end{proposition}
\begin{proof}
Introduce
\begin{equation}
  \label{W46}
V^m_\tau(\rho|\gamma) = \sum_{q=0}^\infty \frac{\rho^q}{q!}
\varPhi_{\tau,q}^{m} (\theta|\gamma), \qquad \rho\geq 0,
\end{equation}
where $m$, $\tau$ and $\theta$ are as assumed. Let us estimate the
growth of this function. By (\ref{W34b}) and (\ref{TH4}),
(\ref{W36a}), we have
\[
\theta_i^l (x) \leq c_a^l \psi(x),
\]
which we use to get the following
\begin{equation*}
\widehat{F}^{m_i}_{\tau_i} (\mathbf{v}^c_i|\gamma_i) \leq c_a^{c_1 +
2 c_2 +\cdots } \widehat{F}^{m_i}_{\tau_i}(\gamma_i) = c_a^q
\widehat{F}^{m_i}_{\tau_i}(\gamma_i), \quad c\in
\mathcal{C}_{m_i,q},
\end{equation*}
where we used the second equality in (\ref{W35}). Now we employ the
fact, see (\ref{W37}), that
\begin{equation}
  \label{W47}
\sum_{c\in \mathcal{C}_{p,q}} C_{p,q}(c) = p^q = \varDelta^0 p^q =
w_0(p,q),
\end{equation}
which was proved in \cite[Appendix]{KR}, and obtain from (\ref{W36})
the following estimate
\begin{gather*}
\varPhi_{\tau_i, q}^{m_i} (\theta_i|\gamma_i) \leq c_a^q
\sum_{k=0}^q \tau_i^k w_k(m_i,q) \widehat{F}^{m_i+k}_{\tau_i}
(\gamma_i).
\end{gather*}
We use the latter in (\ref{W40}) and then in (\ref{W46}) to get the
following
\begin{eqnarray}
  \label{W49}
V^m_\tau(\rho|\gamma) & \leq & \sum_{q=0}^\infty \frac{(c_a
\rho)^q}{q!}
\sum_{l=0}^q \frac{q!}{l!(q-l)!} \\[.2cm] \nonumber & \times & \sum_{k_0=0}^{q-l}
\sum_{k_1=0}^l \tau_0^{k_0} \tau_1^{k_1} w_{k_0}(m_0,q-l)
w_{k_1}(m_1, l)
\widehat{F}^{m_0+k_0}_{\tau_0}(\gamma_0)\widehat{F}^{m_1+k_1}_{\tau_1}(\gamma_1)  \\[.2cm] \nonumber
& = & \sum_{k_0=0}^\infty \sum_{k_1=0}^\infty
\frac{\tau_0^{k_0}\tau_1^{k_1}}{k_0! k_1!}
W_{\tau_0,k_0}^{m_0}(\rho|\gamma_0)W_{\tau_1,k_1}^{m_1}(\rho|\gamma_1),
\end{eqnarray}
where we also used the fact that $w_k(p,q) =0$ whenever $k>q$, see
(\ref{W37a}). The functions that appear in the last line of
(\ref{W49}) are
\begin{eqnarray}
  \label{W50}
W_{\tau_i,k_i}^{m_i}(\rho|\gamma_i)& := & \left(\sum_{l=0}^\infty
\frac{(c_a  \rho)^l}{l!} w_{k_i} (m_i,
l)\right)\widehat{F}_{\tau_i}^{m_i+k_i}(\gamma_i)
 \\[.2cm] \nonumber & = & \left(e^{c_a \rho}-1 \right)^{k_i} e^{m_i c_a
\rho}\widehat{F}_{\tau_i}^{m_i+k_i}(\gamma_i),
\end{eqnarray}
where the second line was derived by means of  the second line of
(\ref{W37}). Then we have
\begin{eqnarray}
  \label{W51}
V^m_\tau(\rho|\gamma)& \leq & e^{(m_0+m_1) c_a \rho}
Y^{m_0}_{\tau_0}(\rho|\gamma_0)Y^{m_1}_{\tau_1}(\rho|\gamma_1), \\[.2cm]
\nonumber Y^{m_i}_{\tau_i}(\rho|\gamma_i)& := & \sum_{k=0}^\infty
\frac{(\tau_i \ell(\rho))^k}{k!}
\widehat{F}^{m_i+k}_{\tau_i}(\gamma_i), \quad i=0,1,
\end{eqnarray}
where $\ell(\rho) = e^{c_a \rho}-1$. By means of the estimate
obtained in (\ref{de3}), (\ref{ND}) with $u_j(x) = \psi(x) e^{\tau_i
\psi(x)} \leq e^{\tau_i}\psi(x)$ and $\tau_i\leq 1$, we then obtain
\[
\widehat{F}^{m_i+k}_{\tau_i}(\gamma_i) \leq
\left(\frac{m_i+k}{\tau_i }\right)^{m_i+k},
\]
which yields that both series in (\ref{W51}) converge whenever
$e^{c_a \rho}-1 < 1$. Take $\varepsilon\in (0,1)$ and
$\rho_\varepsilon$ as in (\ref{W36c}), then set
\[
\bar{Y}_i = \frac{1}{\tau_i^{m_i}} \sum_{k=0}^\infty
\frac{\varepsilon^k}{k!} (m_i+k)^{m_i+k}, \qquad i=0,1.
\]
Now (\ref{W36c}) follows by
\[
V^m_\tau(\theta|\gamma) \leq (1+\varepsilon)^{m_0+m_1} \bar{Y}_0
\bar{Y}_1 =: \bar{C},
\]
see (\ref{W46}) and (\ref{W51}).
\end{proof}
\noindent {\it Proof of Lemma \ref{W1lm}.} By  (\ref{W40}), each
$\varPhi^m_\tau(\theta|\cdot)$ is a linear combination of the
elements of $\widehat{\mathcal{F}}$, and hence
$\varPhi^m_{\tau,q}(\theta|\cdot)\in \mathcal{D}(L)$, see
(\ref{de2}) and Definition \ref{THdf}. If $t\mapsto \mu_t$ solves
(\ref{FPE}), see Definition \ref{A01df}, then
\begin{gather*}
\mu_t (\varPhi^m_{\tau,q}(\theta|\cdot)) =  \mu_0
(\varPhi^m_{\tau,q}(\theta|\cdot)) + \int_0^t \mu_s
(L\varPhi^m_{\tau,q}(\theta|\cdot)) ds \\[.2cm] \nonumber \leq \mu_0
(\varPhi^m_{\tau,q}(\theta|\cdot)) + \int_0^t \mu_s
(\varPhi^m_{\tau,q+1}(\theta|\cdot)) ds,
\end{gather*}
see (\ref{W40a}). Now we repeat this estimate due times and arrive
at the following one
\begin{eqnarray*}
\mu_t (\varPhi^m_{\tau}(\theta|\cdot))& \leq & \sum_{q=0}^{n-1}
\frac{t^q}{q!} \mu_0 (\varPhi^m_{\tau, q}(\theta|\cdot)) + \int_0^t
\int_0^{t_1} \cdots \int_0^{t_{n-1}} \mu_{t_n}  (\varPhi^m_{\tau,
n}(\theta|\cdot)) d t_n dt_{n-1} \cdots d t_1 \qquad \\[.2cm] \nonumber & \leq & \sum_{q=0}^{n-1}
\frac{t^q}{q!} \mu_0 (\varPhi^m_{\tau, q}(\theta|\cdot)) +
\left(\frac{t}{\rho_\varepsilon} \right)^n \bar{C},
\end{eqnarray*}
where we also used (\ref{W36c}) and the fact that $\mu_t$ is a
probability measure. For $t\in (0, \rho_\varepsilon)$, the second
summand in the last line vanishes as $n\to +\infty$, which yields
\begin{equation}
  \label{B2}
\mu_t (\varPhi^m_{\tau}(\theta|\cdot)) \leq \sum_{q=0}^{\infty}
\frac{t^q}{q!} \mu_0 (\varPhi^m_{\tau, q}(\theta|\cdot)), \qquad t<
\log(1+\varepsilon)/c_a.
\end{equation}
Now we recall that $\widehat{F}^{m_i}_{\tau_i}(\mathbf{v}_i|\cdot)$
can be written as the
$K\widehat{G}^{m_i}_{\tau_i}(\mathbf{v}_i|\cdot)$, see (\ref{V4}).
Since $\widehat{F}^{m_i}_{\tau_i}(\cdot)$ is a particular case of
$\widehat{F}^{m_i}_{\tau_i}(\mathbf{v}_i|\cdot)$, see (\ref{de11}),
we can also write it as $K\widehat{G}^{m_i}_{\tau_i}(\cdot)$, where
$\widehat{G}^{m_i}_{\tau_i}(\cdot)$ is obtained by the corresponding
choice of $\mathbf{v}_i$ in (\ref{V3}), (\ref{V4}). This allows us
to write
\begin{eqnarray*}
\varPhi^{m_i}_{\tau_i, q} (\theta_i|\gamma_i) & = &  \sum_{\eta_i
\subset \gamma_i} \Pi^{m_i}_{\tau_i, q} (\theta_i|\eta_i), \qquad
i=0,1, \\[.2cm] \nonumber  \Pi^{m_i}_{\tau_i, q} (\theta_i|\eta_i)
& = & \sum_{c\in \mathcal{C}_{m_i,q}} C_{m_i,q}(c)
\widehat{G}^{m_i}_{\tau_i} (\mathbf{v}_i^c|\eta_i) + c_a^q
\sum_{k=1}^q \tau_i^k w_k(m_i,q)
\widehat{G}^{m_i+k}_{\tau_i}(\eta_i),
\end{eqnarray*}
where $\mathbf{v}_i^c$ is as in (\ref{W36}). Then by (\ref{W40}) we
obtain
\begin{eqnarray*}
\varPhi^{m}_{\tau, q} (\theta|\gamma)  =   \sum_{\eta_0 \subset
\gamma_0} \sum_{\eta_1\subset \gamma_1}  \sum_{l=0}^q {q \choose l}
\Pi^{m_0}_{\tau_0, q-l} (\theta_0|\eta_0)\Pi^{m_1}_{\tau_1, l}
(\theta_1|\eta_1).
\end{eqnarray*}
Now we may use the fact that $\mu_0$ is sub-Poissonian and write,
see (\ref{Qo}), (\ref{W6}),
\begin{equation*}
\mu_0 (\varPhi^{m}_{\tau, q} (\theta|\cdot)) = \sum_{l=0}^q {q
\choose l} \int_{\Gamma_0} \int_{\Gamma_0} k_{\mu_0}(\eta_0, \eta_1)
\Pi^{m_0}_{\tau_0, q-l} (\theta_0|\eta_0)\Pi^{m_1}_{\tau_1, l}
(\theta_1|\eta_1) \lambda( d\eta_0)  \lambda( d\eta_1).
\end{equation*}
Let $e^{\vartheta_0}$ be the type of $\mu_0$. Then by (\ref{15b}) we
have
\begin{gather*}
\mu_0 (\varPhi^{m}_{\tau, q} (\theta|\cdot)) \leq \sum_{l=0}^q {q
\choose l} \varOmega^{m_0}_{\tau_0, q-l} (\theta_0)
\varOmega^{m_1}_{\tau_1, l} (\theta_1),
\end{gather*}
which yields in (\ref{B2}),
\begin{equation}
  \label{B3u}
\mu_t (\varPhi^m_{\tau} (\theta|\cdot)) \leq
\widehat{\Omega}^{m_0}_{\tau_0} (\theta_0|t)
\widehat{\Omega}^{m_1}_{\tau_1} (\theta_1|t), \qquad t<
\log(1+\varepsilon)/c_a,
\end{equation}
where
\begin{eqnarray}
  \label{B3v}
\widehat{\Omega}^{m_i}_{\tau_i} (\theta_i|t) = \sum_{q=0}^\infty
\frac{t^q}{q!} \varOmega^{m_i}_{\tau_i,q} (\theta_i), \qquad i=0,1,
\end{eqnarray}
and, see (\ref{V6}),
\begin{eqnarray}
  \label{B3a}
\varOmega^{m_i}_{\tau_i, n} (\theta_i)& := &  \int_{\Gamma_0}
e^{\vartheta_0 |\eta_i|}\left| \Pi^{m_i}_{\tau_1, n}
(\theta_i|\eta_i)\right| \lambda ( d \eta_i) \\[.2cm] \nonumber
& \leq & \sum_{c\in \mathcal{C}_{m_i,n}} C_{m_i,n}(c)
\int_{\Gamma_0} e^{\vartheta_0 |\eta_i|} \left|
\widehat{G}^{m_i}_{\tau_i} (\mathbf{v}^c_i|\eta_i)\right| \lambda (
d \eta_i) \\[.2cm] \nonumber & + & c_a^n \sum_{k=1}^n \tau_i^k
w_k(m_i,n) \int_{\Gamma_0} e^{\vartheta_0 |\eta_i|} \left|
\widehat{G}^{m_i+k}_{\tau_i} (\eta_i)\right| \lambda ( d \eta_i) \\[.2cm] \nonumber
& \leq & \exp\left(m_i \vartheta_0 + \tau_i \langle \psi \rangle
e^{\vartheta_0} \right)\sum_{c\in \mathcal{C}_{m_i,n}} C_{m_i,n}(c)
\langle \theta_i \rangle^{c_0} \langle \theta^1_i \rangle^{c_1}
\cdots \langle \theta_i^l \rangle^{c_l} \cdots \\[.2cm] \nonumber
& + & c_a^n \sum_{k=1}^n \tau_i^k w_k(m_i,n)\langle \psi
\rangle^{m_i+k} \exp\left((m_i+k) \vartheta_0 + \tau_i\langle \psi
\rangle e^{\vartheta_0} \right).
\end{eqnarray}
By (\ref{W34b}) and (\ref{TH3}), (\ref{TH4}), for $\theta_i\in
\varTheta^{+}_\psi$, we have, see  (\ref{L2}),
\begin{equation*}
\langle \theta_i^l \rangle \leq (\alpha+1)^l \langle \theta_i
\rangle
\end{equation*}
since
\begin{gather*}
\langle \theta_i^l \rangle = \int_{X}\int_X a_i(x-y) \theta_i^{l-1}
(y) d x dy + \langle \theta_i^{l-1} \rangle \leq (\alpha + 1)
\langle \theta_i^{l-1} \rangle.
\end{gather*}
Then, see (\ref{W35}),
\[
\langle \theta_i \rangle^{c_0} \langle \theta^1_i \rangle^{c_1}
\cdots \langle \theta_i^l \rangle^{c_l} \cdots \leq \langle \theta_i
\rangle^{c_0 + c_1 + \cdots} (\alpha+1)^{c_1 + 2 c_2 + \cdots } =
\langle \theta_i^l \rangle^{m_i}  (\alpha+1)^{n}.
\]
We use this in (\ref{B3a}) and obtain, see also (\ref{W47}),
\begin{gather*}
\varOmega^{m_i}_{\tau_i, n} (\theta_i) \leq \exp\left(m_i
\vartheta_0 +\tau_i \langle \psi \rangle e^{\vartheta_0} \right)
\bigg{[} \langle \theta_i\rangle^{m_i} [m_i(\alpha+1)]^n \\[.2cm] \nonumber + c_a^n
\langle \psi \rangle^{m_i} \sum_{k=1}^n w_k (m_i,n) \left(\tau_i
\langle \psi \rangle e^{\vartheta_0} \right)^k \bigg{]}.
\end{gather*}
Now we use this in (\ref{B3v}), and finally arrive at the following
estimate
\begin{eqnarray}
  \label{B7}
\widehat{\Omega}^{m_i}_{\tau_i} (\theta_i|t) &\leq & \exp\left(m_i
\vartheta_0 + \tau_i \langle \psi \rangle e^{\vartheta_0}
\right)\bigg{[} e^{(\alpha+1)m_i t} \langle \theta_i \rangle^{m_i}
\\[.2cm] \nonumber & + & \langle \psi \rangle^{m_i} \sum_{q=0}^\infty\frac{(c_a
t)^q}{q!} \sum_{k=1}^\infty w_k (m_i,q) \left( \tau_i \langle \psi
\rangle e^{\vartheta_0} \right)^k \bigg{]},
\end{eqnarray}
where we took into account that $w_k(m_i,q) =0$ for $k>q$, see
(\ref{W37a}). By (\ref{W37}) we have
\begin{eqnarray}
  \label{B7a}
& & \sum_{q=0}^\infty\frac{(c_a t)^q}{q!} \sum_{k=1}^\infty w_k
(m_i,q) \left( \tau_i \langle \psi \rangle e^{\vartheta_0} \right)^k
\\[.2cm] \nonumber & & \quad  = \sum_{k=1}^\infty \frac{1}{k!} \left( \tau_i \langle \psi
\rangle e^{\vartheta_0} \right)^k\sum_{l=0}^k (-1)^{k-l} {k \choose
l} \sum_{q=0}^\infty \frac{(c_a t)^q}{q!} (m_i+ l)^q \\[.2cm] \nonumber & & \quad
= e^{m_i c_at} \sum_{k=1}^\infty \frac{1}{k!} \left( \tau_i \langle
\psi \rangle e^{\vartheta_0} \right)^k \sum_{l=0}^k (-1)^{k-l} {k
\choose l} e^{l c_a t} \\[.2cm] \nonumber & & \quad
= e^{m_i c_at} \bigg{[}\exp\left((e^{c_a t}-1) \tau_i \langle \psi
\rangle e^{\vartheta_0} \right) -1 \bigg{]}  \\[.2cm] \nonumber & &
\quad \leq \tau_i \varepsilon (1+\varepsilon)^{m_i}
\exp\left(\langle \psi\rangle e^{\vartheta_0} \right),
\end{eqnarray}
where $t$ is as in (\ref{B3u}) and $\tau_i\leq 1$, $\varepsilon <1$.
We use now (\ref{B7a}) in (\ref{B7}) and then turn (\ref{B3u}) into
the following estimate
\begin{eqnarray}
  \label{B4}
\mu_t (\varPhi^m_\tau(\theta|\cdot))& \leq & \exp\left( (m_0 + m_1)
\vartheta_0 + (\tau_0 + \tau_1) \langle \psi \rangle
e^{\vartheta_0}\right) \\[.2cm] \nonumber & \times & \left[ e^{(\alpha + 1)m_0 t} \langle \theta_0 \rangle^{m_0} + \tau_0 \varepsilon
(1+\varepsilon)^{m_0} \exp\left(\langle \psi\rangle e^{\vartheta_0}
\right) \right]\\[.2cm] \nonumber & \times & \left[ e^{(\alpha + 1)m_1 t} \langle \theta_1 \rangle^{m_1} + \tau_1 \varepsilon
(1+\varepsilon)^{m_1} \exp\left(\langle \psi\rangle e^{\vartheta_0}
\right) \right], \quad t< \log (1+\varepsilon)/c_a.
\end{eqnarray}
For each $\gamma\in \Gamma_*^2$ and a decreasing sequence of
positive $\tau_k \to 0$, the sequence
$\varPhi^m_\tau(\theta|\gamma)$, $\tau_i = \tau^k$, $i=0,1$, is
nondecreasing, see (\ref{W36}) and (\ref{de2}). By (\ref{B4}) and
the Beppo Levi monotone convergence lemma we conclude that the
pointwise limit, see (\ref{W32a}), (\ref{W32}) and (\ref{no6}),
(\ref{no10}),
\begin{gather}
  \label{B5}
\varPhi^m_0(\theta|\gamma) = \lim_{k\to
+\infty}\varPhi^{m}_{\tau^k}(\theta|\gamma) = H^m_{\theta }(\gamma)=
\sum _{(\mathbf{x}^{m_0}, \mathbf{y}^{m_1}) \in
\gamma}\theta^{\otimes m} (\mathbf{x}^{m_0}, \mathbf{y}^{m_1}),
\end{gather}
is $\mu_t$-integrable. Moreover, by the same lemma and (\ref{B4}) it
follows that
\begin{equation}
  \label{B9}
\mu_t (\varPhi^m_0(\theta|\cdot)) = \chi_{\mu_t}^{(m)}
(\theta^{\otimes m}) \leq \varkappa_t^{|m|}
\|\theta_0\|^{m_0}_{L^1(X)} \|\theta_1\|^{m_1}_{L^1(X)},
\end{equation}
with
\begin{equation}
  \label{B10}
\varkappa_t= e^{\vartheta_0 + (\alpha+1) t},
\end{equation}
which by (\ref{no7}) yields the property in question for $t<
\log(1+\varepsilon)/c_a$. Since this length of the validity interval
is independent of $\mu_0$, the further continuation can be done by
the literal repetition of the procedure above. \hfill $\square$
\begin{remark}
  \label{Brk}
According to Proposition \ref{W1pn}, see also \cite[Theorem
3.5]{asia1}, the type of the solution $\mu_t$ obtained in Lemma
\ref{W2lm} does not exceed $\exp(\vartheta_0 + \alpha t)$, which is
a more precise estimate than that given in (\ref{B9}), (\ref{B10}).
\end{remark}

\section{Existence: Approximating Models}
The aim of this section is introducing approximating models, for
which the corresponding processes can be constructed by employing
explicitly derived Markov transition functions. By this result the
process in question will be obtained as the limit of such
approximating processes. Similarly as in \cite{KR}, the Kolmogorov
operators for the approximating models are obtained from that in
(\ref{L}).

\subsection{The models}

We begin by introducing the basic function, cf. (\ref{P})
\begin{equation}
  \label{Sq}
\psi_\sigma (x) = \frac{1}{1+\sigma |x|^{d+1}}, \qquad \sigma \in
(0,1], \ \ x\in X=\mathds{R}^d,
\end{equation}
and then define
\begin{equation}
  \label{Sq1}
a_{i}^\sigma (x,y) = a_i (x-y) \psi_\sigma(x), \qquad i=0,1,
\end{equation}
where $a_i$ are the jump kernels that appear in (\ref{L}). It is
clear that these $a^\sigma_{i}$ satisfy, cf. (\ref{C7}), (\ref{C8}),
\begin{gather}
  \label{Sq2}
\max_{i=0,1}\sup_{(x,y)\in X^2} a^\sigma_i (x,y)\leq \|a\|, \\[.2cm] \nonumber
\max\left\{\sup_{y\in X} \int_X |x|^l a_i^{\sigma} (x,y) d x ;\
\sup_{y\in X} \int_X |x|^l a_i^{\sigma} (y,x) d x\right\} \leq
\bar{a}^{(l)}_i, \quad l=0, \dots , d+1, \quad i=0,1.
\end{gather}
The Kolmogorov operator corresponding to the approximation model is
obtained by replacing in (\ref{L}) $a_i(x-y)$ with $a_i^\sigma
(x,y)$, $i=0,1$; that is, it has the form
\begin{eqnarray}
  \label{Lsi}
  ( {L}^\sigma F) ({\gamma}) & = & \sum_{x\in {\gamma}_0}\int_{X} a_0^\sigma (x,y) \exp\left( - \sum_{z\in {\gamma}_1}
  \phi_0
  (z-y)\right) \left[F({\gamma}\setminus x \cup_0 y) - F({\gamma})
  \right] d y \qquad  \\[.2cm] \nonumber & + & \sum_{x\in {\gamma}_1} \int_{X} a_1^\sigma (x,y)
  \exp\left( - \sum_{z\in {\gamma}_0} \phi_1
  (z-y)\right) \left[F({\gamma}\setminus x \cup_1 y) - F({\gamma})
  \right] d y.
\end{eqnarray}
Noteworthy, in the approximating model the kernels corresponding to
the jumps from $x$ to $y$ get smaller if $x$ goes away from the
origin. Now we introduce $L^{\Delta, \sigma}$ by replacing in
(\ref{L1}) $a_i(x-y)$ with $a_i^\sigma (x,y)$, $i=0,1$, where one
should take into account that $a^\sigma_i(x,y) \neq
a^\sigma_i(y,x)$. Then, cf. (\ref{W7}),
\begin{equation}
  \label{Sq3}
\mu (L^\sigma KG) = \langle \! \langle L^{\Delta, \sigma} k_\mu, G
\rangle \! \rangle , \qquad \sigma \in [0,1],
\end{equation}
holding for each $\mu\in \mathcal{P}_{\rm exp}$ and $G \in
\mathcal{G}_\infty$. For $\sigma =0$, $L^{\Delta, \sigma}$ coincides
with the operator given in (\ref{L1}). Clearly, for all $\sigma \in
[0,1]$, $L^{\Delta, \sigma}$ satisfies (\ref{L8}) and similar
estimates, which allows one to define bounded operators $(L^{\Delta,
\sigma})^n_{\vartheta'\vartheta}$, $n\in \mathds{N}$, and thus
construct, cf. (\ref{W18}),
\begin{equation}
  \label{Sq4}
\varXi^\sigma_{\vartheta' \vartheta} (t) = 1+\sum_{n=1}^\infty
\frac{t^n}{n!}(L^{\Delta, \sigma})^n_{\vartheta'\vartheta}, \qquad
t< T(\vartheta', \vartheta),
\end{equation}
where the latter is the same as in (\ref{W16}). Similarly, one
obtains $\widehat{L}^\sigma$ by making the aforementioned
replacements in (\ref{W13}), and then defines, cf. (\ref{W15}),
\begin{equation}
  \label{Sq5}
\varSigma^\sigma_{\vartheta \vartheta'} (t) = 1+\sum_{n=1}^\infty
\frac{t^n}{n!}(\widehat{L}^{ \sigma})^n_{\vartheta\vartheta'},
\qquad t< T(\vartheta', \vartheta).
\end{equation}
Thereafter, one sets
\begin{equation}
  \label{Sq6}
k^\sigma_t = \varXi^\sigma_{\vartheta' \vartheta} (t) k_0, \quad
G^\sigma_t = \varSigma^\sigma_{\vartheta \vartheta'} (t) G_0, \qquad
t< T(\vartheta', \vartheta),
\end{equation}
holding for each $k_0 \in \mathcal{K}_{\vartheta}$,  $G_0 \in
\mathcal{G}_{\vartheta'}$, $\vartheta' > \vartheta$. For $\sigma
=0$, both $k^\sigma_t $ and $G^\sigma_t$ coincide with those that
appear in (\ref{W17}), (\ref{W18z}), etc. Let us now consider the
Fokker-Planck equation
\begin{equation}
  \label{FPEs}
\mu_{t_2}(F) = \mu_{t_1}(F) + \int_{t_1}^{t_2} \mu_s (L^\sigma F) d
s, \qquad F\in \mathcal{D}(L),
\end{equation}
where the latter is as in Definition \ref{THdf}. Since $L^\sigma$
satisfies all the estimates used in the proof of Lemma \ref{W2lm},
see, e.g., (\ref{Sq2}), we have the following.
\begin{proposition}
\label{Sq1pn} For each $\mu\in \mathcal{P}_{\rm exp}$ and $\sigma\in
[0,1]$, the Fokker-Planck equation (\ref{FPEs}) with $\mu_0=\mu$ has
exactly one solution $t \mapsto \mu^\sigma_t \in \mathcal{P}_{\rm
exp}$ defined by the map $t \mapsto k_t^\sigma \in
\mathcal{K}^\star$ constructed with the help of (\ref{Sq4}),
(\ref{Sq6}), similarly as in the case $\sigma=0$, see Remark
\ref{W1rk}. Let also $\vartheta$ be such that $k_0 \in
\mathcal{K}_\vartheta$. Then for each $\sigma \in [0,1]$,
$\vartheta'> \vartheta$ and $G \in \mathcal{G}_{\vartheta'}$, the
following holds, see (\ref{W17a}),
\begin{equation}
  \label{Sq7}
\langle \! \langle k^\sigma_t , G \rangle \! \rangle = \langle \!
\langle k_{0} , G^\sigma_t \rangle \! \rangle, \qquad t<
T(\vartheta', \vartheta),
\end{equation}
where $G^\sigma_t = \varSigma^\sigma_{\vartheta\vartheta'}(t) G$,
see (\ref{Sq5}), (\ref{Sq6}).
\end{proposition}

\subsection{The weak convergence}
Our aim is to prove that $\mu_t^\sigma \Rightarrow \mu_t$ as $\sigma
\to 0$. We begin by proving the following statement.
\begin{proposition}
  \label{Sqqpn}
Let $\{\mu_n\}_{n\in \mathds{N}} \subset \mathcal{P}_{\rm exp}$ be
such that the type of each $\mu_n$ does not exceed  $e^\vartheta$,
$\vartheta\in \mathds{R}$, and $\mu_n \Rightarrow \mu$ for some $\mu
\in \mathcal{P}(\Gamma^2_*)$. Then $\mu\in \mathcal{P}_{\rm exp}$
and its type $\leq e^{ \vartheta}$. Moreover, for each $G\in
\mathcal{G}_{\vartheta}$, it follows that
\begin{equation}
  \label{Sq26}
\langle \! \langle k_{\mu_n} , G \rangle \! \rangle \to \langle \!
\langle k_{\mu} , G \rangle \! \rangle, \qquad n \to +\infty.
\end{equation}
\end{proposition}
\begin{proof}
Since $\widehat{F}^m_\tau (\mathbf{v}|\cdot) \in C_{\rm b}
(\Gamma^2_*)$, see Proposition \ref{TH1pn} and (\ref{de2}), the
assumed convergence yields $\mu_n(F) \to \mu(F)$, holding for all
$F\in \widehat{\mathcal{F}}$, including $F= \varPhi^m_\tau$, see
(\ref{W32a}), (\ref{W32}). Therefore, by (\ref{B5}), (\ref{B9}) we
have
\begin{equation*}
\mu(\varPhi^m_\tau) \leq \sup_{n} \mu_n(\varPhi^m_\tau) \leq
e^{|m|\vartheta} \|\theta_0\|^{m_0}_{L^1(X)}
\|\theta_1\|^{m_1}_{L^1(X)},
\end{equation*}
holding for all $\theta_0,\theta_1\in \varTheta_\psi^+$. As in the
proof of Lemma \ref{W1lm}, this yields $\mu\in \mathcal{P}_{\rm
exp}$ and its type does not exceed $e^\vartheta$. The validity of
(\ref{Sq26}) follows by the fact just mentioned.
\end{proof}

Now we prove that the solutions of the Pokker-Planck equations
(\ref{FPE}) and (\ref{FPEs}) have the property $\mu_t^\sigma
\Rightarrow \mu_t$ as $\sigma \to 0$, holding for each $t>0$. We
obtain this result by proving a bit more general statement, which
will be used in the subsequent part of this paper.
\begin{lemma}
  \label{Sqlm}
Let $\{\mu^\sigma\}_{\sigma\in (0,1]}\subset \mathcal{P}_{\rm exp}$
be such that the type of each $\mu^\sigma$ does not exceed
$e^{\vartheta_0}$ for some $\vartheta_0\in \mathds{R}$, and
$\mu^\sigma \Rightarrow \mu$ as $\sigma\to 0$. Let also $t\mapsto
\mu_t^\sigma$, $\sigma\in (0,1]$, $\mu_t^\sigma|_{t=0}=\mu^\sigma$,
be the solution of the Fokker-Planck equation (\ref{FPEs}) mentioned
in Proposition \ref{Sq1pn}. Then for each $t>0$, it follows that
$\mu_t^\sigma \Rightarrow \mu_t$ as $\sigma \to 0$, where $\mu_t$ is
the solution of (\ref{FPE}) with $\mu_t|_{t=0}=\mu$.
\end{lemma}
Noteworthy, by Proposition \ref{Sqqpn} it follows that the limiting
measure $\mu$ in Lemma \ref{Sqlm} is sub-Poissonian and its type
does not exceed $e^{\vartheta_0}$. The proof of Lemma \ref{Sqlm} is
based on the following statement.
\begin{lemma}
  \label{Sq1lm}
For a given $t>0$, let $k_t^\sigma$ and $k_t$ be the correlation
functions of the measures $\mu_t^\sigma$ and $\mu_t$ mentioned in
Lemma \ref{Sqlm}. Then there exists $\tilde{\vartheta}(t)\in
\mathds{R}$ such that
\begin{equation}
  \label{Sq8}
\forall G \in \mathcal{G}_{\tilde{\vartheta}(t)} \qquad \langle \!
\langle k^\sigma_t , G \rangle \! \rangle \to \langle \! \langle k_t
, G \rangle \! \rangle, \qquad {\rm as}  \ \ \sigma \to 0.
\end{equation}
\end{lemma}
\begin{proof}
As the type of each $\mu^\sigma$ does not exceed $e^{\vartheta_0}$,
both $k_t^\sigma$ and $k_t$ lie in $\mathcal{K}_{\vartheta(t)}$ with
$\vartheta(t) = \vartheta_0 +\alpha t$, see Remark \ref{W1rk} and
the proof of Proposition \ref{Sq1pn}. Moreover, $k_t^\sigma$ and $G$
satisfy (\ref{Sq7}) with appropriate $\vartheta, \vartheta'$. Recall
that the map $\vartheta' \mapsto T(\vartheta', \vartheta)$ attains
its maximum $T_*(\vartheta)$ given in (\ref{Ts}).

Let now the convergence stated in (\ref{Sq8}) hold for a given
$t\geq0$.  By the assumed convergence $\mu^\sigma \Rightarrow \mu$
and Proposition \ref{Sqqpn} this certainly holds for $t=0$. Our aim
is to prove that it holds also for all $t+s$, $s\leq s_0$, with a
possibly $t$-dependent $s_0>0$. Set $\bar{\vartheta}_t =
\vartheta(t) + \delta (\vartheta(t))$, see (\ref{Ts}). For $s<
T_*(\vartheta(t))$, the norm of $\varXi_{\bar{\vartheta}_t
\vartheta(t)}(s)$ satisfies
\[
\| \varXi_{\bar{\vartheta}_t \vartheta(t)}(s)\| \leq
\frac{T_*(\vartheta(t))}{T_*(\vartheta(t))-s},
\]
see (\ref{W18z}). In the same way, one estimates also the norm of
$\varXi^\sigma_{\bar{\vartheta}_t \vartheta(t)}(s)$, $\sigma\in
(0,1]$ since the norms of the  corresponding
$(L^{\Delta,\sigma})^n_{\bar{\vartheta}_t \vartheta(t)}$ have the
same bounds as for $\sigma=0$. For $\sigma\in (0,1]$, we write
\begin{equation}
  \label{Sq9}
q^\sigma_s = k_{t+s} - k^\sigma_{t+s} = \varXi_{\bar{\vartheta}_t
\vartheta(t)}(s) k_t - \varXi^\sigma_{\bar{\vartheta}_t
\vartheta(t)}(s) k^\sigma_t, \qquad s < T_*(\vartheta(t)).
\end{equation}
Note that
\begin{equation}
  \label{Sq9a}
\forall G\in \mathcal{G}_{\vartheta(t)} \qquad \langle \! \langle
q^\sigma_0, G\rangle\! \rangle \to 0 \qquad {\rm as} \ \ \sigma\to
0.
\end{equation}
At the same time, (\ref{Sq9}) can be written in the form
\begin{eqnarray}
  \label{Sq10}
q^\sigma_s & = &\varXi_{\bar{\vartheta}_t \vartheta(t)}(s)q^\sigma_0
- \varPi^\sigma _{\bar{\vartheta}_t \vartheta(t)}(s) k^\sigma_t,
\\[.2cm] \nonumber \varPi^\sigma _{\bar{\vartheta}_t
\vartheta(t)}(s) & := & \int_0^s \frac{d}{du} \left[
\varXi_{\bar{\vartheta}_t \vartheta}(s-u)\varXi^\sigma_{{\vartheta}
\vartheta(t)}(u) \right] du,
\end{eqnarray}
where $s$ and $\vartheta \in (\vartheta(t), \bar{\vartheta}_t)$ are
chosen in such a way that
\begin{equation}
  \label{Sq11}
s < \min\{T(\bar{\vartheta}_t, \vartheta); T({\vartheta},
\vartheta(t)) \},
\end{equation}
and hence the Bochner integral in the second line of (\ref{Sq10})
makes sense, see (\ref{Sq4}). Since the map $(\vartheta, \vartheta')
\mapsto T(\vartheta', \vartheta)$ is continuous, see (\ref{W16}),
one can pick $\vartheta_1< \vartheta$ and $\vartheta_2>\vartheta$,
$\vartheta$ being as in (\ref{Sq11}), such that
\begin{equation}
  \label{Sq11a}
s < \min\{T(\bar{\vartheta}_t, \vartheta_2); T({\vartheta}_1,
\vartheta(t)) \}.
\end{equation}
Keeping this in mind, we use an evident identical extension of
(\ref{W18Z}) to all $\sigma\leq 1$ and obtain
\begin{gather}
  \label{Sq12}
\varPi^\sigma _{\bar{\vartheta}_t \vartheta(t)}(s) = - \int_0^s
\varXi_{\bar{\vartheta}_t \vartheta_2}(s-u) \widetilde{L}^{\Delta,
\sigma}_{\vartheta_2\vartheta_1}
\varXi^\sigma_{{\vartheta}_1 \vartheta(t)}(u) d u,\\[.2cm]
\nonumber \widetilde{L}^{\Delta, \sigma}_{\vartheta_2\vartheta_1} :=
{L}^{\Delta}_{\vartheta_2\vartheta_1} - {L}^{\Delta,
\sigma}_{\vartheta_2\vartheta_1}.
\end{gather}
We apply this in (\ref{Sq10}) and get
\begin{equation*}
q^\sigma_s = \varXi_{\bar{\vartheta}_t \vartheta(t)}(s)q^\sigma_0
+\int_0^s \varXi_{\bar{\vartheta}_t \vartheta_2}(s-u)
\widetilde{L}^{\Delta, \sigma}_{\vartheta_2\vartheta_1}
k_{t+u}^\sigma  d u.
\end{equation*}
Note that $\widetilde{L}^{\Delta, \sigma}$ can be written in the
same form as $L^\Delta$, see (\ref{L1}), in which $a_i(x-y)$,
$i=0,1$, ought to be replaced by $\tilde{a}^\sigma_i(x,y) :=
a_i(x-y)(1- \psi_\sigma(x))$. Now let us turn to picking $s_0$ and
$\vartheta_j$, $j=1,2$, such that (\ref{Sq11a}) holds for $s\leq
s_0$. First we set $\vartheta_1 = \vartheta(t) +
\delta(\vartheta(t))/2$, see (\ref{Ts}). By (\ref{W16}) and
(\ref{Ts}) we then get
\begin{equation*}
T(\bar{\vartheta}_t, \vartheta_1) = T_*(\vartheta(t))/2 <
T(\vartheta_1, \vartheta(t)).
\end{equation*}
Now we fix some $\epsilon \in (0,1)$ and set
\begin{equation}
  \label{Sq14}
s_0 = \epsilon T_* (\vartheta(t))/2 = \epsilon T(\bar{\vartheta}_t,
\vartheta_1).
\end{equation}
Since the map $\vartheta \mapsto T(\vartheta', \vartheta)$ is
continuous, one can pick $\vartheta_2 \in (\vartheta_1,
\bar{\vartheta}_t)$ such that $s_0 < T(\bar{\vartheta}_t,
\vartheta_2)$, see (\ref{Sq14}). Then (\ref{Sq11a}) holds for these
$\vartheta_j$, $j=1,2$, and $s\leq s_0$. Now we take $G\in
\mathcal{G}_{\bar{\vartheta}_t}$ and set
\begin{equation}
  \label{Sq14a}
G_s = \varSigma_{\vartheta_2\bar{\vartheta}_t}(s)G.
\end{equation}
Note that $G_s \in \mathcal{G}_{\vartheta_2}\subset
\mathcal{G}_{\vartheta(t)}$; that is, $G_s$ can be considered as an
element of $\mathcal{G}_{\vartheta(t)}$ since
\[
G_s = I_{\vartheta(t) \vartheta_2} \varSigma_{\vartheta_2
\bar{\vartheta}_t} (s) G ,
\]
where $I_{\vartheta(t) \vartheta_2}= \varSigma_{\vartheta(t)
\vartheta_2} (0)$ is the embedding operator. For these $G$ and
$G_s$, by (\ref{Sq11a}) and (\ref{Sq7}) we then have
\begin{gather}
  \label{Sq15}
\langle \! \langle q_s^\sigma , G \rangle \! \rangle = \langle \!
\langle q_0^\sigma , G_s \rangle \! \rangle + R^\sigma (s), \\[.2cm]
\nonumber R^\sigma (s):= \int_0^s \langle \! \langle
\widetilde{L}^{\Delta, \sigma}_{\vartheta_2\vartheta_1}
k^\sigma_{t+u} , G_{s-u} \rangle \! \rangle d u.
\end{gather}
In view of (\ref{Sq9a}), it remains to prove that $R^\sigma (s)\to
0$ as $\sigma\to 0$. To this end, we split $R^\sigma (s)$ into four
terms in accord with the structure of
$\widetilde{L}^{\Delta,\sigma}$, see (\ref{L1}). Thus, we write
\begin{gather}
  \label{Sq16}
R^\sigma (s)= \sum_{j=1}^4 R^\sigma_j (s),
\end{gather}
with
\begin{eqnarray}
  \label{Sq17}
R^\sigma_1 (s)& = & \int_0^s
\bigg{(}\int_{\Gamma_0}\int_{\Gamma_0}\bigg{(} \sum_{y\in \eta_0}
\int_X \tilde{a}^\sigma_0 (x,y) e(\tau^0_y;\eta_1) (\Upsilon^0_y
k_{t+u}^\sigma) (\eta_0\setminus y \cup x, \eta_1) dx
\bigg{)}\\[.2cm] \nonumber &\times&  G_{s-u} (\eta_0, \eta_1) \lambda ( d\eta_0) \lambda (d
\eta_1)\bigg{)} d u,
\\[.2cm] \nonumber
& = & \int_0^s \bigg{(} \int_{\Gamma_0}\int_{\Gamma_0} \bigg{(}
\int_X \int_X  \tilde{a}^\sigma_0 (x,y) e(\tau^0_y;\eta_1)
(\Upsilon^0_y k_{t+u}^\sigma) (\eta_0 \cup x, \eta_1)
\\[.2cm] \nonumber &\times&  G_{s-u} (\eta_0\cup y, \eta_1) d x dy \bigg{)} \lambda ( d\eta_0) \lambda (d
\eta_1) \bigg{)} du
\\[.2cm] \nonumber
 R^\sigma_2 (s)& = - & \int_0^s \bigg{(} \int_{\Gamma_0}\int_{\Gamma_0}\bigg{(}
\int_X \int_X \tilde{a}^\sigma_0 (x,y) e(\tau^0_y;\eta_1)
(\Upsilon^0_y k_{t+u}^\sigma) (\eta_0\cup x, \eta_1) \\[.2cm] \nonumber &\times&  G_{s-u} (\eta_0\cup x, \eta_1)
dx dy \bigg{)}
 \lambda ( d\eta_0) \lambda (d
\eta_1)\bigg{)}du,
\\[.2cm] \nonumber
R^\sigma_3 (s)& = & \int_0^s \bigg{(}
\int_{\Gamma_0}\int_{\Gamma_0}\bigg{(} \int_X \int_X
\tilde{a}^\sigma_1 (x,y) e(\tau^1_y;\eta_0) (\Upsilon^1_y
k_{t+u}^\sigma) (\eta_0, \eta_1 \cup x) \\[.2cm] \nonumber &\times&  G_{s-u} (\eta_0, \eta_1\cup
y) dx dy \bigg{)}
 \lambda ( d\eta_0) \lambda (d
\eta_1)\bigg{)} du,
\\[.2cm] \nonumber
R^\sigma_4 (s)& = - & \int_0^s \bigg{(}
\int_{\Gamma_0}\int_{\Gamma_0}\bigg{(} \int_X \int_X
\tilde{a}^\sigma_1 (x,y) e(\tau^1_y;\eta_1)
(\Upsilon^1_y k_{t+u}^\sigma) (\eta_0, \eta_1)\\[.2cm] \nonumber &\times&  G_{s-u} (\eta_0, \eta_1\cup x) dx dy \bigg{)}
  \lambda ( d\eta_0) \lambda (d
\eta_1) \bigg{)} du.
\end{eqnarray}
By (\ref{L5}), (\ref{W8}), (\ref{W9}) and (\ref{C7a}) for each
$\vartheta\in \mathds{R}$, $i=0,1$, $s\geq 0$ and
$(\eta_0,\eta_1)\in \Gamma_0^2$, we have
\begin{gather}
  \label{Sq18}
\left|(\Upsilon^i_y k^\sigma_{s}) (\eta_0,\eta_1)\right| \leq
\|k_s^\sigma\|_\vartheta \exp\bigg{(}\vartheta (|\eta_0| + |\eta_1|)
\bigg{)} \int_{\Gamma_0} e^{\vartheta|\xi|} e\left(|t^i_y|; \xi
\right)\lambda (d \xi) \\[.2cm] \nonumber = \|k_s^\sigma\|_\vartheta \exp\bigg{(}\vartheta (|\eta_0| + |\eta_1|)
\bigg{)} \sum_{n=0}^\infty \frac{1}{n!} e^{n\vartheta} \left( \int_X
\left[1- e^{-\phi_i (x-y)}\right] d x\right)^n \\[.2cm] \nonumber
\leq \|k_s^\sigma\|_\vartheta \exp\bigg{(}\vartheta (|\eta_0| +
|\eta_1|) + \varphi e^{\vartheta} \bigg{)}.
\end{gather}
By (\ref{Sq12}) we know that $k^\sigma_{t+s}\in
\mathcal{K}^\star_{\vartheta(t+u)} \subset
\mathcal{K}^\star_{\vartheta_1}$, which by (\ref{no9}) implies
$\|k^\sigma_{t+s}\|_{\vartheta_1} \leq 1$. We take this into account
in (\ref{Sq17}), and also that $\tau_y^i(x) \leq 1$, see (\ref{W8}),
and then estimate the summands in (\ref{Sq16}) as follows
\begin{equation}
  \label{Sq19}
\left|R^\sigma_j (s) \right| \leq \int_X r^\sigma_j (y) g_j(y) dy,
\end{equation}
with
\begin{gather}
  \label{Sq20}
r^\sigma_1 (y) = \int_X (1 - \psi_\sigma (x)) a_0(x-y) d x, \quad
r^\sigma_3 (y) = \int_X (1 - \psi_\sigma (x)) a_1(x-y) d x,\\[.2cm]
\nonumber r^\sigma_2 (y) = (1-\psi_\sigma (y)) \bar{a}^{0}_0, \qquad
r^\sigma_4 (y) = (1-\psi_\sigma (y)) \bar{a}^{0}_1,
\end{gather}
see (\ref{C8}). It is clear that $r^\sigma_j (y) \leq r^1_j (y)$,
$j=1,\dots, 4$, and
\begin{equation}
  \label{Sq201}
\forall y \in X \qquad r^\sigma_j (y) \to 0 \quad \sigma \to 0, \ \
=1,\dots, 4.
\end{equation}
Furthermore,
\begin{gather}
  \label{Sq22}
g_1 (y) = g_2 (y) = c(\vartheta_1)\int_0^s \int_{\Gamma_0}
\int_{\Gamma_0} \left|G_{u}(\eta_0 \cup y, \eta_1)\right|
e^{\vartheta_1(|\eta_0| +
|\eta_1|)}  \lambda (d\eta_0) \lambda (d\eta_1) d u, \\[.2cm]
\nonumber g_3 (y) = g_4 (y) = c(\vartheta_1)\int_0^s \int_{\Gamma_0}
\int_{\Gamma_0} \left|G_{u}(\eta_0, \eta_1 \cup y)\right|
e^{\vartheta_1(|\eta_0| + |\eta_1|)} \lambda (d\eta_0) \lambda
(d\eta_1) d u,
\end{gather}
where $c(\vartheta_1) = \exp( \vartheta_1 + \varphi
e^{\vartheta_1})$. Let us show that each $g_j$, $j=1, \dots , 4$, is
integrable for all $s\leq s_0$. Since $G_u \in
\mathcal{G}_{\vartheta_2}$, by (\ref{W18z}) and (\ref{Sq14a}) for
$u\leq s \leq s_0$, see (\ref{Sq14}), we have
\begin{equation}
  \label{Sq23}
|G_u|_{\vartheta_2} \leq
\frac{T(\bar{\vartheta}_t,\vartheta_2)}{T(\bar{\vartheta}_t,\vartheta_2)-s_0}
|G|_{\bar{\vartheta}_t}=: C_G.
\end{equation}
By (\ref{Sq22}) we then have
\begin{eqnarray}
  \label{Sq24}
\int_X g_1(y) dy & \leq &c(\vartheta_1) \int_0^s
\left(\int_{\Gamma_0} \int_{\Gamma_0} \int_X  \left|G_{u}(\eta_0
\cup y, \eta_1)\right| e^{\vartheta_1(|\eta_0| + |\eta_1|)}  \lambda
(d\eta_0) \lambda
(d\eta_1) dy \right)d u \qquad \\[.2cm] \nonumber & = & c(\vartheta_1)
e^{-\vartheta_1} \int_0^s \left(\int_{\Gamma_0} \int_{\Gamma_0}
|\eta_0| \left|G_{u}(\eta_0 , \eta_1)\right| e^{\vartheta_1(|\eta_0|
+ |\eta_1|)}  \lambda (d\eta_0) \lambda
(d\eta_1) \right)d u \\[.2cm] \nonumber & \leq &
\frac{s c(\vartheta_1)}{(\vartheta_2 - \vartheta_1)
e^{1+\vartheta_1}} C_G.
\end{eqnarray}
Clearly, the same estimate holds for the remaining $g_j$. Then by
the dominated convergence theorem and (\ref{Sq201}), (\ref{Sq19}) it
follows that $R^\sigma(s) \to 0$ as $\sigma\to 0$, holding for all
$s\leq s_0$, see (\ref{Sq14}). By (\ref{Sq9a}) and (\ref{Sq15}) this
yields $\langle \! \langle q^\sigma_s, G \rangle \! \rangle \to 0$
as $\sigma\to 0$.

To complete the proof of this statement, let us consider the
following sequences, cf. (\ref{Sq14}),
\begin{equation}
  \label{Sq25}
t_l = t_{l-1} + s_{0l}, \quad s_{0l} = \epsilon
T_*(\vartheta_{t_{l-1}})/2, \quad t_0 = 0, \quad l \in \mathds{N}.
\end{equation}
Now we may use the construction just made and the induction in $l$,
which yields (\ref{Sq8}) holding for all $t\leq t_l$. Thus, the
proof will follow if we show $t_l\to +\infty$ as $l\to +\infty$. Set
$\sup_{l}t_l =:t_*$. By (\ref{Sq25}) we have $t_l = s_{01} +\cdots +
s_{0l}$. Hence, $t_*<\infty$ yields $s_{0l}\to 0$, $l\to+\infty$. By
passing to the limit in the second formula in (\ref{Sq25}) we then
get $T_*(\vartheta_{t_*})=0$, which is impossible, see (\ref{Ts}).
\end{proof}
\noindent {\it Proof of Lemma \ref{Sqlm}.} By Lemma \ref{Sq1lm} it
follows that $\mu_t^\sigma (F) \to \mu_t(F)$, $\sigma  \to 0$,
holding for all $F\in \mathcal{F}_\infty$,  which by Proposition
\ref{V1pn} yields that $\mu_t^\sigma (F) \to \mu_t(F)$, $\sigma  \to
0$, for all $F\in \widetilde{\mathcal{F}}$. Then the property in
question follows by claim (ii) of Proposition \ref{T1pn}. \hfill
$\square$ \\ \noindent We end up this subsection with the following
complement to Lemma \ref{Sqlm}. For $F\in \widetilde{\mathcal{F}}$,
see (\ref{T4}), and a sequence $\{\mu_n \}_{n\in \mathds{N}}\subset
\mathcal{P}_{\rm exp}$ as in Proposition \ref{Sqqpn}, consider
\begin{equation}
  \label{Sq28}
\tilde{\mu}_n (d \gamma) = C^{-1}_n F(\gamma) \mu_n (d \gamma),
\qquad n \in \mathds{N}.
\end{equation}
where
\begin{equation}
  \label{Sq29}
C_n= \mu_n(F) >0,
\end{equation}
since each $F\in \widetilde{\mathcal{F}}$ is strictly positive.
\begin{proposition}
  \label{SQpn}
Let $\tilde{\mu}_n$ and $\mu_n$ be as in (\ref{Sq28}) and assume
that $\mu_n \Rightarrow \mu$ as $n \to +\infty$. Then $\tilde{\mu}_n
\Rightarrow \tilde{\mu}$, where
\[
\tilde{\mu} (d \gamma) = C^{-1} F(\gamma) \mu (d \gamma), \qquad C =
\mu(F).
\]
\end{proposition}
\begin{proof}
By assumption, $C_n \to C$. Take any $F'\in\widetilde{ \mathcal{F}}$
and set $F'' = F' F$, which is an element of $\widetilde{
\mathcal{F}}$ since the latter is closed under multiplication, see
Proposition \ref{T1pn}. Then $\tilde{\mu}_n (F') = C_n^{-1} \mu_n
(F'') \to C^{-1} \mu(F'')$ as $n\to +\infty$. Since $\widetilde{
\mathcal{F}}$ is convergence determining, see claim (ii) of
Proposition \ref{T1pn}, the sequence $\{\tilde{\mu}_n \}_{n\in
\mathds{N}}$ converges to some $\tilde{\mu}\in \mathcal{P}_{\rm
exp}$ (by Proposition \ref{Sqqpn}), such that $\tilde{\mu}(F') =
C^{-1} \mu(F'')$. This implies that $\tilde{\mu}$ is as stated since
$\widetilde{ \mathcal{F}}$ is separating.
\end{proof}

\section{Existence: Approximating Processes}
The aim of this section is proving Theorem \ref{1tm} by constructing
path measures for the model described by $L^\sigma$ introduced in
the preceding section. This will be done in a direct way by means of
the corresponding Markov transition functions.

\subsection{The transition function}

We start by introducing the real linear space of signed measures on
$\Gamma^2_*$, see \cite[Chapter 4]{Cohn}, which we denote by
$\mathcal{M}$. That is, each $\mu\in \mathcal{M}$ is a
$\sigma$-additive map $\mu:\mathcal{B}(\Gamma^2_*)\to \mathds{R}$
taking finite values only. Let $\mathcal{M}^{+}$ be the set of
$\mu\in \mathcal{M}$ such that $\mu(\mathbb{A}) \geq 0$ for all
$\mathbb{A} \in \mathcal{B}(\Gamma^2_*)$. Then the Jordan
decomposition of a given $\mu\in \mathcal{M}$ is the unique
representation $\mu=\mu^{+} - \mu^{-}$, $\mu^{\pm}\in
\mathcal{M}^{+}$, in view of which  the cone $\mathcal{M}^{+}$ is
generating. Set $|\mu| = \mu^{+}  +  \mu^{-}$. Then
\begin{equation}
  \label{U}
\|\mu\| = |\mu|(\Gamma^2_*)
\end{equation}
is a norm, additive on the cone $\mathcal{M}^{+}$. According to
\cite[Proposition 4.1.8, page 119]{Cohn}, $\mathcal{M}$ is a Banach
space with this norm. Set $\Psi_1 = 1 + \Psi$, where the latter was
defined in (\ref{no20}), and then define
\begin{equation}
  \label{U1}
\mathcal{M}_n = \{ \mu \in\mathcal{M}: \|\mu\|_n:=|\mu|(\Psi_1^n) <
\infty\}, \qquad n\in \mathds{N}.
\end{equation}
By the same \cite[Proposition 4.1.8, page 119]{Cohn} $\mathcal{M}_n$
with the norm $\|\mu\|_n$ is also a real Banach space. In the
sequel, we extend (\ref{U1}) to $n=0$ by setting
$\mathcal{M}_0=\mathcal{M}$ and $\|\mu\|_0=\|\mu\|$. Additionally,
for $n\in \mathds{N}_0$, we set
\begin{equation}
\label{U2} \varphi_n (\mu) = \mu(\Psi_1^n).
\end{equation}
Now for $\beta>0$, define
\begin{equation}
  \label{U3}
\|\mu\|_\beta = \int_{\Gamma^2_*}\exp\left(\beta \Psi(\gamma)
\right) |\mu|(d\gamma), \quad  \mathcal{M}_\beta =\{ \mu \in
\mathcal{M}: \|\mu\|_\beta<\infty\},
\end{equation}
and also
\begin{equation}
  \label{U4}
\varphi_\beta (\mu) = \int_{\Gamma^2_*}\exp\left(\beta \Psi(\gamma)
\right) \mu(d\gamma).
\end{equation}
It is clear that
\begin{equation}
  \label{U5}
\forall \mu \in \mathcal{M}_{+} \qquad \qquad \|\mu\|_n =
\varphi_n(\mu), \qquad \|\mu\|_\beta = \varphi^\beta(\mu),
\end{equation}
holding for all $n\in \mathds{N}_0$ and $\beta >0$. In our
construction, we essentially use the cones of positive elements
\begin{equation}
  \label{cones}
\mathcal{M}_n^{+} = \mathcal{M}_n\cap \mathcal{M}^{+}, \qquad
\mathcal{M}_\beta^{+} = \mathcal{M}_\beta\cap \mathcal{M}^{+},
\qquad \beta >0, \ n\in \mathds{N}.
\end{equation}
For a given $\mathcal{N}\subset\mathcal{M}$, by
$\overline{\mathcal{N}}$ we denote  the closure of $\mathcal{N}$ in
$\|\cdot\|$ defined in (\ref{U}). The proof of the next statement is
completely analogous to that of \cite[Lemma 7.4 and Corollary 7.5
pages 39, 40]{KR}, and thus is omitted here.
\begin{proposition}
  \label{U1pn}
For each $n \in \mathds{N}$ and $\beta> 0$, it follows that
$\overline{\mathcal{M}_\beta} = \overline{\mathcal{M}_n} =
\mathcal{M}$ and also $\overline{\mathcal{M}^{+}_\beta} =
\overline{\mathcal{M}^{+}_n} = \mathcal{M}^{+}$.
\end{proposition}
Finally, we denote $\mathcal{M}^{+,1}= \mathcal{P}(\Gamma^2_*)$ and
also
\begin{equation}
  \label{U6}
\mathcal{M}^{+,1}_n = \mathcal{M}^{+,1}\cap \mathcal{M}_\beta,
\qquad \mathcal{M}^{+,1}_n = \mathcal{M}^{+,1}\cap \mathcal{M}_n.
\end{equation}
By (\ref{W4}) it follows that
\begin{equation}
  \label{U7}
\forall \beta >0 \ \  \forall n \in \mathds{N} \qquad \qquad
\mathcal{P}_{\rm exp}\subset \mathcal{M}_\beta^{+,1} \subset
\mathcal{M}_n^{+,1}.
\end{equation}
Now for $\sigma\in (0,1]$, we set, cf. (\ref{Lsi}),
\begin{eqnarray*}
\varPsi^\sigma(\gamma) & = & \sum_{x\in \gamma_0} \int_X a^\sigma_0
(x,y) \exp\left( - \sum_{z\in \gamma_1} \phi_0(z-y)  \right) dy
\\[.2cm] \nonumber & + & \sum_{x\in \gamma_1} \int_X a^\sigma_1
(x,y) \exp\left( - \sum_{z\in \gamma_0} \phi_1(z-y)  \right) dy,
\qquad \gamma\in \Gamma^2_*.
\end{eqnarray*}
By (\ref{P}) and (\ref{Sq}) it follows that $\psi(x) \leq
\psi_\sigma (x) \leq \psi(x)/\sigma$, $\sigma \in (0,1]$. For these
values of $\sigma$, by (\ref{no20}) and (\ref{C8}), (\ref{L2}) we
then have
\begin{equation}
  \label{U9}
 \varPsi^\sigma(\gamma) \leq  (\alpha/\sigma)\varPsi(\gamma), \qquad \gamma\in
\Gamma^2_*.
\end{equation}
As mentioned above, the transition function in question will be
constructed directly, i.e., by the formula
\begin{equation}
  \label{U10}
p^\sigma_t (\gamma, \cdot) = S^\sigma (t) \delta_\gamma (\cdot),
\qquad t>0, \ \ \gamma\in \Gamma^2_*,
\end{equation}
where $\delta_\gamma$ is the Dirac measure centered at $\gamma$ and
$S^\sigma=\{S^\sigma(t)\}_{t\geq 0}$ is the stochastic semigroup of
bounded linear operators acting in $\mathcal{M}$, generated by the
dual $L^{\dagger,\sigma}$ of $L^\sigma$ defined in (\ref{Lsi}). The
mentioned duality means that
\begin{equation}
  \label{U11}
\mu(L^\sigma F) = (L^{\dagger,\sigma}\mu) (F), \qquad F\in
\mathcal{D}(L).
\end{equation}
Recall that the domains of all $L^\sigma$, $\sigma\in [0,1]$ are the
same, i.e., are as in Definition \ref{THdf}. By (\ref{Lsi}) and
(\ref{Sq1}) we then get
\begin{eqnarray}
  \label{U12}
(L^{\dagger,\sigma}\mu) (\mathbb{A}) & = & - \int_{\Gamma_*^2}
\mathds{1}_{\mathbb{A}} (\gamma) \varPsi^\sigma (\gamma) \mu(d
\gamma) + \int_{\Gamma^2_*} \varOmega^\sigma(\mathbb{A}|\gamma)
\mu(d\gamma)
\\[.2cm] \nonumber  & =: & (A \mu)(\mathbb{A}) + (B
\mu)(\mathbb{A}), \quad \mathbb{A}\in \mathcal{B}(\Gamma^2_*).
\end{eqnarray}
Here
\begin{eqnarray}
  \label{U13}
\varOmega^\sigma(\mathbb{A}|\gamma) & = & \sum_{x\in \gamma_0}
\int_X a^\sigma_0 (x,y) \exp\left(-\sum_{z\in \gamma_1}
\psi_0(z-y)\right) \mathds{1}_{\mathbb{A}}(\gamma_0\setminus x
\cup_0
y, \gamma_1) dy \\[.2cm] \nonumber & + &  \sum_{x\in \gamma_1}
\int_X a^\sigma_1 (x,y) \exp\left(-\sum_{z\in \gamma_0}
\psi_1(z-y)\right) \mathds{1}_{\mathbb{A}}(\gamma_0,
\gamma_1\setminus x \cup_1 y) dy.
\end{eqnarray}
Note that $A$ in (\ref{U12}) is just the multiplication operator by
$\varPsi^\sigma$, and the following holds
\begin{equation}
  \label{U13a}
\varOmega^\sigma(\Gamma^2_*|\gamma) = \varPsi^\sigma(\gamma).
\end{equation}
Now we set
\begin{equation}
  \label{U14}
\mathcal{D}(L^{\dagger,\sigma}) = \{\mu \in \mathcal{M}:
|\mu|(\varPsi^\sigma)<\infty\},
\end{equation}
which might have sense if we show that $B$ can act on $\mu \in
\mathcal{D}(L^{\dagger,\sigma})$. By writing $\mu=\mu^{+} -\mu^{-}$
we conclude that it is enough to show $B\mu \in \mathcal{M}$ for
positive $\mu \in \mathcal{D}(L^{\dagger,\sigma})$ only. Since $B$
itself is positive, by (\ref{U5}) and (\ref{U13a}) we have that
\begin{gather}
  \label{U14a}
\|B\mu\| = (B\mu)(\Gamma^2_*) = \int_{\Gamma^2_*} \varPsi^\sigma
(\gamma) \mu ( d\gamma) = \|A\mu\|,
\end{gather}
which yields $L^{\dagger,\sigma}: \mathcal{D}(L^{\dagger,\sigma})\to
\mathcal{M}$. Clearly, $(A,\mathcal{D}( L^{\dagger,\sigma}))$ is
closed and the following holds
\begin{equation}
  \label{U15}
\mathcal{M}_1 \subset \mathcal{D}(L^{\dagger,\sigma}),
\end{equation}
see (\ref{U9}) and (\ref{U1}).
\begin{remark}
  \label{Urk}
Note that $\delta_\gamma \in \mathcal{D}(L^{\dagger,\sigma})$, since
$\delta_\gamma (\varPsi^\sigma) = \varPsi^\sigma(\gamma) <\infty$,
holding for all $\gamma\in \Gamma^2_*$, see (\ref{U9}) and
(\ref{n021z}). At the same time, $\delta_\gamma$ is evidently not
sub-Poissonian.
\end{remark}
Along with constructing the semigroup $S^\sigma$, see (\ref{U10}),
in Lemma \ref{U1lm} below we obtain a number of complementary
results, which we then exploit for proving Theorem \ref{1tm}. To
this end, for $n \in \mathds{N}$ and a positive $\mu$, let us
consider, cf. (\ref{U2}),
\begin{eqnarray}
  \label{U16}
\varphi_n(B\mu) & = & \int_{\Gamma^2_*} \varPsi_1^n(\gamma) (B\mu)
(d \gamma) = \int_{\Gamma^2_*} \int_{\Gamma^2_*}
\varPsi_1^n(\gamma)\varOmega^\sigma (d \gamma|\gamma') \mu (d
\gamma')
\\[.2cm] \nonumber & = & \int_{\Gamma^2_*} \bigg{(} \sum_{x\in
\gamma_0} \int_X a_0^\sigma (x,y) \exp\left(-\sum_{z\in \gamma_1}
\phi_0(z-y) \right) \varPsi_1^n (\gamma \setminus x\cup_0 y)
dy\bigg{)}
\mu(d\gamma) \\[.2cm] \nonumber & + & \int_{\Gamma^2_*} \bigg{(} \sum_{x\in
\gamma_1} \int_X a_1^\sigma (x,y) \exp\left(-\sum_{z\in \gamma_0}
\phi_1(z-y) \right) \varPsi_1^n (\gamma \setminus x\cup_1 y) dy
\bigg{)} \mu(d\gamma).
\end{eqnarray}
Since $\varPsi_1(\gamma\setminus x\cup_i y) = 1 +
\varPsi(\gamma\setminus x\cup_i y) = \varPsi_1(\gamma) + \psi(y) -
\psi(x)$, $i=0,1$, see (\ref{no20}), then
\begin{equation}
  \label{U17}
\varPsi^n_1 (\gamma\setminus x\cup_i y) \leq 2^n
\varPsi^n_1(\gamma),
\end{equation}
which by (\ref{U9}) yields in (\ref{U16}) the following estimate
\begin{equation}
  \label{U18}
\forall \mu \in \mathcal{M}_{n+1}^{+} \qquad  \|B\mu\|_{n} =
\varphi_n(B\mu) \leq 2^n \alpha \sigma^{-1} \|\mu\|_{n+1},
\end{equation}
and hence
\begin{equation}
  \label{U19}
\forall n\in \mathds{N}_0 \qquad B:\mathcal{M}_{n+1}^{+} \to
\mathcal{M}_{n}^{+}.
\end{equation}
In a similar way, one shows that $\|A\mu\|_n \leq (\alpha/ \sigma)
\|\mu\|_{n+1}$ and
\begin{equation}
\label{U20}
 -A:\mathcal{M}_{n+1}^{+} \to \mathcal{M}_{n}^{+},
\end{equation}
which finally yields that $L^{\dagger,\sigma}:\mathcal{M}_{n+1}^{+}
\to \mathcal{M}_{n}^{+}$, holding for all $n\in \mathds{N}_0$. By
means of (\ref{U18}) and the corresponding estimate for $A$ we then
define bounded linear operators
\begin{equation}
  \label{U21}
(L^{\dagger,\sigma})^l_{n,n+l}: \mathcal{M}_n \to \mathcal{M}_{n+l},
\qquad l\in \mathds{N},
\end{equation}
the norms of which satisfy
\begin{equation}
  \label{U22}
\|(L^{\dagger,\sigma})^l_{n,n+l}\| \leq \left(
\frac{\alpha}{\sigma}\right)^l (2^n +1) (2^{n+1}+1) \cdots
(2^{n+l-1}+1).
\end{equation}
Next, similarly as in (\ref{W14}) we also define bounded operators
$(L^{\dagger,\sigma}_{\beta',\beta})^l: \mathcal{M}_\beta \to
\mathcal{M}_{\beta'}$, $\beta>\beta' >0$, see (\ref{U3}). To
estimate their norms, for a given $\mu=\mu^{+}-\mu^{-}\in
\mathcal{D}(L^{\dagger,\sigma})$, we write
\begin{gather*}
L^{\dagger,\sigma} \mu = (A + B) (\mu^{+} - \mu^{-}) =
(B\mu^{+}-A\mu^{-}) - (B\mu^{-}-A\mu^{+}) =:\mu_1^{+} - \mu_1^{-}.
\end{gather*}
It is clear that $\mu_1^{\pm}\in \mathcal{M}^{+}$. Then
\begin{equation}
  \label{U24}
\|L^{\dagger,\sigma} \mu\|_{\beta'} \leq \|\mu_1^{+}\|_{\beta'} +
\|\mu_1^{-}\|_{\beta'} = \|A\mu^{+}\|_{\beta'} +
\|A\mu^{-}\|_{\beta'} + \|B\mu^{+}\|_{\beta'} +
\|B\mu^{-}\|_{\beta'}.
\end{equation}
Here we used the additivity of the norms on the positive cone as
well as the positivity of $B$ and $-A$, see (\ref{U19}),
(\ref{U20}). Now by (\ref{U9}) we have
\[
\varPsi^\sigma (\gamma) \exp\left(\beta'\varPsi(\gamma) \right) \leq
\frac{\alpha}{e \sigma (\beta-\beta')}\exp
\left(\beta\varPsi(\gamma) \right),
\]
which for $\mu\in \mathcal{M}^{+}_\beta$ yields
\begin{equation}
  \label{U25}
\|A\mu\|_{\beta'} \leq \frac{\alpha}{e \sigma (\beta-\beta')}
\|\mu\|_\beta.
\end{equation}
Next, similarly as in (\ref{U16}), (\ref{U18}) by (\ref{U9}) we have
\begin{eqnarray*}
& & \int_{\Gamma^2_*} \exp\left(\beta'\varPsi(\gamma)\right)
(B\mu)(d\gamma) = \int_{\Gamma^2_*} \int_{\Gamma^2_*}
\exp\left(\beta'\varPsi(\gamma')\right) \varOmega^\beta(d
\gamma'|\gamma) \mu (d \gamma) \\[.2cm] \nonumber & & \quad =
\int_{\Gamma^2_*}\exp\left(\beta'\varPsi(\gamma)\right)\bigg{(}
\sum_{x\in \gamma_0}\int_X a_0^\sigma(x,y) \exp\left(-\sum_{z\in
\gamma_1}\phi_0(z-y) + \beta'(\psi(y) - \psi(x))\right) dy \\[.2cm]
\nonumber & &\quad  + \sum_{x\in \gamma_1}\int_X a_1^\sigma(x,y)
\exp\left(-\sum_{z\in \gamma_0}\phi_1(z-y) + \beta'(\psi(y) -
\psi(x))\right) dy \bigg{)} \mu(d\gamma) \\[.2cm]
\nonumber & & \quad \leq e^{\beta'} \int_{\Gamma^2_*}
\varPsi^\sigma(\gamma) \exp\left(\beta'\varPsi(\gamma)\right) \mu(d
\gamma) \leq \frac{e^{\beta'} \alpha}{\sigma e(\beta -
\beta')}\|\mu\|_\beta, \qquad \mu\in \mathcal{M}_\beta^{+}.
\end{eqnarray*}
We combine now this estimate with (\ref{U25}) and obtain in
(\ref{U24}) the following, cf. (\ref{W14})
\begin{equation}
  \label{U27}
\|(L^{\dagger,\sigma}_{\beta',\beta})^l\| \leq \left(\frac{l}{e
T_\sigma(\beta, \beta')}\right)^l, \qquad l\in \mathds{N},
\end{equation}
with
\begin{equation}
  \label{U28}
T_\sigma(\beta, \beta') = \frac{\sigma(\beta-\beta')}{\alpha
e^\beta}.
\end{equation}
By (\ref{U21})), for each $l\in \mathds{N}$ and $\mu\in
\mathcal{M}_\beta$, we have that $(L^{\dagger, \sigma})^l\mu
\in\mathcal{M}_{\beta'}$, $\beta'<\beta$, and the following holds
\begin{equation}
  \label{U29}
(L^{\dagger, \sigma}_{\beta',\beta})^l\mu = (L^{\dagger,
\sigma})^l\mu, \qquad l\in \mathds{N}.
\end{equation}
In  the next statement, we employ a perturbation technique for
constructing stochastic semigroups of bounded linear operators in
ordered Banach spaces with norms additive on the cones of positive
elements. Note that the spaces defined in (\ref{U1}), (\ref{U3})
have this property, see (\ref{U5}), (\ref{cones}). The details of
this technique can be found in our previous work \cite[subsect.
7.1.1]{KR}. Here we just recall that a semigroup $S=\{S(t)\}_{t\geq
0}$ of such operators is called stochastic (resp. substochastic) if
each $S(t)$ is positive and $\|S(t) u\| = \|u\|$ (resp. $\|S(t) u\|
\leq \|u\|$) for each positive $u$ and $t>0$. Also, for a given
$n\in \mathds{N}$ and
\begin{equation}
  \label{U29a}
\mathcal{D}^\sigma_n :=\{\mu \in \mathcal{M}_n: \
|\mu|(\varPsi^\sigma \varPsi_1^n) <\infty\},
\end{equation}
cf. (\ref{U1}), (\ref{U14}), by the trace of $A$ in $\mathcal{M}_n$
we mean the operator $(A, \mathcal{D}^\sigma_n )$ acting therein.
\begin{lemma}
  \label{U1lm}
For each $\sigma\in (0, 1]$, the closure of
$(L^{\dagger,\sigma},\mathcal{D}(L^{\dagger,\sigma}))$, see
(\ref{U14}), is the generator of a stochastic semigroup, $ S^\sigma
= \{S^\sigma(t)\}_{t\geq 0}$, in $\mathcal{M}$ such that
$S^\sigma(t): \mathcal{M}_n \to \mathcal{M}_n$, holding for each
$t>0$ and $n\in \mathds{N}$. For each $n\in \mathds{N}$, the
restrictions $S^\sigma(t)|_{\mathcal{M}_n}$ constitute a
$C_0$-semigroup on $\mathcal{M}_n$. Additionally, for each $\beta >
0$ and $\beta' \in (0, \beta)$, $S^\sigma(t) :
\mathcal{M}^{+}_{\beta} \to \mathcal{M}^{+}_{\beta'}$  for $t <
T_\sigma(\beta, \beta')$, see (\ref{U28}).
\end{lemma}
\begin{proof}
We basically follow the way of proving \cite[Lemma 7.6]{KR}, based
on the Thieme-Voigt theorem \cite{TV} in the form adapted to the
context of the present work, see \cite[Assumption 7.1 and
Proposition 7.2]{KR}. Thus, we begin by mentioning that all the
items of Assumption 7.1 \emph{ibid} are satisfied. That is: (i) each
$\mathcal{M}_n$ is dense in $\mathcal{M}$, see Proposition
\ref{U1pn}; (ii) each $\mathcal{M}_n$ is a Banach space (by the
aforementioned \cite[Proposition 4.1.8, page 119]{Cohn}); (iii) each
cone $\mathcal{M}_{n}^{+}$, $n\in \mathds{N}$, is
$\mathcal{M}^{+}\cap\mathcal{M}_n$ and $\|\cdot \|_n$ is additive on
this cone, see (\ref{U5}); (iv) each $\mathcal{M}_{n}^{+}$ is dense
in $\mathcal{M}^{+}$, see Proposition \ref{U1pn}. Now we can apply
\cite[Proposition 7.2]{KR}, which amounts to checking that:
\begin{itemize}
  \item[(i)] $-A$ and $B$ map
$\mathcal{D}(L^{\dagger,\sigma})\cap\mathcal{M}^{+}$ to
$\mathcal{M}^{+}$, which follows by the very definition of $A$ and
(\ref{U14a});
\item[(ii)] $(A, \mathcal{D}(L^{\dagger,\sigma}))$ generates a
substochastic semigroup, $S_0^\sigma = \{S_0^\sigma(t) \}_{t\geq
0}$, such that (a) $S^\sigma_0 (t) : \mathcal{M}_{n} \to
\mathcal{M}_{n}$, (b) the restrictions
$S_0^\sigma(t)|_{\mathcal{M}_n}$ constitute a $C_0$-semigroup on
$\mathcal{M}_n$ generated by the trace of $A$ in $\mathcal{M}_n$;
\item[(iii)] $B: \mathcal{D}^\sigma_n \to \mathcal{M}_n$ and
$\varphi((A+B)\mu) =0$, holding for all $\mu \in
\mathcal{D}(L^{\dagger,\sigma})\cap \mathcal{M}^{+}$;
\item[(iv)] there exist positive $c_n$ and $\varepsilon_n$ such that the following holds
\begin{equation}
  \label{U30}
\forall \mu \in \mathcal{D}^\sigma_n \cap \mathcal{M}^{+} \qquad
\varphi_n ((A+B)\mu) \leq c_n \varphi_n(\mu) - \varepsilon_n \|A
\mu\|.
\end{equation}
\end{itemize}
The semigroup $S_0^\sigma$ mentioned in item (ii) consists of the
multiplication operators
\begin{equation}
  \label{U31}
(S_0^\sigma (t) \mu) (d \gamma) = \exp\left( - t \varPsi^\sigma
(\gamma) \right) \mu( d \gamma),
\end{equation}
which is certainly such that (a) holds for each $n\in \mathds{N}$.
To check the strong continuity of $S_0^\sigma$, we take $\mu\in
\mathcal{M}^{+}$ and $\varepsilon >0$, and then show that $\|\mu -
S_0^\sigma (t)\mu\|< \varepsilon$ whenever $t$ is smaller than an
$\varepsilon$-specific $\delta >0$. The validity of such estimates
for an arbitrary $\mu\in \mathcal{M}$ then simply follows by the
Jordan decomposition. Since $\mathcal{D}(L^{\dagger,\sigma})$ is
dense in $\mathcal{M}$, see (\ref{U15}) and Proposition \ref{U1pn},
one finds $\mu'\in
\mathcal{D}(L^{\dagger,\sigma})\cap\mathcal{M}^{+}$ such that $\|
\mu - \mu'\| < \varepsilon/3$. By (\ref{U31}) and (\ref{U9}) we then
have
\begin{gather}
  \label{U32}
\| \mu - S^\sigma_0 (t) \mu\| \leq \|\mu - \mu'\| + \|S^\sigma_0 (t)
(\mu - \mu')\| + \|\mu' - S^\sigma_0 (t) \mu'\|\\[.2cm] \nonumber \leq 2 \|\mu -
\mu'\| + t\int_{\Gamma^2_*} \varPsi^\sigma (\gamma) \mu'(d\gamma)
\leq \frac{2}{3}\varepsilon + \frac{t\alpha}{\sigma} \|\mu'\|_1 <
\varepsilon,
\end{gather}
for $t< \sigma \varepsilon/3 \alpha \|\mu'\|_1$. Moreover,
(\ref{U31}) can be considered as the definition of bounded linear
operators acting in a given $\mathcal{M}_n$. These operators
constitute a $C_0$ semigroup, which can be proved similarly as in
(\ref{U32}). Its generator is then obviously the trace of $A$ in
$\mathcal{M}_n$, see (\ref{U29a}). Thus, it remains to prove the
validity of (\ref{U30}), that is, the validity of
\begin{equation}
  \label{U33}
\int_{\Gamma^2_*} \varPsi_1^n (\gamma) (L^{\dagger,\sigma}\mu) ( d
\gamma) + \varepsilon_n \int_{\Gamma^2_*} \varPsi^\sigma (\gamma)
\mu ( d \gamma)\leq c_n \int_{\Gamma^2+*} \varPsi_1^n (\gamma) \mu (
d \gamma),
\end{equation}
holding for all $\mu \in \mathcal{D}^\sigma_n \cap \mathcal{M}^{+}$
and certain positive $c_n$ and $\varepsilon_n$. This clearly amounts
to proving that each of the summands in the left-hand side of
(\ref{U33}) is $\leq (c_n/2)\mu (\varPsi_1^n)$ with a properly
chosen $c_n$. We begin by proving this for the first summand. By
(\ref{U11}) we have that
\begin{equation}
  \label{U34}
\int_{\Gamma^2_*} \varPsi_1^n (\gamma) (L^{\dagger,\sigma}\mu) ( d
\gamma) = \int_{\Gamma^2_*} (L^\sigma \varPsi_1^n) (\gamma) \mu ( d
\gamma)
\end{equation}
Similarly as in obtaining (\ref{U17}) we have
\begin{equation}
  \label{U35}
\left|\varPsi_1^n (\gamma\setminus x \cup_i y) - \varPsi_1^n
(\gamma) \right| \leq (2^n -1) \left|\psi(y)-\psi(x)
\right|\varPsi_1^{n-1} (\gamma).
\end{equation}
Set
\begin{equation}
  \label{U36}
b_i (x) = \int_X a_i(x-y) \left|\psi(y)-\psi(x) \right| dy, \qquad
i=0,1.
\end{equation}
Assume first that $|x|\geq |y|$. Then
\begin{gather}
  \label{U37}
\left|\psi(y)-\psi(x) \right| = \psi(y)-\psi(x) = (|x|^{d+1} -
|y|^{d+1}| \psi(x) \psi(y)\\[.2cm] \nonumber \leq \psi(x)\left[ \left(|x-y|+ |y|
\right)^{d+1} - |y|^{d+1} \right] \psi(y) \\[.2cm] \nonumber  = \psi(x)
\sum_{l=1}^{d+1} {d+1 \choose l} |x-y|^l |y|^{d+1-l}\psi(y) \\[.2cm] \nonumber \leq
\psi(x) \sum_{l=1}^{d+1}  {d+1 \choose l} |x-y|^l  =: \psi(x)
h(x-y),
\end{gather}
where we have used the fact that $|y|^{d+1-l} \psi(y) \leq 1$ for
all $l\geq 1$ and $y\in X$. For $|x|< |y|$, we have
\begin{gather*}
\left|\psi(y)-\psi(x) \right| =  (|y|^{d+1} - |x|^{d+1}) \psi(x)
\psi(y) \\[.2cm] \nonumber \leq \psi(x) \left[\left(|y-x| + |x| \right)^{d+1} - |x|^{d+1} \right] \psi(y)\\[.2cm] \nonumber
\leq \psi(x) \sum_{l=1}^{d+1}  {d+1 \choose l} |x-y|^l
|y|^{d+1-l}\psi(y) \\[.2cm] \nonumber  \leq \psi(x) h(x-y).
\end{gather*}
Now we use these two estimates in (\ref{U36}) and obtain
\begin{equation}
  \label{U39}
b_i(x) \leq \psi(x) \bar{\alpha}, \qquad \bar{\alpha}:= \max_{i=0,1}
\int_X a_i (x) h(x) dx,
\end{equation}
cf. (\ref{L2}) and (\ref{C8}). By (\ref{U35}) and the latter
estimate we obtain
\begin{gather}
  \label{U40}
\left|(L^\sigma \varPsi_1^n) (\gamma) \right| \leq \sum_{x\in
\gamma_0} \int_{X} a_0 (x-y) \left|\varPsi_1^n (\gamma\setminus x
\cup_0 y)
-\varPsi_1^n(\gamma) \right| d y \\[.2cm] \nonumber + \sum_{x\in \gamma_1}
\int_{X} a_1 (x-y) \left|\varPsi_1^n (\gamma\setminus x \cup_1 y)
-\varPsi_1^n(\gamma) \right| d y \\[.2cm] \nonumber \leq (2^{n} -1)
\bar{\alpha} \varPsi_1^{n-1} (\gamma) \left( \sum_{x\in
\gamma_0}\psi(x) + \sum_{x\in \gamma_1}\psi(x) \right) \leq (2^{n}
-1) \bar{\alpha} \varPsi_1^{n} (\gamma),
\end{gather}
holding for all $\sigma \in [0,1]$, including $\sigma=0$. By
(\ref{U40}) and (\ref{U34}) we then have
\begin{equation}
  \label{U41}
\int_{\Gamma^2_*} \varPsi_1^n (\gamma) (L^{\dagger,\sigma}\mu) ( d
\gamma) \leq (2^{n} -1) \bar{\alpha} \int_{\Gamma^2_*} \varPsi_1^n
(\gamma) \mu ( d \gamma).
\end{equation}
By (\ref{U9}) we have $\varPsi^\sigma (\gamma) \leq (\alpha/\sigma)
\varPsi_1^n (\gamma)$ holding for all $n\in \mathds{N}$ and
$\gamma\in \Gamma_*^2$, which then yields
\begin{equation*}
\int_{\Gamma^2_*} \varPsi^\sigma (\gamma) \mu(d \gamma) \leq
\frac{\alpha}{\sigma} \int_{\Gamma^2_*}  \varPsi_1^n(\gamma) \mu(d
\gamma).
\end{equation*}
The latter estimate together with (\ref{U41}) yields the validity of
(\ref{U33}) with $\varepsilon_n =1$ and $c_n = (2^{n} -1)
\bar{\alpha} + \alpha/\sigma$.

It remains now to prove the concluding statement of the lemma. We
proceed by defining the following bounded operators
\begin{equation}
  \label{U43}
S^\sigma_{\beta',\beta} (t)  = I_{\beta',\beta} +\sum_{l=1}^\infty
\frac{t^l}{l!} (L^{\dagger,\sigma}_{\beta',\beta})^l, \qquad t<
T_\sigma(\beta, \beta'),
\end{equation}
acting from $\mathcal{M}_\beta$ to $\mathcal{M}_{\beta'}$, see
(\ref{U28}). Here the powers $(L^{\dagger,\sigma}_{\beta',\beta})^l$
satisfy (\ref{U27}) and $I_{\beta',\beta}$ is the embedding
operator. By (\ref{U43}) and (\ref{U29}), for each $\mu \in
\mathcal{M}_\beta$, one has
\begin{equation}
  \label{U44}
S^\sigma_{\beta',\beta} (t) \mu = S^\sigma (t) \mu,
\end{equation}
where $S^\sigma(t)$ is the same as in the first part of the lemma.
Then the positivity of $S^\sigma_{\beta',\beta} (t)$ follows by the
positivity of the latter. This completes the whole proof.
\end{proof}
Thus, the lemma just proved yields the existence of the semigroup
$S^\sigma$ which we use in (\ref{U10}) to obtain the Markov
transition function $p_t^\sigma$. The fact that $t\mapsto
p^\sigma_t$ satisfies the corresponding conditions,   see \cite[eqs.
(1.3) -- (1.6), page 156]{EK}, follows directly by (\ref{U10}). We
will use this function to construct the finite dimensional marginals
of the stochastic process corresponding to the approximating model
described by $L^\sigma$. This will be done in the next subsection.

\subsection{Constructing path measures}
By means of the semigroup $S^\sigma$ constructed in Lemma \ref{U1lm}
we may have
\begin{equation}
  \label{U45}
\hat{\mu}^\sigma_t(\cdot) = (S^\sigma (t) \mu)(\cdot) =
\int_{\Gamma^2_*} p^\sigma_t(\gamma ,\cdot) \mu(d\gamma), \qquad \mu
\in \mathcal{M}.
\end{equation}
Recall that here $\sigma\in (0,1]$ and $S^\sigma$ is stochastic. The
latter means that $\hat{\mu}^\sigma_t$ is in
$\mathcal{P}(\Gamma^2_*)$ whenever $\mu$ has this property.
Moreover, $\hat{\mu}^\sigma_t$ may be in
$\mathcal{M}_n\cap\mathcal{P}(\Gamma^2_*)$ under the corresponding
condition. However, so far we do not know whether $S^\sigma$
preserves $\mathcal{P}_{\rm exp}$.
\begin{lemma}
  \label{U2lm}
For  given  $\mu\in\mathcal{P}_{\rm exp}$, let $t\mapsto
\mu_t^\sigma\in \mathcal{P}_{\rm exp}$, $t>0$, be the solution of
(\ref{FPEs}), see Proposition \ref{Sq1pn}. Let also
$\hat{\mu}^\sigma_t$ be as in (\ref{U45}) with the same $\mu$. Then,
for all $t>0$, it follows that ${\mu}^\sigma_t=\hat{\mu}^\sigma_t$.
\end{lemma}
\begin{proof}
By (\ref{U7}) and (\ref{U15}) it follows that $\mathcal{P}_{\rm
exp}\subset \mathcal{D}(L^{\dagger,\sigma})$, which means that
$t\mapsto \hat{\mu}^\sigma_t$ is the classical solution of the
initial value problem
\begin{equation}
  \label{U46}
\frac{d}{dt}\hat{\mu}_t^\sigma = L^{\dagger,\sigma}
\hat{\mu}^\sigma_t, \qquad \hat{\mu}^\sigma_t|_{t=0}=\mu,
\end{equation}
which by (\ref{U11}) yields that $t\mapsto \hat{\mu}_t^\sigma$
solves (\ref{FPEs}). Then the proof follows by Proposition
\ref{Sq1pn}.
\end{proof}

It is a standard fact that the transition function $p^\sigma_t$
determines the finite dimensional distributions of a Markov process,
$\mathcal{X}^\sigma$, with values in $\Gamma^2_*$, see \cite[Theorem
1.1, page 157]{EK}. Our aim now is to prove that it has
c{\`a}dl{\`a}g paths. To this end, we employ Chentsov-like
arguments, cf. \cite{Chentsov} and \cite[Proposition 7.8]{KR}, and
thus the metric $\upsilon^*$, see (\ref{N4a}), (\ref{N5}). By Lemma
\ref{N2pn} it is complete. Set
\begin{gather}
  \label{U47}
w^\sigma_u(\gamma) = \int_{\Gamma^2_*} \upsilon^* (\gamma, \gamma')
p^\sigma_u ( \gamma, d \gamma')\\[.2cm] \nonumber W^\sigma_{u,v}
(\gamma) = \int_{\Gamma^2_*} \upsilon^* (\gamma,\gamma')
w^\sigma_u(\gamma') p^\sigma_v ( \gamma, d \gamma'), \quad u,v\geq
0.
\end{gather}
Thereafter, for  a triple $t_3 > t_2 > t_1 \geq  0$, consider
\begin{equation}
  \label{U48}
\mathcal{W}^\sigma (t_1, t_2 , t_3) =\int_{\Gamma^2_*}
W^\sigma_{t_3-t_2,t_2-t_1} (\gamma) \hat{\mu}^\sigma_{t_1} ( d
\gamma) = \int_{\Gamma^2_*} W^\sigma_{t_3-t_2,t_2-t_1} (\gamma)
{\mu}^\sigma_{t_1} ( d \gamma).
\end{equation}
Note that this $\mathcal{W}^\sigma (t_1, t_2 , t_3)$ depends also on
$\mu=\mu^\sigma_t|_{t=0}$, see Lemma \ref{U2lm}. By combining
\cite[Theorem 1]{Chentsov} and \cite[Theorems 8.6 - 8.8, pages
137-139]{EK} we obtain the following statement.
\begin{proposition}
  \label{U3pn}
Given $T>0$,  $\sigma \in (0,1]$, $s\geq 0$ and $\mu\in
\mathcal{P}_{\rm exp}$, assume that there exist $C_\sigma
> 0$ and $\delta
> 0$ such that, for each triple that satisfies $t_1\geq s$, $ t_3 \leq T$ and $t_3-t_1 <
\delta$, the following holds
\begin{equation}
  \label{U49}
\mathcal{W}^\sigma (t_1, t_2 , t_3) \leq C_\sigma |t_3-t_1|^2.
\end{equation}
Then
\begin{itemize}
 \item[(i)] There exists a probability measure
$P^\sigma_{s,\mu}$ on $\mathfrak{D}_{[0,+\infty)}(\Gamma^2_*)$
uniquely determined by its finite dimensional marginals, cf.
\cite[eq. (1.10), page 157]{EK}, expressed by the formula
\begin{eqnarray}
  \label{U50}
& & P^\sigma_{s,\mu} (\{\gamma: \varpi_{t_n} (\gamma) \in
\mathbb{A}_n, \varpi_{t_{n-1}}(\gamma) \in \mathbb{A}_{n-1}, \dots ,
\varpi_{t_1} (\gamma) \in \mathbb{A}_1, \varpi_{0} (\gamma) \in
\mathbb{A}_0\})
\qquad  \\[.2cm] \nonumber & &  \ =
\int_{\mathbb{A}_{n-1}} \cdots \int_{\mathbb{A}_0}
p^\sigma_{t_{n}-t_{n-1}} (\gamma_{n-1}, \mathbb{A}_n)
p^\sigma_{t_{n-1}-t_{n-2}} (\gamma_{n-2}, d \gamma_{n-1}) \cdots
p^\sigma_{t_{2}-t_{1}} (\gamma_{1}, d \gamma_2) \\[.2cm] \nonumber &
& \ \times p^\sigma_{t_1} (\gamma_0, d \gamma_1) \mu(d \gamma_0),
\end{eqnarray}
holding for all $n\in \mathds{N}$, $t_n > t_{n-1} \cdots t_1$ and
$\mathbb{A}_j \in \mathcal{B}(\Gamma^2_*)$, $j=0, \dots , n$.
\item[(ii)] If the estimate in (\ref{U49}) holds for all $\sigma\in (0,1]$ with one and the same $C >
0$,and the family $\{\hat{\mu}_t^\sigma\}_{\sigma \in (0,1]}$ is
tight for all $t>0$, then the family $\{P^\sigma_{s,\mu}\}_{\sigma
\in (0,1]}$ of measures mentioned in (i) is also tight, and hence
has accumulation points in the weak topology.
\end{itemize}
\end{proposition}
Note that the tightness mentioned in item (ii) of the latter
statement follows by Prohorov's theorem and Lemmas \ref{Sqlm} and
\ref{U2lm}.
\begin{lemma}
  \label{U4lm}
For every $s\geq 0$ and $\mu\in \mathcal{P}_{\rm exp}$, the estimate
as in (\ref{U49}) holds for all $\sigma\in (0,1]$ with one and the
same $C>0$, dependent on $T$ only.
\end{lemma}
\begin{proof}
For convenience, we take here $s=0$ -- the proof for $s>0$ is
completely analogous. Then we begin by recalling that
$\delta_\gamma$ is in $\mathcal{D}(L^{\dagger, \sigma})$, see Remark
\ref{Urk}. Thus, by (\ref{U10}) and the corresponding formulas, see
e.g., \cite[eq. (1.16), page 9]{EK}, we have
\begin{equation}
  \label{U51}
p^\sigma_t (\gamma, \cdot) = \delta_\gamma (\cdot) + \int_0^t
L^{\dagger, \sigma} p^\sigma_s (\gamma, \cdot) ds.
\end{equation}
We use this in (\ref{U47}), which yields
\begin{eqnarray}
  \label{U52}
w^\sigma_u(\gamma) & = & w^\sigma_0(\gamma) +
\int_0^u\left(\int_{\Gamma^2_*} \upsilon^* (\gamma, \gamma')
L^{\dagger, \sigma} p^\sigma_s (\gamma, d\gamma')\right) d s
\\[.2cm] \nonumber & = & \int_0^u\left(\int_{\Gamma^2_*} \upsilon^* (\gamma, \gamma')
L^{\dagger, \sigma} p^\sigma_s (\gamma, d\gamma')\right) d s\\[.2cm] \nonumber & = & \int_0^u\left(\int_{\Gamma^2_*}
L^\sigma\upsilon^* (\gamma, \gamma')p^\sigma_s (\gamma,
d\gamma')\right) d s,
\end{eqnarray}
see  (\ref{U11}). The second equality in (\ref{U52}) follows by  the
fact that $w^\sigma_0(\gamma) = \upsilon^*(\gamma,\gamma)=0$ as
$\upsilon^*$ is a metric. The function $\gamma' \mapsto L^\sigma
\upsilon^* (\gamma, \gamma')=: J^\sigma_\gamma (\gamma')$ has the
following form, see (\ref{N4a}), (\ref{N5}),
\begin{eqnarray}
  \label{U53}
J^\sigma_\gamma (\gamma') & = & \sum_{x\in \gamma'_0}\int_X
a^\sigma_0 (x,y) \exp\left(-\sum_{z\in \gamma_1}\phi_0(z-y) \right)
\left[\upsilon^* (\gamma, \gamma' \setminus x\cup_0 y) -
\upsilon^* (\gamma, \gamma') \right] d y\qquad \\[.2cm] \nonumber & + &
\sum_{x\in \gamma'_1}\int_X a^\sigma_1 (x,y) \exp\left(-\sum_{z\in
\gamma_0}\phi_1(z-y) \right) \left[\upsilon^* (\gamma, \gamma'
\setminus x\cup_1 y) - \upsilon^* (\gamma, \gamma') \right] d y.
\end{eqnarray}
By the triangle inequality for the metric $\upsilon_*$ we then get
\begin{eqnarray}
  \label{U54}
\left|J^\sigma_\gamma (\gamma') \right| & \leq  & \sum_{x\in
\gamma'_0}\int_X a_0 (x-y) \upsilon (\psi(\gamma'_0 \setminus
x\cup y), \psi\gamma_0') d y\\[.2cm] \nonumber & + & \sum_{x\in \gamma'_1}\int_X a_1
(x-y) \upsilon (\psi(\gamma'_1 \setminus x\cup y), \psi\gamma_1') d
y.
\end{eqnarray}
By (\ref{N4}) and (\ref{N5}) it follows that
\begin{equation}
  \label{U55}
\upsilon (\psi(\gamma'_i \setminus x\cup y), \psi\gamma_i') \leq
\sup_{g: \|g\|_{BL}\leq 1} \left|g(x) \psi(x) - g(y)\psi(y) \right|.
\end{equation}
Proceeding similarly as in (\ref{U37}) we obtain
\begin{gather*}
\left|g(x) \psi(x) - g(y)\psi(y) \right| = \psi(x) \psi(y)
\left|\frac{g(x)}{\psi(y)} - \frac{g(y)}{\psi(x)} \right| \\[.2cm]
\nonumber \leq \psi(x)\psi(y) g(y)\left| \frac{1}{\psi(y)} -
\frac{1}{\psi(x)} \right| + \psi(x) \left|g(x)-g(y) \right| \\[.2cm]
\nonumber \leq \psi(x) \left[|x-y| + \sum_{l=1}^{d+1} {d+1 \choose
l} |x-y|^l \right].
\end{gather*}
We apply the latter in (\ref{U55}), (\ref{U54}) and then arrive at
the following, see (\ref{U39}) and (\ref{L2}),
\begin{equation}
  \label{U57}
\left|J^\sigma_\gamma (\gamma') \right|  \leq (\alpha +
\bar{\alpha}) \varPsi(\gamma'),
\end{equation}
which is uniform in  $\sigma$. By (\ref{U57}) we then get from
(\ref{U52}) the following estimate
\begin{equation}
  \label{U58}
w^\sigma_u(\gamma) \leq (\alpha + \bar{\alpha})\int_0^u
\kappa^\sigma_s (\gamma) d s, \qquad \kappa^\sigma_s (\gamma):=
\int_{\Gamma^2_*} \varPsi(\gamma')p^\sigma_s (\gamma, d \gamma').
\end{equation}
By (\ref{U51}), similarly as in (\ref{U52}) we have
\begin{equation}
  \label{U59}
 \kappa^\sigma_s (\gamma) =\varPsi(\gamma) + \int_0^s
\left(\int_{\Gamma^2_*} L^\sigma \varPsi(\gamma') p^\sigma_v
(\gamma, d \gamma') \right) d v
\end{equation}
Proceeding as in (\ref{U40}) we obtain
\[
\left|L^\sigma \varPsi(\gamma') \right| \leq \bar{\alpha}
\varPsi(\gamma'),
\]
by which we obtain from (\ref{U58}), (\ref{U59}) the following
\begin{equation}
  \label{U60}
\kappa^\sigma_s (\gamma) \leq \varPsi(\gamma) + \bar{\alpha}
\int_0^s \kappa^\sigma_v (\gamma) d v,
\end{equation}
which by the Gr\"onwall inequality and (\ref{U58}) leads to
\begin{equation}
  \label{U61}
w^\sigma_u(\gamma) \leq (\alpha + \bar{\alpha}) u e^{\bar{\alpha}u}
\varPsi(\gamma).
\end{equation}
Now we may pass to estimating $W^\sigma_{u,v}(\gamma)$. By the
second line in (\ref{U47}) and (\ref{U61}) we have
\begin{equation}
  \label{U62}
W^\sigma_{u,v}(\gamma) \leq (\alpha + \bar{\alpha}) ue^{\bar{\alpha}
u} V^\sigma_v(\gamma), \qquad V^\sigma_v(\gamma) :=
\int_{\Gamma_*^2} \upsilon^* (\gamma, \gamma') \varPsi(\gamma')
p^\sigma_v ( \gamma, d \gamma').
\end{equation}
Here we again apply (\ref{U51}) and then get, cf. (\ref{U52}),
\begin{equation}
  \label{U63}
V^\sigma_v(\gamma) = \int_0^v\left( \int_{\Gamma_*^2} \left[
L^\sigma \varPsi(\gamma') \upsilon^*(\gamma, \gamma') \right]
p^\sigma_s ( \gamma, d \gamma') \right) d s
\end{equation}
Proceeding as in (\ref{U53}) we get, see also (\ref{U36}),
\begin{eqnarray*}
\left| L^\sigma \varPsi(\gamma') \upsilon^*(\gamma, \gamma') \right|
& \leq & \sum_{x\in \gamma'_0} \int_X a_0(x-y)
\bigg{|}\varPsi(\gamma'\setminus x \cup_0 y)
\upsilon^*(\gamma, \gamma' \setminus x \cup_0 y) \\[.2cm] \nonumber & & \qquad \qquad \qquad \qquad \qquad  -
\varPsi(\gamma') \upsilon^*(\gamma,
\gamma') \bigg{|} dy \\[.2cm] \nonumber & + &\sum_{x\in \gamma'_1} \int_X a_1(x-y)
\bigg{|}\varPsi(\gamma'\setminus x \cup_1 y)
\upsilon^*(\gamma, \gamma' \setminus x \cup_1 y) \\[.2cm] \nonumber & & \qquad \qquad \qquad \qquad \qquad  - \varPsi(\gamma')
\upsilon^*(\gamma, \gamma') \bigg{|} dy \\[.2cm] \nonumber & \leq &
2\bigg{(} \sum_{x\in \gamma'_0} b_0 (x) +  \sum_{x\in \gamma'_1}
b_1(x)\bigg{)}  \\[.2cm] \nonumber & + &
\varPsi(\gamma') \bigg{(} \sum_{x\in \gamma_0} \int_X a_0(x-y)
\upsilon^* (\gamma'\setminus x \cup_0 y, \gamma') dy  \\[.2cm] \nonumber & & \qquad + \sum_{x\in \gamma_0} \int_X a_1(x-y)
\upsilon^* (\gamma'\setminus x \cup_1 y, \gamma') dy\bigg{)}  \\[.2cm] \nonumber & \leq & 2 \bar{\alpha}
 \varPsi(\gamma') + (\alpha + \bar{\alpha}) \varPsi^2(\gamma'),
\end{eqnarray*}
where we used also (\ref{U39}) and (\ref{U54}). Now we apply the
latter estimate in (\ref{U63}) and obtain
\begin{equation}
  \label{U65}
V^\sigma_v(\gamma) \leq 2 \alpha \int_0^v \kappa^\sigma_s (\gamma) d
s  + (\alpha + \bar{\alpha}) \int_0^v  K^\sigma_s (\gamma) d s,
\end{equation}
where $\kappa^\sigma_s (\gamma)$ is the same as in (\ref{U58}) and
\begin{equation*}
K^\sigma_s (\gamma) = \int_{\Gamma^2_*} \varPsi^2 (\gamma')
p_s^\sigma (\gamma, d \gamma').
\end{equation*}
By (\ref{U51}) we have, cf. (\ref{U59}), (\ref{U60}),
\begin{equation}
  \label{U67}
K^\sigma_s (\gamma) =  \varPsi^2(\gamma) + \int_0^s
\left(\int_{\Gamma^2_*} (L^\sigma \varPsi^2(\gamma') ) p_u^\sigma
(\gamma, d \gamma')\right) d u.
\end{equation}
Similarly as in (\ref{U40}) it follows that
\[
\left|(L^\sigma \varPsi^2)(\gamma') \right| \leq 3 \bar{\alpha}
\varPsi^2 (\gamma'),
\]
by which and the Gr\"onwall inequality we get from (\ref{U67}) the
following estimate
\begin{equation}
  \label{U68}
K^\sigma_s (\gamma) \leq \varPsi^2(\gamma) e^{3 \bar{\alpha} s}.
\end{equation}
Now we use (\ref{U60}) and (\ref{U68}) in (\ref{U65}) and arrive at
\begin{equation}
  \label{U69}
V^\sigma_v(\gamma) \leq v\left(2 \alpha e^{\bar{\alpha}v}
\varPsi(\gamma) + (\alpha + \bar{\alpha}) e^{3 \bar{\alpha} v}
\varPsi^2(\gamma)\right)
\end{equation}
Now we may turn to (\ref{U48}) where we use the  estimate, see
(\ref{W3}),
\[
\int_{\Gamma^2_*} \varPsi^n (\gamma) \mu_{t_1}(d \gamma)\leq T_n
(\varkappa \langle \psi \rangle), \qquad n=1,2,
\]
and the fact that $\mu_{t_1} \in \mathcal{P}_{\rm exp}$ of type  not
exceeding $\exp(\vartheta_0 + \alpha t_1) \leq \exp(\vartheta_0 +
\alpha T) =:\varkappa$, see Proposition \ref{W1pn} and Remark
\ref{W1rk}. Here $e^{\vartheta_0}$ is the type of $\mu$. By
(\ref{U62}) and (\ref{U69}) we then conclude that
$\mathcal{W}^\sigma(t_1, t_2 , t_3)$ satisfies (\ref{U49}) with
\[
C = 2 \alpha (\alpha + \bar{\alpha}) e^{2 \bar{\alpha}T} T_1
(\varkappa \langle \psi \rangle) + (\alpha + \bar{\alpha})^2 e^{4
\bar{\alpha}T} T_2 (\varkappa \langle \psi \rangle),
\]
which ends the proof.
\end{proof}

\section{Completing the Proof}

Here the hardest part is the proof of item (i), whereas the validity
of (iii) is rather standard, see  cf. \cite[Theorem 5.1.2, claim
(iv), page 80]{Dawson}.

\subsection{Proving item (i)} First we prove \underline{existence}
by employing the fact that, for a given $\mu\in \mathcal{P}_{\rm
exp}$ and $s\geq 0$, the measure in question, $P_{s,\mu}$, is
obtained as an accumulation point of the family
$\{P^\sigma_{s,\mu}\}_{\sigma \in (0,1]}$. Our aim now is to prove
that such accumulation points have properties (a), (b), (c)
mentioned in Definition \ref{A1df}.

To check (a), we note that $P_{s,\mu}\circ \varpi_s^{-1}=\mu$ since
$P^\sigma_{s,\mu}\circ \varpi_s^{-1}=\mu$ for all $\sigma\in (0,1]$.
Furthermore, by Lemmas \ref{U2lm} and \ref{Sqlm} it follows that
$P^\sigma_{s,\mu}\circ \varpi^{-1}_t \Rightarrow \mu_t$ as $\sigma
\to 0$, which yields $P_{s,\mu}\circ \varpi^{-1}_t = \mu_t$, holding
for all accumulation points in view of Lemma \ref{W2lm}. These facts
yield the validity of (b) of Definition \ref{A1df}.

To check (c), we take $\sf G$ as in (\ref{A1200}) with fixed
$t_2>t_1>s$, $m\in \mathds{N}$ and $s_1 < s_2 <\cdots < s_m$, $s_1
\geq s$, $s_m \leq t_1$. Then we recall that $\mu_{s_1}^\sigma =
\hat{\mu}^\sigma_{s_1}= S^\sigma(s_1 - s) \mu$, the type of which
does not exceed $e^{\vartheta(s_1-s)}$, $\vartheta(t) = \vartheta_0
+ \alpha t$, see Lemma \ref{U2lm}, and set $\chi^\sigma_{s_1} =
C^{-1}_{1,\sigma} F_1  \mu^\sigma_{s_1}$, that is,
\begin{gather}
  \label{U70}
\chi^\sigma_{s_1} (d \gamma) = C^{-1}_{1,\sigma} F_1 (\gamma)
\mu^\sigma_{s_1}(d \gamma), \qquad  C_{1,\sigma} := \mu_{s_1}^\sigma
(F_1).
\end{gather}
Note that $C_{1,\sigma} >0$ since each $F\in
\widetilde{\mathcal{F}}$ is strictly positive, see (\ref{T4}) and
(\ref{Sq29}). By claim (d) of Proposition \ref{T3pn} it follows that
$\chi^\sigma_{s_1}\in \mathcal{P}_{\rm exp}$, and its type does not
exceed that of $\mu_{s_1}$, and hence $\exp(\vartheta(s_1-s))$. Then
we define recursively
\begin{gather}
  \label{U71}
\tilde{\chi}^\sigma_{s_l} (d \gamma) = (S^\sigma(s_l - s_{l-1})
\chi^\sigma_{s_{l-1}})(d\gamma) = \int_{\Gamma^2_*} p^{\sigma}_{s_l
- s_{l-1}} (\gamma', d \gamma) \chi^{\sigma}_{s_{l-1}}( d \gamma'),
\\[.2cm] \nonumber
\chi^\sigma_{s_l} (d \gamma) = C_{l,\sigma}^{-1} F_l(\gamma)
\tilde{\chi}^\sigma_{s_l} (d \gamma) , \qquad C_{l,\sigma} :=
\tilde{\chi}^\sigma_{s_l} (F_l), \qquad l \leq m.
\end{gather}
As above, for all $l\leq m$, $\chi^\sigma_{s_l}$ is sub-Poissonian
of type $\leq \exp(\vartheta(s_l-s))$. Now we take $F\in
\mathcal{D}(L)$, see Definition \ref{THdf}, $t\in [s_m, t_2]$,  set
\begin{equation}
  \label{U72}
{\sf F}_t = F\circ \varpi_t, \qquad {\sf K}_t = (LF)\circ \varpi_t,
\qquad {\sf K}^\sigma_t = (L^\sigma F)\circ \varpi_t, \quad \
\sigma\in (0.1],
\end{equation}
and then consider $P^\sigma_{s,\mu} ({\sf F}_t {\sf G})$ with $\sf
G$ as just discussed. By (\ref{U50}) it follows that
\begin{equation}
  \label{U73}
P^\sigma_{s,\mu} ({\sf F}_t {\sf G}) = C_\sigma
P^\sigma_{s,\chi^\sigma_{s_m}} ({\sf F}_t ) = C_\sigma
P^\sigma_{s,\chi^\sigma_{s_m}}(F\circ \varpi_t) =: C_\sigma
\mu^{\sigma,s_m}_t (F),
\end{equation}
with $C_\sigma= P^\sigma_{s,\mu} ({\sf G})>0$. By (\ref{U50})
\begin{equation}
  \label{U73a}
\mu^{\sigma,s_m}_t = S^\sigma (t-s_m) \chi^\sigma_{s_m},
\end{equation}
and the type of $\mu^{\sigma,s_m}_t$ is $\leq e^{\vartheta (t-s)}$.
By (\ref{U73a}) it follows that
\[
\mu^{\sigma,s_m}_{t_2} (F) - \mu^{\sigma,s_m}_{t_1} (F) =
\int_{t_1}^{t_2} \mu^{\sigma,s_m}_{t} (L^\sigma F) dt =
\int_{t_1}^{t_2} P^\sigma_{s,\mu} ({\sf F}_t {\sf G}) dt,
\]
see (\ref{U73}), which yields $P^\sigma_{s,\mu}({\sf H})=0$, holding
for all $\sigma\in (0,1]$.

Now let $P_{s,\mu}$ be an accumulation point of the family
$\{P^\sigma_{s,\mu}\}_{\sigma \in (0,1]}$. By Lemmas \ref{Sqlm} and
\ref{U2lm}, all such accumulation points have the same one
dimensional marginals coinciding with $\mu_t$. For this $P_{s,\mu}$,
let $\{\sigma_n\}_{n\in \mathds{N}}\subset (0,1]$, $\sigma_n \to 0$,
be such that $P^{\sigma_n}_{s,\mu}\Rightarrow P_{s,\mu}$ as $n\to
+\infty$. Then set, cf. (\ref{U73}),
\begin{gather}
  \label{U73z}
C_n = C_{\sigma_n}= P^{\sigma_n}_{s,\mu}({\sf G}), \qquad C_{\infty}
= P_{s,\mu}({\sf G}), \\[.2cm] \nonumber
\mu^{s_m}_t (\mathbb{A}) = C^{-1}_{\infty} P_{s,\mu}({\sf
G}\mathds{1}_{\mathbb{A}}\circ \varpi_t), \qquad \mathbb{A}\in
\mathcal{B}(\Gamma^2_*), \qquad t\in [s_m, t_2].
\end{gather}
Let us show that the assumed convergence $P^{\sigma_n}_{s,\mu}
\Rightarrow P_{s,\mu}$ implies $\mu_t^{\sigma_n, s_m} \Rightarrow
\mu^{s_m}_t$, as $n\to+\infty$. To this end, by $\tilde{\chi}_{s_l}
$ we denote $\mu_{s_l}^{s_{l-1}}$, where $\mu_t^{s_{l-1}}$, $t\geq
s_{l-1}$ is the solution  of (\ref{FPE}) with the initial condition
$\chi_{s_l}:=C^{-1}_{l-1, \infty} F_{l-1} \tilde{\chi}_{s_{l-1}}$,
$l=2, \dots m$, where $C_{l,\infty} = \tilde{\chi}_{s_{l}}(F_l)$,
cf. (\ref{U71}), and $\tilde{\chi}_{s_1} = \mu_{s_1}=P_{s,\mu}\circ
\varpi^{-1}_{s_1}$, which solves (\ref{FPE}) on $[s,s_1]$ with the
initial condition $\mu$. The assumed convergence of the path
measures implies $\tilde{\chi}_{s_1}^{\sigma_n} \Rightarrow
\tilde{\chi}_{s_1}$, see (\ref{U70}). By Lemma \ref{Sqlm} this
yields $\tilde{\chi}_{s_2}^{\sigma_n} \Rightarrow
\tilde{\chi}_{s_2}$, and thus $\tilde{\chi}_{s_l}^{\sigma_n}
\Rightarrow \tilde{\chi}_{s_l}$ for all $l\leq m$. Since
$\mu_{t}^{s_m}$ defined in (\ref{U73z}) is the solution of
(\ref{FPE}) on $[s_m , t]$ with the initial condition
$\chi_{s_m}:=C^{-1}_{m, \infty} F_{m} \tilde{\chi}_{s_{m}}$, this
yields the convergence in question. By Proposition \ref{Sqqpn} this
yields in turn that $\mu^{s_m}_t\in \mathcal{P}_{\rm exp}$ and the
type of $\mu^{s_m}_t$ is $\leq \exp(\vartheta(t-s))$. Note that
$C_\infty$ defined in (\ref{U73z}) is $C_{m,\infty}$.

Keeping the aforementioned facts in mind we write, see (\ref{A120}),
\begin{equation}
  \label{U74}
P_{s,\mu}({\sf H}) = P_{s,\mu}({\sf F}_{t_2} {\sf G}) -
P_{s,\mu}({\sf F}_{t_1} {\sf G}) - \int_{t_1}^{t_2} P_{s,\mu}({\sf
K}_{t} {\sf G}) dt,
\end{equation}
and also set
\begin{gather}
  \label{U75}
a_n(t) = P_{s,\mu}({\sf F}_{t} {\sf G}) - P^{\sigma_n}_{s,\mu}({\sf
F}_{t} {\sf G}), \\[.2cm] \nonumber b_n(t) = P_{s,\mu}({\sf K}_{t} {\sf G}) - P^{\sigma_n}_{s,\mu}({\sf
K}_{t} {\sf G}), \\[.2cm] \nonumber c_n(t) =  P^{\sigma_n}_{s,\mu}\left(({\sf
K}_{t} - {\sf K}^{\sigma_n}_{t} ) {\sf G}\right).
\end{gather}
Since $P^\sigma_{s,\mu}({\sf H})=0$, by (\ref{U74}) and (\ref{U75})
it follows that
\begin{gather*}
P_{s,\mu}({\sf H}) = P_{s,\mu}({\sf H}) - P^{\sigma_n}_{s,\mu}({\sf
H}) = \left[a_n(t_2)-a_n(t_1) \right] \\[.2cm] \nonumber -
\int_{t_1}^{t_2} b_n (t) dt - \int_{t_1}^{t_2} c_n (t) dt =: a_n
-b_n -c_n.
\end{gather*}
By $P_{s,\mu}^{\sigma_n} \Rightarrow P_{s,\mu}$ we have $a_n\to 0$
as $n\to +\infty$. However, the same conclusion for $b_n$ and $c_n$
does not follow in so simple way as $LF$ and $L^\sigma F$ need not
be continuous. To settle this case, by means of (\ref{U73z}) we
write
\begin{gather}
  \label{U77}
b_n = C_n \int_{t_1}^{t_2} \left( \mu^{s_m}_t (LF) -
\mu^{\sigma_n,s_m}_t (LF) \right) dt + (C_\infty - C_n)
\int_{t_1}^{t_2} \mu_t^{s_m} (L F ) dt.
\end{gather}
By item (a) of Proposition \ref{T3pn}, $LF$ is a bounded function;
hence, the second summand in (\ref{U77}) vanishes as $n\to +\infty$
since $C_n \to C_\infty$ by the assumed weak convergence, see
(\ref{U73z}). To prove the same for the first summand -- denote it
$b_n^{(1)}$ -- we employ the fact that $\mu^{\sigma_n,s_m}_t$ and
$\mu^{s_m}_t$ are sub-Poissonian and each $F\in \mathcal{D}(L)$ can
be written as $KG$ with $G\in\widetilde{ \mathcal{G}}_\infty$, see
Proposition \ref{V1pn}. Then
\begin{gather}
  \label{U77a}
\mu^{s_m}_t (LF) - \mu^{\sigma_n,s_m}_t (LF) = \langle \! \langle
k_{\mu^{s_m}_t} - k_{\mu^{\sigma_n,s_m}_t} , \widehat{L} G \rangle
\! \rangle \to 0, \qquad n\to+\infty,
\end{gather}
where we have taken into account that $\widehat{L} G \in
\widetilde{\mathcal{G}}_\infty$ whenever $G \in
\widetilde{\mathcal{G}}_\infty$, see (\ref{L9}), and also the fact
that $\mu^{\sigma_n,s_m}_t \Rightarrow \mu^{s_m}_t$ implies the
convergence of the  integrals in (\ref{U77a}), see Proposition
\ref{Sqqpn}. As mentioned above, $LF$ is a bounded function (by
claim (i) of Proposition \ref{T3pn}), which means that both terms of
the left-hand side of (\ref{U77a}) are bounded by $\sup_{\gamma\in
\Gamma^2_*} |(LF)(\gamma)|$. Together with the convergence $C_n \to
C_\infty$ this yields $b_n^{(1)}\to 0$ as $n\to +\infty$.

Now we turn to $c_n$. By (\ref{U72}) and (\ref{U73}), and then by
(\ref{Sq3}), we have
\begin{gather*}
c_n(t) = C_n\left[ \mu_t^{\sigma_n,s_m} (LKG) - \mu_t^{\sigma_n,s_m}
(L^{\sigma_n} KG)\right] \\[.2cm] = C_n \langle \! \langle (L^{\Delta} -
L^{\Delta,\sigma_n}) k_{\mu_t^{\sigma_n,s_m}}, G \rangle\!\rangle =
C_n \langle \! \langle \widetilde{L}^{\Delta,\sigma_n}
k_{\mu_t^{\sigma_n,s_m}}, G \rangle\!\rangle, \nonumber
\end{gather*}
cf. (\ref{Sq15}). Here we have also taken into account that
$\mu_t^{\sigma_n,s_m}\in \mathcal{P}_{\rm exp}$, that was
established above, and  the operator
$\widetilde{L}^{\Delta,\sigma_n}$ is the same as in (\ref{Sq15}). To
make precise in which spaces $\mathcal{K}_\vartheta$ it acts, we
will take into account that $G\in
\mathcal{G}_\infty=\cap_{\vartheta\in
\mathds{R}}\mathcal{G}_\vartheta$, see  Proposition \ref{V1pn}, and
that the type of each $\mu^{\sigma_n, s_m}_t$ does not exceed $\exp(
\vartheta (t_2 -s))=:e^\vartheta$. Then we write, cf. (\ref{Sq15})
and (\ref{Sq16}),
\begin{equation*}
\langle \! \langle \widetilde{L}^{\Delta,\sigma_n}
k_{\mu_t^{\sigma_n,s_m}}, G \rangle\!\rangle =:R_n (t) =
\sum_{j=1}^4 R_{n,j}(t),
\end{equation*}
where
\begin{eqnarray*}
R_{n,1} (t) & = & \int_{\Gamma_0} \int_{\Gamma_0} \bigg{(}
\sum_{y\in \eta_0} \int_X a^{\sigma_n}_0(x,y) e(\tau^0_y;\eta_1)
(\Upsilon^0_y k_{\mu_t^{\sigma_n,s_m}}) (\eta_0\setminus y\cup x,
\eta_1) dx \bigg{)} \\[.2cm] \nonumber &\times &  G(\eta_0,\eta_1) \lambda(d\eta_0)
\lambda(d\eta_1)\\[.2cm] \nonumber &= & \int_{\Gamma_0} \int_{\Gamma_0} \bigg{(}
\int_X \int_X a^{\sigma_n}_0(x,y) e(\tau^0_y;\eta_1) (\Upsilon^0_y
k_{\mu_t^{\sigma_n,s_m}}) (\eta_0\cup x,
\eta_1) G(\eta_0\cup y,\eta_1) dx dy \bigg{)} \\[.2cm] \nonumber &\times &   \lambda(d\eta_0)
\lambda(d\eta_1),
\end{eqnarray*}
and likewise
\begin{eqnarray*}
R_{n,2} (t) & = & - \int_{\Gamma_0} \int_{\Gamma_0} \bigg{(} \int_X
\int_X a^{\sigma_n}_0(x,y) e(\tau^0_y;\eta_1) (\Upsilon^0_y
k_{\mu_t^{\sigma_n,s_m}}) (\eta_0\cup x,
\eta_1) G(\eta_0\cup x,\eta_1) dx dy \bigg{)} \\[.2cm] \nonumber &\times &   \lambda(d\eta_0)
\lambda(d\eta_1), \qquad \\[.2cm] \nonumber
R_{n,3}(t) &= & \int_{\Gamma_0} \int_{\Gamma_0} \bigg{(} \int_X
\int_X a^{\sigma_n}_1(x,y) e(\tau^1_y;\eta_0) (\Upsilon^1_y
k_{\mu_t^{\sigma_n,s_m}}) (\eta_0,
\eta_1\cup x) G(\eta_0,\eta_1\cup y) dx dy \bigg{)} \\[.2cm] \nonumber &\times &   \lambda(d\eta_0)
\lambda(d\eta_1),
\\[.2cm] \nonumber
R_{n,4} (t) & = & - \int_{\Gamma_0} \int_{\Gamma_0} \bigg{(} \int_X
\int_X a^{\sigma_n}_1(x,y) e(\tau^1_y;\eta_0) (\Upsilon^1_y
k_{\mu_t^{\sigma_n,s_m}}) (\eta_0,
\eta_1\cup x) G(\eta_0,\eta_1\cup x) dx dy \bigg{)} \\[.2cm] \nonumber &\times &   \lambda(d\eta_0)
\lambda(d\eta_1).
\end{eqnarray*}
Now we take into account that $k_{\mu_t^{\sigma_n,s_m}} \in
\mathcal{K}_\vartheta$ and $G\in \mathcal{G}_\infty$, see above,
employ (\ref{Sq18}), and then get, cf. (\ref{Sq19}),
\begin{equation}
  \label{U81}
\left| R_{n,j} (t)\right| \leq \int_X r^{\sigma_n}_j(y) g_i ( y) dy,
\qquad j=1, \dots , 4.
\end{equation}
with $r^{\sigma_n}_j(y)$ given in (\ref{Sq20}) and
\begin{gather}
  \label{U82}
g_1 (y) = g_2(y) = c(\vartheta) \int_{\Gamma_0}\int_{\Gamma_0}
\left|G(\eta_0\cup y, \eta_1) \right| \exp\left(\vartheta|\eta_0| +
\vartheta|\eta_1| \right) \lambda (d\eta_0) \lambda (d \eta_1),
\\[.2cm] \nonumber
g_3 (y) = g_4(y) = c(\vartheta) \int_{\Gamma_0}\int_{\Gamma_0}
\left|G(\eta_0, \eta_1\cup y) \right| \exp\left(\vartheta|\eta_0| +
\vartheta|\eta_1| \right) \lambda (d\eta_0) \lambda (d \eta_1),
\end{gather}
where $c(\vartheta)$ is the same as in (\ref{Sq22}). Note that the
bound in (\ref{U81}) is uniform in $t\in [s_m,t_2]$, for which
$k_{\mu_t^{\sigma_n,s_m}} \in \mathcal{K}_\vartheta$. Now, similarly
as in (\ref{Sq24}), we get
\begin{eqnarray*}
& & \int_X g_1(y) d y \leq c(\vartheta)e^{-\vartheta}
\int_{\Gamma_0}\int_{\Gamma_0} |\eta_0|
 \left|G(\eta_0, \eta_1) \right|
\exp\left(\vartheta|\eta_0| + \vartheta|\eta_1| \right) \lambda
(d\eta_0) \lambda (d \eta_1) \\[.2cm] \nonumber & & \quad  \leq
c(\vartheta)e^{-\vartheta  -1 - \log \varepsilon}
\int_{\Gamma_0}\int_{\Gamma_0}
 \left|G(\eta_0, \eta_1) \right|
\exp\bigg{(}(\vartheta +\varepsilon)\left(|\eta_0| + |\eta_1|\right)
\bigg{)} \lambda (d\eta_0) \lambda (d \eta_1)  \\[.2cm] \nonumber & &
\quad = c(\vartheta)e^{-\vartheta  -1 - \log
\varepsilon}|G|_{\vartheta + \varepsilon} < \infty
\end{eqnarray*}
that holds for all $\varepsilon>0$ as $G\in \mathcal{G}_\infty$.
Similar estimates can be obtained for the remaining $|R_{n,j}(t)|$.
By the dominated convergence theorem we then get that $R_n(t) \to 0$
as $n\to +\infty$, uniformly in $t\in [s_m,t_2]$, which together
with the aforementioned convergence $C_n \to C_\infty$, yields
$c_n\to 0$ as $n\to +\infty$. Therefore, for each limiting point
$P_{s,\mu}$, it follows that $P_{s,\mu}({\sf H})=0$, that yields the
proof of item (c), and thus the existence in question.

Now we turn to \underline{uniqueness}. To this end we employ the
following fact.
\begin{proposition}
  \label{UKpn}
Assume that two solutions $\{P^{(j)}_{s,\mu}: s\geq 0, \ \mu \in
\mathcal{P}_{\rm exp}, \ j=1,2\}$, see Definition \ref{A1df},
satisfy $P^{(1)}_{s,\mu}\circ \varpi^{-1}_t= P^{(2)}_{s,\mu}\circ
\varpi^{-1}_t$, holding for all $t\geq s$, $s\geq 0$, and $\mu \in
\mathcal{P}_{\rm exp}$. Then they coincide, i.e., $P^{(1)}_{s,\mu} =
P^{(2)}_{s,\mu}$ for all $s\geq 0$ and $\mu \in \mathcal{P}_{\rm
exp}$.
\end{proposition}
The proof of this statement -- based on Lemma \ref{W2lm} -- is
completely analogous to that of \cite[Lemma 5.4]{KR}, and this can
be omitted here. Then the uniqueness in question is straightforward.
This completes the proof of item (i) of the theorem.

\subsection{Proving item (ii)}  We begin by recalling that
$X=\mathds{R}^d$. Let $\{r_j\}_{j\in \mathds{N}}\subset
\mathds{R}_{+}$ be a strictly increasing sequence such that
$\lim_{j\to +\infty} r_j = +\infty$. Set $\varDelta_j = \{x\in X:
|x|< r_j\}$ and $\gamma_{i,j} = \gamma_i\cap \varDelta_j$, $i=0,1$,
$j\in \mathds{N}$, $\gamma=(\gamma_0,\gamma_1)\in \Gamma^2_*$. We
also will use the notation $\gamma_j$ for $(\gamma_{0,j},
\gamma_{1,j})$. Then we define
\begin{equation*}
\Gamma^2_{*,j} =\{ \gamma\in \Gamma^2_*: \gamma_{j} \in
\breve{\Gamma}^2\}.
\end{equation*}
That is, $\gamma\in \Gamma^2_*$ belongs to $\Gamma^2_{*,j}$ if and
only if $n_0(x) + n_1(x) =1$, holding for all $x\in p(\gamma)\cap
\varDelta_j$, see (\ref{no16}). By (\ref{no21}) we then have
\begin{equation}
  \label{U85}
\breve{\Gamma}^2_* = \bigcap_{j\in \mathds{N}} \Gamma^2_{*,j}.
\end{equation}
Similarly as in \cite[Lemma 2.7]{KR}, one proves that each
$\Gamma^2_{*,j}$ is an open subset of $\Gamma^2_*$, see also Lemma
\ref{N2pn}. Define, cf. (\ref{no17}),
\begin{eqnarray*}
h_N (x,y) & = & \psi(x) \psi(y) \min\{N; |x-y|^{-d\epsilon}\},
\qquad N\in \mathds{N}, \\[.2cm] \nonumber
H_N (\gamma) & = & \sum_{x\in \gamma_0}\sum_{y\in \gamma_0\setminus
x}h_N (x,y) + \sum_{x\in \gamma_1}\sum_{y\in \gamma_1\setminus x}h_N
(x,y + \sum_{x\in \gamma_0}\sum_{y\in \gamma_1}h_N (x,y),
\end{eqnarray*}
where $\epsilon \in (0,1)$ and $\psi(x)$ is as in (\ref{P}). Now for
$\mu\in \mathcal{P}_{\rm exp}$ of type $\varkappa$, similarly as in
(\ref{no18}) we get
\begin{gather}
  \label{U87}
\mu(H_N) \leq  3 \varkappa^2 \mathcal{I}_N,
\end{gather}
where, cf. (\ref{no19}),
\begin{gather*}
\mathcal{I}_N =  \int_{X^2} h_N (x,y) dx dy \leq \int_X
\psi(x) \left(\int_X \psi(y) |x-y|^{-d \epsilon} dy \right) dx  \\[.2cm]
\nonumber \leq \int_X \psi(x) \left( \int_{B_r}
\frac{dz}{|z|^{d\epsilon}} + \frac{\langle \psi
\rangle}{r^{d\epsilon}}\right) dx = \frac{c_d
r^{d(1-\epsilon)}}{d(1-\epsilon)} \langle\psi\rangle + \frac{\langle
\psi \rangle^2}{r^{d\epsilon}}.
\end{gather*}
Then, similarly as in (\ref{S4a}), we have that
\begin{equation*}
H(\gamma) := \lim_{N\to +\infty} H_N(\gamma) <\infty
\end{equation*}
for $\mu$-almost all $\gamma\in \breve{\Gamma}^2_*$. At the same
time,
\[
H_j(\gamma) := H(\gamma_j) \leq  r^{-d\epsilon}_{\gamma_j}\left(
|\gamma_0\cap \varDelta_j| + |\gamma_1\cap \varDelta_j|\right)^2,
\]
holding for all $\gamma\in {\Gamma}^2_{*,j}$. Here $r_{\gamma_j}$ is
the minimal distance between two distinct $x,y \in
(\gamma_0\cap\varDelta_j) \cup   (\gamma_0\cap\varDelta_j)$, which
is positive since the number of pairs of such points is finite and
$\gamma_j$ is simple.

Let $\{P_{s,\mu}: s\geq 0, \ \mu \in \mathcal{P}_{\rm exp}\}$ be the
solution which exists and is unique according to item (i). Fix some
$s\geq 0$ and $\mu\in \mathcal{P}_{\rm exp}$ of type $\varkappa$,
and let $\mathcal{X}$ be the process corresponding to $P_{s,\mu}$.
For $N\in \mathds{N}$, we define the stopping time, cf. \cite[page
180]{EK},
\begin{equation*}
T_{N,j}  =\inf\{t\geq s: H_j (\mathcal{X}(t))>N\}.
\end{equation*}
Then for a fixed $j\in \mathds{N}$ and $T_{N,j} \wedge t:=
\min\{t;T_{N,j}\}$, we set $\mathcal{Z}(t) = \lim_{N\to +\infty}
\mathcal{X} (T_{N,j} \wedge t)$ and $T_j = \lim_{N\to
+\infty}T_{N,j}$. Both limits exist as $T_{N,j}\leq T_{N+1,j}$. Let
$\tilde{\mu}_t$ be the law of $\mathcal{Z}(t)$. For $\varPhi^m_\tau
(\theta|\cdot) \in \mathcal{D}(L)$ defined in (\ref{W32a}),
\begin{equation}
  \label{U91}
\varPhi^m_\tau (\theta|\mathcal{X}(t)) - \int_s^t L\varPhi^m_\tau
(\theta|\mathcal{X}(u))du
\end{equation}
is a right-continuous martingale. Then, similarly as in \cite[page
180]{EK}, by the optional sampling theorem, we can write
\begin{equation*}
\mathbb{E}\left[\varPhi^m_\tau (\theta|\mathcal{X}(T_{N,j}\wedge t))
\right] = \mathbb{E}\left[\varPhi^m_\tau (\theta|\mathcal{X}(s))
\right] + \mathbb{E}\left[\int_0^{T_{N,j}\wedge t}L\varPhi^m_\tau
(\theta|\mathcal{X}(u))du \right],
\end{equation*}
where we pass to the limit $N\to +\infty$ and get, see also
(\ref{W34}),
\begin{eqnarray}
  \label{U93}
\tilde{\mu}_t(\varPhi^m_\tau (\theta|\cdot))& = &\mu(\varPhi^m_\tau
(\theta|\cdot)) + \lim_{N \to +\infty}
\mathbb{E}\left[\int_0^{T_{N,j}\wedge t}L\varPhi^m_\tau
(\theta|\mathcal{X}(u))du \right] \\[.2cm] \nonumber & \leq & \mu(\varPhi^m_\tau
(\theta|\cdot)) +  \lim_{N \to +\infty}
\mathbb{E}\left[\int_0^{T_{N,j}\wedge t}\left|L\varPhi^m_\tau
(\theta|\mathcal{X}(u))\right|du \right] \\[.2cm] \nonumber & \leq & \mu(\varPhi^m_\tau
(\theta|\cdot)) + \mathbb{E}\left[\int_0^{ t}\left|L\varPhi^m_\tau
(\theta|\mathcal{X}(u))\right|du \right] \\[.2cm] \nonumber & \leq &
\mu(\varPhi^m_\tau (\theta|\cdot)) + \mathbb{E}\left[\int_0^{
t}\varPhi^m_{\tau,1}
(\theta|\mathcal{X}(u))du \right] \\[.2cm] \nonumber & = & \mu(\varPhi^m_\tau (\theta|\cdot)) +
\int_s^t \mu_u (\varPhi^m_{\tau,1} (\theta|\cdot) du.
\end{eqnarray}
Here $\varPhi^m_{\tau,1} (\theta|\cdot)\in \mathcal{D}(L)$ is as in
(\ref{W34}) and (\ref{W40}), and $\mu_u= P_{s,\mu}\circ
\varpi^{-1}_u$ is the law of $\mathcal{X}(u)$. Similarly as in
(\ref{B2}), by (\ref{U93}) we then obtain
\begin{equation*}
\tilde{\mu}_t(\varPhi^m_\tau (\theta|\cdot))\leq \sum_{q=0}^\infty
\frac{(t-s)^q}{q!} \mu \left( \varPhi^m_{\tau,q}
(\theta|\cdot)\right), \qquad t-s < \log(1+\varepsilon)/c_a.
\end{equation*}
Now we proceed here as in obtaining (\ref{B4}), which finally
yields, see (\ref{B9}), (\ref{B10},
\begin{equation}
  \label{U96}
\tilde{\mu}_t(\varPhi^m_ (\theta|\cdot)) :=
\lim_{\max\{\tau_0,\tau_1\} \to 0} \tilde{\mu}_t(\varPhi^m_\tau
(\theta|\cdot))  \leq \varkappa_t^{|m|} \|\theta_0\|^{m_0}_{L^1(X)}
\|\theta_1\|^{m_1}_{L^1(X)},
\end{equation}
where $t-s < \log(1+\varepsilon)/c_a$ and $\varkappa_t = \varkappa
e^{(\alpha+1)(t-s)}$. By Definition \ref{no1dfg} (\ref{U96}) yields
$\tilde{\mu}_t\in \mathcal{P}_{\rm exp}$ and hence $\mathcal{Z}(t)
\in \breve{\Gamma}^2_*$ (almost surely) for this $t$, see
(\ref{C4}). Thus, $T_{j} >t$. Now we take $\delta\in (0,1)$ and then
$s_1 = s+ \delta \log(1+\varepsilon)/c_a$, take into account that
$\mu_{s_1}= P_{s,\mu}\circ \varpi_{s_1}^{-1}\in \mathcal{P}_{\rm
exp}$, and repeat the above procedure with $s$ replaced by $s_1$.
Since the type of $\mu_t$ is $\varkappa e^{\alpha (t-s)}$ -- and
hence is finite for all $t$ -- the construction can be repeated ad
infinitum to cover the whole $[s,+\infty)$. This implies that the
paths of $\mathcal{X}$ remain in $\mathfrak{D}_{[0,+\infty)}
(\Gamma^2_{*,j})$ for all $j\in \mathds{N}$, which by (\ref{U85})
yields the proof of item (ii).

\section*{Acknowledgments} This work was supported by the Deutsche
Forschungsgemeinschaft through SFB 1238 ``Taming uncertainty and
profiting from randomness and low regularity in analysis, statistics
and their applications", that is cordially acknowledged by the
authors.

\end{document}